\documentclass[bj,preprint]{imsart}

\usepackage{xr-hyper}
\RequirePackage[utf8]{inputenc}
\RequirePackage[OT1]{fontenc}
\RequirePackage[usenames,dvipsnames]{xcolor}

\usepackage{titletoc}

\RequirePackage{amsthm,amsmath,amsfonts,amssymb}
\RequirePackage[numbers]{natbib}
\RequirePackage[hypertexnames=false,colorlinks,citecolor=blue,urlcolor=blue,linkcolor=blue,filecolor=blue,breaklinks=true]{hyperref}
\RequirePackage{graphicx}

\startlocaldefs
\newtheorem{theorem}{Theorem}
\newtheorem{proposition}{Proposition}
\newtheorem{lemma}{Lemma}
\newtheorem{corollary}{Corollary}
\newtheorem{assumption}{Assumption}

\newtheorem{remark}{Remark}

\usepackage{mathrsfs}
\usepackage{dsfont}
\usepackage{bm}

\usepackage[sectionbib]{bibunits}
\usepackage{nameref}
\usepackage[nameinlink,noabbrev,capitalise]{cleveref}
\crefname{assumption}{Assumption}{Assumptions}
\creflabelformat{enumi}{(#2#1#3)}

\newcommand{\params}{\phi}       
\newcommand{\TrueSize}{t}            
\newcommand{\size}{s}       
\newcommand{\spars}{\sigma} 

\newcommand{\PoiPr}{M}           
\newcommand{\PPlab}{\theta}      
\newcommand{\PPlat}{\vartheta}   

\newcommand{\nedg}{D_{\TrueSize}^{\star}}  
\renewcommand{\deg}[1]{D_{\TrueSize,#1}}        
\newcommand{\nver}{N_{\TrueSize}}          
\newcommand{\ndeg}[1]{N_{\TrueSize,#1}}      
\newcommand{\nedgsq}{D_{\TrueSize}^{\star 2}}

\newcommand{\tiltstab}{G_{\params}}                 
\newcommand{\LTS}{\psi}                             
\newcommand{\loglik}{L_{\TrueSize}}                 
\newcommand{\qloglik}{\mathcal{Q}_{\TrueSize}}      
\newcommand{\sparsfunc}{\mathcal{C}_{\TrueSize}}

\makeatletter
\def\steepfunc{\@ifnextchar[{\@SteepWith}{\@SteepWithout}}
\def\@SteepWith[#1]{F_{#1}}
\def\@SteepWithout{F_{\params}}
\makeatother

\newcommand{\KLparams}{\params_{*}}
\newcommand{\KLsize}{\size_{*,t}}
\newcommand{\KLspars}{\spars_{*}}
\newcommand{\KLtau}{\tau_{*}}

\newcommand{\MLEparams}{\hat\params_{\TrueSize}}

\newcommand{\qMLEparams}{\hat\params_{\TrueSize}}
\newcommand{\qMLEsize}{\hat\size_{\TrueSize}}
\newcommand{\qMLEspars}{\hat\spars_{\TrueSize}}
\newcommand{\qMLEtau}{\hat\tau_{\TrueSize}}
\newcommand{\qMLEeps}{\hat{\varepsilon}_t}
\newcommand{\qMLEu}{\hat{u}_t}

\newcommand{\EE}{\mathbb{E}}            
\newcommand{\PP}{\mathbb{P}}            

\newcommand{\var}{\mathrm{var}}
\newcommand{\poiDist}{\mathrm{Poisson}} 
\newcommand{\gammaDist}{\mathrm{Gamma}} 

\renewcommand{\digamma}{\varphi}

\global\long\def\Reals{\mathbb{R}}
\global\long\def\Complex{\mathbb{C}}
\global\long\def\Nats{\mathbb{N}}
\global\long\def\Ints{\mathbb{Z}}
\global\long\def\NNInts{\Ints_{+}}
\global\long\def\NNReals{\Reals_{+}}

\global\long\def\intd{\mathrm{d}} 
\newcommand{\Ind}{\boldsymbol{1}}
\newcommand{\1}{\boldsymbol{1}}

\usepackage{mathtools}

\providecommand\given{} 
\newcommand\SetSymbol[1][]{
  \nonscript\,#1:\nonscript\,\mathopen{}\allowbreak}
\DeclarePairedDelimiterX\Set[1]{\lbrace}{\rbrace}%
{ \renewcommand\given{\SetSymbol[]} #1 }

\DeclareMathOperator*{\argmax}{arg\,max}

\DeclareMathOperator{\Poissonpdf}{Poisson}
\newcommand{\multigraphon}{W_m}
\newcommand{\oW}{\overline W}
\newcommand{\tW}{\widetilde W}
\newcommand{\tmu}{\widetilde \mu}

\renewcommand{\1}[1]{\mathds{1}_{#1}}

\newcommand{\Tr}{\mathrm{Tr}}

\endlocaldefs

\begin{document}

\defaultbibliography{bnpnetwork}


\begin{bibunit}[imsart-number]
  \startcontents[main]%


\begin{frontmatter}
\title{Asymptotic Analysis of Statistical Estimators related to MultiGraphex Processes under Misspecification}
\runtitle{Asymptotic Analysis of Statistical Estimators related to MultiGraphex Processes}

\begin{aug}
\author[A]{\fnms{Zacharie} \snm{Naulet}\ead[label=e1]{zacharie.naulet@universite-paris-saclay.fr}},
\author[B]{\fnms{Judith} \snm{Rousseau}\ead[label=e2,mark]{judith.rousseau@stats.ox.ac.uk}}
\and
\author[B]{\fnms{François} \snm{Caron}\ead[label=e3,mark]{caron@stats.ox.ac.uk}}
\address[A]{Université Paris-Saclay, Laboratoire de mathématiques d’Orsay, 91405, Orsay, France.
\printead{e1}}

\address[B]{University of Oxford, Department of Statistics, Oxford, UK.
\printead{e2,e3}}
\end{aug}

\begin{abstract}
  This article studies the asymptotic properties of Bayesian or frequentist
estimators of a vector of parameters related to structural properties of sequences of graphs. The estimators studied originate from a particular class of graphex model introduced by Caron and Fox \cite{Caron2017}. The analysis is however performed here under very weak assumptions on the underlying data generating process, which may be different from the model of \cite{Caron2017} or from a graphex model. In particular, we consider generic sparse graph models, with unbounded degree, whose degree distribution satisfies some assumptions. We show that one can relate the limit of the estimator of one of the parameters to the sparsity
constant of the true graph generating process. When taking a Bayesian approach, we also show that the posterior distribution is asymptotically normal. We discuss situations where classical random graphs models such as configuration models,
sparse graphon models, edge exchangeable models or graphon processes satisfy our assumptions. 
\end{abstract}


\end{frontmatter}

\section{Introduction}
\label{sec:intro}

\par The increasing availability of large graph-valued data has led to the
development of a variety of statistical network models, capturing different
features of real-world graphs, together with computational methods to estimate their parameters and, in some cases, the theoretical analysis of the statistical properties of these procedures. In this
paper we are interested in the latter, i.e. the analysis of the statistical
properties of inference procedures in networks, with a particular focus on the class of
sparse multigraphs with unbounded average degree: that is graphs such that the
number of edges grows sub-quadratically, but super-linearly, with the number of
nodes.

\par Probabilistic sparse network models with unbounded average degree include
sparse
graphons~\cite{Bollobas2009,BickelChen2009,BorgsChayesCohnZhao2014a,BorgsChayesCohnZhao2014b},
models based on exchangeable random measures (also known as graphex or graphon
processes)~\cite{Caron2017,Veitch2015,Borgs2018}, some sparse configuration
models~\cite{Borgs2019}, and edge-exchangeable
graphs~\cite{Crane2018,Cai2016,Williamson2016}. While these models can all
capture sparsity, they exhibit very different structural properties. For
instance, in sparse graphon models, the clustering coefficient, a measure of the
transitivity of the graph, vanishes as the size of the graph increases.
Additionally, the proportion of nodes with a given degree $j\geq 1$ goes to zero
for any $j$. Other models, such as graphex models, can exhibit a very different
behaviour, as shown by~\cite{caron:rousseau:18}. For some models within this
class, the clustering coefficient converges to a strictly positive constant;
moreover, the proportion of nodes of a given degree $j$ tends to a constant
$p_j>0$, which behaves like a power-law for large $j$.

\par Most of the literature on statistical properties of estimators in sparse
graphs with unbounded degrees has focused on the specific class of sparse
graphon models, with a particular interest in sparse stochastic
blockmodels~\cite{BickelChen2009,BickelChenLevina2011,Wolfe2013,Borgs2015,Gao2016,Klopp2017,klopp2019optimal}; see \cite{Gao2021} for a recent survey on optimal procedures for (sparse) graphon estimation.
Little work has been done besides sparse graphon models. Notable exceptions are
the work of \cite{Veitch2019} which considers consistent estimators of the
(generalised) graphon function, and \cite{Naulet2017} who propose a novel
estimator for the power-law exponent in graphex-based models and analyse some of
its properties. Both of these work consider the well-specified case, where the
graphs are sampled from a graphex process.

\par \cite{Caron2017} introduced and studied a specific model of sparse random
multigraph based on exchangeable random measures. This model, which belongs to
the class of multigraphex and has connections to sparse configuration
models~\cite{Borgs2019}, is parameterized by three parameters
 $ (\sigma,\tau,s)\in (-\infty,1)\times(0,\infty)^2 \eqqcolon \mathcal{S}$.
Although simple, the model captures important properties of real-world graphs
via its interpretable parameters. The parameter $\sigma$ controls both the sparsity
rate and the exponent of the asymptotic degree distribution, larger values
corresponding to sparser graphs; the parameter $\tau$ tunes the exponential tilting
of the degree distribution for finite graphs; the parameter $s$ corresponds to
an effective sample size of the graph. Additionally, \cite{Caron2017} develops
efficient algorithms to perform a Bayesian inference for this model.

\par The aim of this paper is to study the properties of Bayesian or frequentist
methods of estimation of $\phi=(\sigma,\tau,s)$ based on the likelihood of the model of
\cite{Caron2017}. More precisely let $\MLEparams$ and $\Pi_t(\cdot)$ denote
respectively the maximum likelihood estimator and the posterior distribution for
some prior distribution $p(\cdot)$ over the parameter $\phi$ based on the observation
of a graph of effective sample size $t$ under this model. We study the existence
and the convergence of the estimator $\MLEparams$ in~\cref{thm:concent:1}, as
well as the asymptotic behaviour of the distribution $\Pi_t(\cdot)$
in \cref{posterior:cons:1,posterior:cons:2} as the size of the graph goes to
infinity. In both cases, we do not assume that the true generating distribution
of the graph belongs to the model of \cite{Caron2017} nor do we even assume that
it belongs to a multigraphex model. Our assumptions on the true data generating
process are very weak.

\par In particular, we consider generic sparse graphs models satisfying, in
probability or almost surely
\begin{align*}
  \nedg \sim C \nver^{\frac{2}{1 + \alpha_0}}, \quad \nver \rightarrow \infty
\end{align*}
for some sparsity constant $\alpha_0\in(0,1)$ and some $C>0$, where $\nedg$ is twice
the number of (multi)edges and $N_t$ the number of nodes. We show that, under
some additional assumptions on the observed degree distribution, one can relate
the limit of the estimator of the sparsity parameter $\qMLEspars$ to the
sparsity constant $\alpha_0$ of the true graph generating process. When taking a
Bayesian approach, we also show that the posterior distribution $\Pi_t(\cdot)$
is asymptotically normal with mean $\MLEparams$ and a variance which we
characterize.

\par We then investigate situations where classical random graphs models such as
\textit{configuration models}, \textit{sparse graphon models}, or \textit{edge exchangeable models} may or may not generate sequence of graphs for which $\MLEparams$ and $\Pi_t(\cdot)$
do concentrate. Of special interest, we study the case where the true data
generating process is a multigraphex process, possibly different from
\cite{Caron2017}'s model. This section has interests in its own right since we
provide some theoretical results on the asymptotic behaviour of the degrees in
this class of multigraphex, complementing the results of
\cite{caron:rousseau:18}.

In \cref{sec:mainresults} we describe the model which is used to make inference and  we  present the results on the limiting
behaviour of the maximum likelihood estimator and of the posterior distribution, in the well and  mis-specified cases.
In \cref{sec:examples-processes}, we discuss the applicability of our results to
various classes of probabilistic sparse graphs or multigraphs and in
\cref{sec:multigraphex} we study in details the case where the true data
generating process is a general multigraphex process, showing that the results
of \cref{sec:mainresults} apply under mild assumptions. Additionally, we derive
the sparsity properties and power-law properties of multigraphex process. The
main proofs are given in \cref{sec:proofs-main-results}.

\textbf{Notations.}
Throughout the document, $(\mathcal G_t)_{t\geq 0}$ denotes a deterministic or
random family of multigraphs, without vertices of degree zero, indexed by a size
parameter $t\geq 0$. We note $\mathcal G_t=(\mathcal V_t, \mathcal E_t)$, where $\mathcal V_t\subset\mathbb N$ is the vertex set, and  $\mathcal E_t=\{(i,j,\widetilde n_{i,j})\text{, for }i\leq j \in \mathcal V_t \}$ denotes the edge set, composed of the number of multiedges $\widetilde n_{i,j}$ between nodes $i$ and $j$ in the vertex set. By convention, assume $\widetilde n_{j,i}=\widetilde n_{i,j}$. The size parameter may correspond to the number of nodes or edges or of some other quantity related to these numbers.

For any $i\in\mathcal V_t$, let $D_{t,i}=\sum_{j\in\mathcal V_t} \widetilde n_{i,j}\geq 1$ denote the degree of node $i$, so that
$\nedg=\sum_{i=1}^{N_t}D_{t,i}$ is twice the number of multiedges (up to the number
of self edges). Denote also
 $ N_{t,j}=\sum_{i=1}^{N_t} \Ind_{D_{t,i}=j}$,
the number of nodes with degree $j$. Observe that $\sum_{j\geq 1} N_{t,j}=N_t$ and
$\sum_{j\geq 1} j N_{t,j}=\nedg$.

Throughout the paper $a_n\sim b_n$ means that $a_n/b_n$ converges to $1$,
$a_n =o(b_n)$ means that $a_n/b_n $ converges to $0$, and $a_n = O(b_n)$ means that
$a_n/b_n$ is bounded. We also use the notation $a_n\lesssim b_n$ to denote
$a_n = O(b_n)$. We say that $a_n\asymp b_n$ if there exist constants $c_1, c_2>0$
such that $c_1a_n \leq b_n \leq c_2 a_n$. Also for any $B$ subset of $\mathcal X$ $B^c$ denotes the complementary of $B$ in $\mathcal X$;  in particular, if the set refers to a function $f$ whose domain is $\mathcal X \subset \mathbb R^d$, then $B^c$ denotes the complementary set of $B$ in  $\mathcal X$.


\section{Main results}
\label{sec:mainresults}

In this section, we derive our main results on the maximum likelihood estimator
and the posterior distribution. We consider a given deterministic family of
multigraphs $(\mathcal G_t)_{t\geq 0}$.



In this section we are studying the statistical properties of either the maximum
likelihood or the Bayesian approach associated to the particular multigraphex
process proposed by \cite{Caron2017}, without assuming that the true data
generating process belongs to the model of \cite{Caron2017}.

In order to make explicit the likelihood, we briefly summarize the model of
\cite{Caron2017}. This model is a special example of the class of multigraphex models, which are presented in more details in \cref{sec:multigraphex}. Let $M$ be
a unit rate Poisson process:
$M \coloneqq \Set{(\vartheta_i,\theta_i) \given i \in \Nats}$ on $\NNReals^2$. In the model, the observations consist of a subset of an  infinite symmetric array $(\tilde{n}_{i,j})_{(i,j)\in \Nats^2}$ of multi-edges whose conditional distribution given $M$ is given by
\begin{equation}
  \label{eq:55}
  \tilde n_{i,j}\mid M \sim%
  \begin{cases}
    \poiDist( 2 \bar{\rho}^{-1} ( \vartheta_i)\bar{\rho}^{-1} ( \vartheta_j)) &\mathrm{if}\ i < j,\\
    \poiDist( \bar{ \rho}^{-1} ( \vartheta_i)^2) &\mathrm{if}\ i =j,
  \end{cases}
\end{equation}
where
$\bar{\rho}^{-1}(y) \coloneqq \inf\Set{x \in \NNReals^{*} \given \bar{\rho}(x) \leq y}$
denotes the generalized inverse of the tail Lévy intensity of a generalized gamma process~\cite{Brix1999}:

\begin{align}
  \label{eq:GGLevyMeasure}
 \bar{\rho}(x) = \int_x^\infty \rho(w)dw, \quad  \rho(w)=\frac{1}{\Gamma(1-\sigma)}w^{-1-\sigma}e^{-\tau w}, \quad (\sigma,\tau) \in (-\infty,1)\times \NNReals^{*}.
\end{align}
The form \cref{eq:55} of the conditional distribution of the multi-edges is a special case of  multigraphon, defined in particular in \cite[Definition
4]{Borgs2019}. In \cref{sec:multigraphex} we study the behaviour of the multigraphex for more general forms of multigraphons.

The infinite array $(\tilde{n}_{i,j})_{(i,j)\in \Nats^2}$ can be seen as the
adjacency matrix of an infinite multigraph $\mathcal{G}$. A projective sequence
of (almost-surely) finite multigraphs $(\mathcal{G}_{\size})_{\size> 0}$ is obtained from
$\mathcal{G}$ by keeping only the nodes $i \in \Nats$ such that $\theta_i \leq \size$ and such
that they have at least one connection with a node $\theta_j \leq \size$ (with possibly $j=i$): $\mathcal{G}_{\size} = (\mathcal V_\size, \mathcal E_\size)$,  with $\mathcal V_s = \Set{ i \given\, \theta_i \leq \size,\ \sum_{j; \theta_j\leq \size} \tilde n_{i,j} >0}$.

The previous construction describes a parametric model for multigraphs whose
parameters are $(\sigma,\tau)$, for a given size $\size$. As shown by \cite{Caron2017,caron:rousseau:18}, within this model, the parameter $\sigma$, when in the range $(0,1)$, can be interpreted as the sparsity parameter $\alpha_0$ presented in the introduction; as we show in \cref{sec:multigraphex}, it also controls the asymptotic degree distribution. The expected number of edges under this model is asymptotically $\mathbb E[D^\star_s] \sim 2s^2 \tau^{2\sigma-2}$, and thus $\tau$ controls the overall number of edges.

 As the effective size $\size$ is typically unobserved,
it is customary to also consider it as a parameter of the model, so that we have
a parametric family with parameter $\phi \coloneqq (\sigma,\tau,\size)$. An important
aspect of this model is that $(N_{t,j})_{j_{\geq 1}}$ is a sufficient statistic for
$\phi$. Note that in the notation $N_{t,j}$, $t$ represents a size index, which can be thought of as the number nodes or the true (unknown) size if the observed multigraph is generated from the model of \cite{Caron2017}.

\subsection{Maximum likelihood estimator (MLE) and concentration of the likelihood function}
\label{sec:conc-likel-funct}
 Having observed a multigraph $\mathcal G_t$,  the log-likelihood $L_t(\phi)$, with $\phi \coloneqq (\sigma,\tau,\size)$,  has the following expression (see  \cite[\cref{supp-sec:proof-crefl-like}]{caron:naulet:rousseau:supplement})
  \begin{equation}
    \label{eq:likelihood}
    e^{\loglik(\params)}%
    \propto s^{\nver}%
    \prod_{i=1}^{\nver}\left[
      \frac{\Gamma(\deg{i} -\spars)}{\Gamma(1-\spars)}\right] %
    \int_{\NNReals^2}%
    e^{-(x+y)^{2}}\,
    \frac{%
      x^{\nedg - \spars\nver -1} e^{-\tau x}%
    }{%
      \Gamma(\nedg - \spars\nver)%
    }%
    \tiltstab(\intd y)\intd x,
  \end{equation}
  where $\tiltstab$ is the \textit{exponentially tilted stable distribution} on $\NNReals$,
  whose Laplace transform is
  \begin{equation}
    \label{eq:4}
    \int_{\NNReals} e^{-\xi x}\,\tiltstab(\intd x)=e^{-\size\psi(\spars,\tau;\xi+\tau)},%
     \quad
     \psi(\sigma,\tau;\xi)=
    \begin{cases}
      \frac{1}{\spars}\left( \xi^\spars-\tau^\spars\right ) &\mathrm{if}\ \spars \ne 0,\\
      \log(\xi/\tau) &\mathrm{if}\ \spars = 0,
    \end{cases}
  \end{equation}
for $\xi \in \Complex \backslash \Reals^{*}_-$.



The log-likelihood defined in \eqref{eq:likelihood} is complex, but we show in \cref{lem:3} that it is close to the function
\begin{equation}
  \label{eq:304}
  \qloglik(\params)%
  \coloneqq
  \nver \log(\size) + \sparsfunc(\spars) - \nedg \mathcal{A}(\params;\zeta(\params))%
  - \frac{\log(2)}{2},
\end{equation}
where  $\sparsfunc : (-\infty,1) \rightarrow \Reals$ is defined by
\begin{equation*}
  \sparsfunc(\spars) \coloneqq \log\left( \prod_{i=1}^{\nver}\left[
      \frac{\Gamma(\deg{i} -\spars)}{\Gamma(1-\spars)}\right]\right) = \sum_{j\geq 2} \ndeg{j}
  \sum_{k=1}^{j-1}\log(k-\spars),
\end{equation*}
and,
\begin{equation}
  \label{eq:2}
  \mathcal{A}(\params;z) \coloneqq%
  - \frac{(z - \tau)^2}{4\nedg}
  + \Big(1 - \frac{\spars \nver}{\nedg}\Big)\log(z)%
  + \frac{\size}{\nedg} \LTS(z;\sigma,\tau); \quad z\in \Complex \backslash \Reals_-^{*} ,
\end{equation}
and $\zeta(\params)$ is the (unique) positive real value solution of the equation in $z$: $\partial_z  \mathcal{A}(\params;z) = 0$. Note that $ \sparsfunc(\spars)$ appears in the definition of $\loglik(\params)$ in \eqref{eq:likelihood}.

In this paper we do not characterize the maximizer of $\loglik$ but instead the maximizer $\qMLEparams$ of $\qloglik$, which, from \cref{lem:3}, is close to any maximizer of $\loglik$. With a slight abuse, we call $\qMLEparams$ the MLE.

To ensure that $\qMLEparams$ has a limit in $t$, we also require some
assumptions on the observed multigraph. Let $\hat{\alpha}_t \in (0,1)$ be the solution
to the equation $ \hat{\alpha}_t \sparsfunc'(\hat{\alpha}_t) = - \nver$. It can be easily
seen that $\hat{\alpha}_t$ exists uniquely as soon as $\ndeg{1}\ne \nver$.  We then
consider the following assumption on $\mathcal G_t$.
\begin{assumption}
  \label{ass:concent:degdist}
  We assume that $\ndeg{1} \ne \nver$ and the existence of $\alpha_0 \in (0,1)$
  and $\tau_{*} \in \NNReals^{*}$ such that the following two conditions are
  satisfied:
  \begin{align}
      &\lim_{t\to\infty} \sqrt{2\nedg}\Big(\frac{\alpha_0
      \nver}{\nedg}\Big)^{\frac{1}{1-\alpha_0}} = \tau_{*}, \quad \lim_{t \to \infty}\log (\nver ) |\hat{\alpha}_t -
    \alpha_0| = 0 \label{eq:assumpt1_cond1}
  \end{align}
  where $\hat{\alpha}_t$ is solution of the equation
  $- \hat{\alpha}_t \sparsfunc'(\hat{\alpha}_t) = \nver$; that is solution of
  \begin{equation}
    \label{eq:defalphat}
    \sum_{j\geq 2}\ndeg{j}\sum_{k=1}^{j-1} \frac{\hat{\alpha}_t}{k -
      \hat{\alpha}_t} = \nver.
  \end{equation}
\end{assumption}

The first part of \eqref{eq:assumpt1_cond1} requires the graph sequence to be sparse, with sparsity constant $\alpha_0$, as it implies that as $t\to\infty$
\begin{align}
  \label{eq:constantsparsity}
  \nedg \sim C(\alpha_0,\tau_{*}) \times\nver^{2/(1+\alpha_0)},~~~\text{ where }~~~
  C(\alpha_0,\tau_{*})=\left (\Big(\frac{\sqrt{2}}{\tau_{*}}\Big)^{1-\alpha_0}\alpha_0\right )^{2/(1+\alpha_0)}.
\end{align}
To interpret the second part of \eqref{eq:assumpt1_cond1}, first note that we will see in the proof of \cref{thm:concent:1} in Section \ref{sec:concent:proofs-conc-likel-funct} that  $\qMLEspars$ is very close to $ \hat{\alpha}_t$  since $\qMLEspars$ satisfies
$\sparsfunc'(\qMLEspars) = -\nver/\qMLEspars( 1 +o(1)) $. Note also that $\hat{\alpha}_t$ is a stationary point of functions of $\alpha \mapsto N_t\log(\alpha)+C_t(\alpha)$. It follows that $\hat{\alpha}_t$ takes the alternative form
\begin{align}
  \label{eq:alphahatMLE}
\hat{\alpha}_t&=\argmax_{\alpha\in(0,1)} ~~\log\left( \prod_{i=1}^{N_t} \frac{\alpha\Gamma(D_{t,i}-\alpha)}{D_{t,i}!\Gamma(1-\alpha)}\right ),
\end{align}
and may therefore be interpreted as the maximum likelihood estimator under a model where the degrees $D_{t,i}$ are drawn i.i.d. from the Karlin-Rouault distribution~\cite{Karlin1967,Rouault1978}, whose probability mass function is
\begin{equation}
  \label{eq:KarlinRouaultdist}
  p_j=\frac{\alpha\Gamma(j-\alpha)}{j!\Gamma(1-\alpha)},~~j\geq 1.
\end{equation}
This distribution originally arose as the limit of the proportion of clusters of a given size in infinite-urn/random partition models~\cite{Gnedin2007}. Under some assumptions, this distribution also arises as the asymptotic degree distribution for graphex processes~\citep{caron:rousseau:18}, multigraphex processes (see \cref{sec:multigraphex} and in particular \cref{th:sparsity1}), and some edge-exchangeable multigraphs~\cite{Crane2018,Cai2016}. In light of this, the second condition in ~\eqref{eq:assumpt1_cond1} may be interpreted as an implicit assumption on the asymptotic degree distribution of the graph, and how this asymptotic degree distribution relates to the sparsity parameter $\alpha_0$. Note that this condition is weaker than requiring that the asymptotic degree distribution of the sequence of graphs $(\mathcal G_t)_{t\geq 0}$ is the Karlin-Rouault law; we provide some examples in \cref{sec:examples-processes} where it departs from it.

The estimator $\hat{\alpha}_t$ may also be related to the popular two-parameter Poisson-Dirichlet partition model~\cite{Pitman1995}. If one interprets $N_t$ as the number of clusters, and $D_{t,i}$ as the size of cluster $i$, then the right-handside of \cref{eq:alphahatMLE} is almost identical (up to a constant independent of $\alpha$) to the Poisson-Dirichlet partition model with parameters $(\alpha,0)$ (also known as $\alpha$-stable Poisson-Kingman
  model, see~\cite{Pitman2003,favaro2021near}), which is recovered if one replaces the term $\alpha^{N_t}$ by $\alpha^{N_t-1}$. See also \cref{sec:examples-processes} for the discussion on the Hollywood model, a random graph based on the two-parameter Poisson-Dirichlet partition model.


\begin{theorem}
  \label{thm:concent:1}
  Let \cref{ass:concent:degdist} be valid. 
  Then,
  \begin{enumerate}
    \item\label{thm:concent:1:item:1} There exists a unique maximizer
    $\qMLEparams = (\qMLEspars,\qMLEtau,\qMLEsize)$ of $\qloglik$, which we call
    MLE.

    \item\label{thm:concent:1:item:2}  For every $K > 0$, there exists $K' > 0$ such that
    for all $\params$ satisfying $|\spars - \qMLEspars| > K' \sqrt{\log(\nedg)/\nver}$ or
    $|\tau - \qMLEtau| > K' \sqrt{\log(\nedg)/(\nedg)^{1/2}}$ or
    $|\frac{\size}{\qMLEsize} - 1| > K' \sqrt{\log(\nedg)/\nver}$, then
    \begin{equation*}
      \qloglik(\params) - \sup\qloglik \leq -K \log(\nedg).
    \end{equation*}

    \item\label{thm:concent:1:item:3} With
    $\KLsize \coloneqq \tau_{*}^{1-\alpha_0}\sqrt{2\nedg}/2$, the MLE
    satisfies %
    \begin{align*}
      \lim_{t\to \infty}\log (\nedg) |\qMLEspars-\alpha_0| = 0,\qquad%
      \lim_{t\to\infty} \qMLEtau = \tau_{*},\qquad%
      \lim_{t \to \infty} \frac{\qMLEsize}{\KLsize}%
      = 1.
    \end{align*}

  \end{enumerate}
\end{theorem}

The proof of \cref{thm:concent:1} is given in
\cref{sec:concent:proofs-conc-likel-funct}. A remarkable feature of
\cref{thm:concent:1} is that it does not require the model to be well specified
and it holds under a weak assumption on the true data generating process, namely \cref{ass:concent:degdist}. In Section \ref{sec:examples-processes} we discuss the generality of this assumption in more details.  Note also that \cref{thm:concent:1} is a consistency result not a convergence rate. We study in the case were the true generating process of $\mathcal G_t$ belongs to the model of \cite{CaronFox2014} a much faster rate is obtained, as seen in Section \ref{sec:wellspecified}.



\subsection{Posterior concentration  and asymptotic normality of the posterior distribution}
\label{sec:conc-post-asympt}

In this Section we study the posterior distribution associated to Model
\eqref{eq:likelihood} and a prior distribution on $\phi= ( \sigma, \tau, s)$. We derive two
results on posterior consistency: one under \cref{ass:concent:degdist},
\textit{i.e}. for sparse graphs and the other for dense graphs. We also provide
a Bernstein von Mises type of result in the sparse regime. The Bernstein von
Mises result applies to the rescaled parameters $\params_u = (\spars, \tau, u)$,
with $u \coloneqq \size/\KLsize$ and to the rescaled estimate
$\hat\params_{\TrueSize,u} = (\qMLEspars, \qMLEtau, \hat u)$ with
$\hat u \coloneqq \qMLEsize/\KLsize$. Define the $3\times 3$ matrix $\Sigma_t(\params_u)$ (as a function of $\params_u$) as minus the inverse of   the Hessian matrix of the mapping
$\params_u \mapsto \qloglik(\spars,\tau,u \KLsize)$
and set $\hat{\Sigma}_{\TrueSize} = \Sigma_t(\hat\params_{\TrueSize,u})$. In many ways, the
matrix $\hat{\Sigma}_{\TrueSize}^{-1}$ can be seen as the empirical Fisher Information
matrix. Analytic expressions as well as various estimates for
$\Sigma_{\TrueSize}(\params_u)^{-1}$ are available in \cite[\cref{supp-sec:local-analys-likel}]{caron:naulet:rousseau:supplement}.

\begin{theorem}[Bayes, Sparse regime]\label{posterior:cons:1}
  Assume model \eqref{eq:likelihood} together with a proper prior on $\phi$, which
  has positive and continuous density on $\mathcal S$ and such that for all
  $x > 0$ there are $c,B> 0$ such that for all $u \leq \nedg$
  \begin{equation}
    \label{cond:Pis}
    \Pi\left( \left| \frac{ \size}{\sqrt{2\nedg}} - x\right| \leq u \right) \geq c \left(\frac{u}{\nedg}\right)^B .
  \end{equation}
  If \cref{ass:concent:degdist} is valid, then for all $\epsilon>0$,
  \begin{equation}
    \label{post:consist}
    \lim_{t\to\infty}\Pi\left( | \sigma - \alpha_0| +|\tau -\tau_{*}| + \Big|\frac{\size}{\KLsize} - 1 \Big| >\epsilon \ \Big|\ \mathcal G_t \right) = 0.
  \end{equation}
  Moreover, assume that the prior has a density $\pi$ such that for all $\epsilon'>0$ there exists $\epsilon>0$
  such that for all $\TrueSize$ large enough,
  \begin{equation}
    \label{prior:conti}
    \sup_{B_\epsilon(\KLparams)} \left| \frac{ \pi(\params) }{ \pi( \KLparams ) } - 1 \right| \leq \epsilon',%
    \quad B_\epsilon(\KLparams) = \Set*{ \params \given |\spars - \KLspars| + |\tau - \KLtau| + \left| \frac{ \size}{ \KLsize} - 1\right| \leq \epsilon }.
  \end{equation}
  Then, under \cref{ass:concent:degdist}, the posterior distribution of
  $\hat{\Sigma}_{\TrueSize}^{-1/2} (\params_u - \hat\params_{\TrueSize,u})$, written
  here $\tilde \Pi_t$, satisfies,%

  \begin{equation}
    \label{eq:BvM}
    \lim_{t\to \infty}\| \tilde \Pi_t - \mathcal{N}(0,1)^{\otimes 3}\|_{\mathrm{TV}} =  0,
  \end{equation}
  where $\mathcal{N}(0,1)^{\otimes 3}$ denote the standard Normal distribution on
  $\Reals^3$. Furthermore when $t$ is large enough, there are constants
  $C_1,C_2 > 0$ such that for all $x = (x_1,x_2,x_3) \in \Reals^3$%
  \begin{align*}
    x^t\hat \Sigma_\TrueSize^{-1} x%
    &\geq
    C_1\Big( \KLsize^{1+\alpha_0}x_1^2 + \KLsize x_2^2 + \frac{\nver}{\log(\nedg)^2}x_3^2 \Big) \\
   & \leq
    C_2\Big( \KLsize^{1+\alpha_0}\log(\nedg)^2x_1^2 + \KLsize x_2^2 + \nver x_3^2\Big).
  \end{align*}
\end{theorem}

The proof of \cref{posterior:cons:1} is given in \cref{sec:posterior:cons:1}.
From \cref{post:consist} we see that the posterior concentrates around
$(\alpha_0, \tau_{*}, \KLsize)$ and the \cref{eq:BvM} refines the analysis by stating a
Bernstein von Mises type of result, namely that the posterior is asymptotically
normally around the MLE with variance matrix $\hat \Sigma_t$. It is interesting to see that the posterior on $\spars$ and $\size$ concentrates at the rate $\sqrt{\nver}$ while the posterior on  $\tau$  concentrates at the rate $(\nedg)^{1/4} = o(\sqrt{\nver})$, up to $\log \nver$ terms. This reveals that there is much more signal in the likelihood on $\spars$ and $\size$.

Note that condition \eqref{cond:Pis} on the prior is very mild and allows any smooth prior on $\size$ which decreases polynomially with $\size$. Condition \eqref{prior:conti} is very similar. It is in particular satisfied if $\spars, \tau, \size$ are a priori independent with positive and continuous density and if the prior density on $\size$ is regularly varying at infinity. For instance the conditions are satisfied if $1 - \spars $ and  $\tau $ follow  Gamma distributions, and $\size$ is the positive part of a student random variable.

\Cref{posterior:cons:1}  does not  cover the dense
case. In the following theorem we prove that in the dense case the posterior
distribution concentrates on non positive values of $\sigma$.
\begin{theorem}[Bayes, dense regime]
  \label{posterior:cons:2}
Assume model \eqref{eq:likelihood} together with a proper prior on $\phi$. Assume that there exists $c_1>0$ such that
\begin{equation}\label{ass:dense}
 \forall c>0, \quad \liminf_t\frac{ \nedg }{\nver^{2-c} } = +\infty, \quad \liminf_t \frac{ \sum_{j\geq 2} \ndeg{j} \log (j) }{ \nver \log \nedg } \geq c_1,
 \end{equation}
 and assume that a priori $\spars, \tau, \size$ are independent  and  that there
 exists  $ c_2, c_3, c_4>0$ such that for all $(\nedg)^{-c_3} \leq x \leq  1$ and all
 $\tau < 1$
 \begin{gather}
   \label{cond:Pis2}
   \Pi\left( \left| \frac{ \size}{\KLsize} - x \right| \leq \frac{ 1 }{ \nver (\log \nedg)^2} \right) \geq e^{-c_2
     \nver},\qquad%
   \pi_\tau(\tau) \geq e^{-c_4/\tau^{c_4}}.
\end{gather}
Then for all $\epsilon >0$,
$$
\Pi\left( \sigma  >\epsilon  \mid \mathcal G_{\TrueSize} \right) \rightarrow 0.
$$%
\end{theorem}
The proof of \cref{posterior:cons:2} is given in \cref{sec:posterior:cons:2}. To
simplify the presentation we have assumed that $\spars , \tau, \size$ are a priori
independent. This is however not necessary. A weaker sufficient condition is
that %
for all $\epsilon , c>0 $, $\Pi(B_{\TrueSize}) > e^{-c \nver}$ when $\TrueSize$ is large
enough, where $B_{\TrueSize}$ is a set whose precise definition is given in the proof of
the \cref{posterior:cons:2}.

A consequence of \cref{cond:Pis2} is that when the true graph is dense then the posterior distribution on $\sigma $ concentrates on $(-\infty, 0]$, which in the model of \cite{CaronFox2014} also corresponds to dense graphs.

As proved in \cref{sec:rates-conv-mle}, if the true data generating process
comes from \cite{Caron2017}, i.e. if the model is well specified
\cref{ass:concent:degdist} and \cref{posterior:cons:1} and
\cref{posterior:cons:2} are valid with
$\alpha_0 = \spars_0, \tau_{*} = \tau_0, \TrueSize = t$, i.e. the true values of the
parameters. In the following Section we prove that $\qMLEparams$ concentrates
around the true value of the parameters with a rate given by $\Sigma_t$.

\subsection{Well-specified case}\label{sec:wellspecified}

We now investigate the behaviour of the MLE when the multigraph $\mathcal{G}_t$
has been obtained as a sample of the GGP model of \cite{Caron2017}; see also
\cref{sec:multigraphex} below for details. In other words, we consider the
situation where the model \eqref{eq:likelihood} is well specified. In
particular, we show that if $\mathcal{G}_t$ is generated according to the GGP
model with parameters $(\sigma_0,\tau_0)$ and size $t$, then the MLE
concentrates on $(\qMLEspars,\qMLEtau) \approx (\sigma_0,\tau_0)$ and $\qMLEsize/t \approx 1$,
at least when $\sigma_0 > 0$. This specializes the \cref{thm:concent:1} to the
well-specified scenario. Obviously this implies similar results for the Bayesian
counterparts of \cref{sec:conc-post-asympt} with straightforward adaptations.

\begin{theorem}
  \label{thm:well-specified}
  Let $(\mathcal{G}_t)_{t>0}$ be a sequence of random multigraphs generated
  according to the multigraph GGP model of \cite{Caron2017} with parameters
  $(\sigma_0,\tau_0)$. If $\sigma_0 > 0$, then in addition to the results of Theorems
  \ref{thm:concent:1} and \ref{posterior:cons:1}, we have
  \begin{align*}
    | \qMLEspars - \sigma_0| =O_p\left(t^{-(1+\sigma_0)/2} \right),\qquad%
    |\qMLEtau - \tau_0 | =O_p\left(\frac{1}{ t } \right),\qquad%
    \left| \frac{\qMLEsize}{t} -  1 \right| =O_p\left(t^{-(1+\sigma_0)/2}\log(t) \right).
  \end{align*}

\end{theorem}

The proof of \cref{thm:well-specified} is given in \cref{sec:proof-crefthm:w-spec}.
We see in the proof of \cref{thm:well-specified}, that in fact
\begin{equation*}
  \left( \qMLEspars - \sigma_0, \qMLEtau - \tau_0,  \frac{\qMLEsize}{t}-1 \right) = \Sigma_t( \params_0 )^{-1} \nabla L_t(\params_0) +o_p(1),
\end{equation*}
where the variance of $\Sigma_t( \params )^{-1} \nabla L_t(\params_0)$ is equal to 1.
Asymptotic normality of the MLE could then be deduced from the asymptotic
normality of the asymptotic normality of $\nabla L_t(\params_0)$, which is a non
trivial result. In particular, $\nabla L_t(\params_0)$ is a nonlinear functional of a
Poisson process, for which no generic theorem for asymptotic normality exists in
the literature. A first step toward understanding the asymptotic normality of
$\nabla L_t(\params_0)$ has been investigated in \cite{caron:rousseau:18}, where the
authors establish the (marginal) asymptotic normality of $\nedg$ and $\nver$,
yet it is unclear if the techniques they use can be adapted to understand the
limiting distribution of $\nabla L_t(\params_0)$.

An interesting consequence of \cref{thm:well-specified}, together with \cref{post:consist} is that in the well specified case, Bayesian credible regions on $\params$ will have reasonnable frequentist coverage.
 For instance consider a credible interval for $\spars$ in the form
 $(\spars_1, \spars_2) $ with
 $$\Pi ( \spars <\spars_1 \mid \mathcal G_{\TrueSize} ) = \Pi ( \spars >\spars_2 \mid \mathcal G_{\TrueSize} ) = \gamma/2 ,$$
 then \cref{eq:BvM} in \cref{posterior:cons:1} implies that
$$ \spars_1 = \qMLEspars + (\hat\Sigma_{\TrueSize}^{-1})_{1,1} \Phi^{-1}(\gamma/2) (1 + o(1)), \quad \spars_2 = \qMLEspars + (\hat\Sigma_{\TrueSize}^{-1})_{1,1} \Phi^{-1}(1-\gamma/2)(1 +o(1))$$
and
$$ \mathbb P( \spars_0 \notin (\spars_1, \spars_2) ) \leq \frac{ 2  }{ \Phi^{-1}(1-\gamma/2)(1+o(1))} $$
is small when $\gamma $ is small.



\section{Examples of Random Graphs satisfying \texorpdfstring{\cref{ass:concent:degdist}}{Assumption \ref{ass:concent:degdist}}}
\label{sec:examples-processes}


Theorems \ref{thm:concent:1}, \ref{posterior:cons:1} and \ref{posterior:cons:2}
provide results which are non stochastic in nature, in the sense that the
results hold true as soon as \cref{ass:concent:degdist} or \cref{ass:dense} are
valid. The interest is that these results do not require particular
probabilistic models on the multi-graph. However, \cref{ass:concent:degdist} can
be shown to hold for a variety of probabilistic models on sparse graphs or
multigraphs. In particular we do not require that $\mathcal G_{\TrueSize}$ is a
multigraph and \cref{ass:concent:degdist} can still hold if we observe a simple
graph. In thiss section we investigate which types of  models  satisfy
\cref{ass:concent:degdist}. A special interest will be given in multigraphex
processes in \cref{sec:multigraphex}, which the model of \cite{Caron2017} is a
special case of. In this section, we restrict our attention to models outside
the family of multigraphex processes.

\paragraph*{Configuration models}

A first example is the class of configuration models, which by
definition are constructed conditionally on the degree sequence $(D_{t,i})_{i\geq 1}$, see
\cite{molloy1995critical,vanderHofstad2014,Borgs2019}. The simplest version of
configuration models is to consider that the degrees $D_{t,1},\ldots,D_{t,N_t}$
are drawn i.i.d. from some distribution\footnote{Except the last degree, which
  should be chosen so that the sum of the degrees is even; as this makes no
  difference asymptotically, we ignore this technical point in this
  discussion.}, where the number of nodes $N_t=t$ is deterministic. An obvious
example would be to take this distribution to be the Karlin-Rouault distribution
with pmf given by \eqref{eq:KarlinRouaultdist} for some $\alpha_0\in(0,1)$. This
would automatically satisfy the second part of condition
\eqref{eq:assumpt1_cond1} of Assumption \ref{ass:concent:degdist} as
$\widehat\alpha_t$ is the MLE under this model, as discussed in
\cref{sec:mainresults}. However, using limit theorems for iid random variables
with polynomial tails, we have
$\nedg=\sum_{i=1}^{N_t} D_{t,i}\asymp N_t^{1/\alpha_0}$ almost-surely, which does not match the first part of condition \eqref{eq:assumpt1_cond1}  of Assumption \ref{ass:concent:degdist}.

We now consider a variation of the above iid configuration model, proposed by
\cite{VanDenEsker2005}. In this model, we condition the degrees to be less or
equal than some value $D_{\max,t}$, which increases with $t$. More specifically,
for a a pmf $(f_j)_{j\geq 1}$ and a given sequence $(D_{\max,t})_{t\geq 0}$ of
maximum degrees, we
draw a sequence $(D_{t,1},\dots,D_{t,t})$ that are independent and
identically distributed according to
\begin{align}\label{eq:iidoncfigurationconstrained}
\Pr(D_{t,i}=j)= \left \{
\begin{array}{ll}
  \frac{f_j}{\sum_{k=1}^{D_{\max,t}} f_k} & j=1,\ldots,D_{\max,t} \\
  0 & \text{otherwise}
\end{array} \right .
\end{align}
We obtain the degree sequence $(\tilde{D}_{t,1},\dots,\tilde{D}_{t,t})$ of the
configuration model from $(D_{t,1},\dots,D_{t,t})$ by letting
$\tilde{D}_{t,i} = D_{t,i}$ if $i=1,\dots, t-1$, and $\tilde{D}_{t,t} = D_t$ if
$\sum_{i=1}^tD_{t,i}$ is even, $\tilde{D}_{t,t} = D_{t,i} + 1$ otherwise. This
guarantees that $\sum_{i=1}^t\tilde{D}_{t,i}$ is always even, as required to
build the graph \cite{molloy1995critical}.

\begin{lemma}
  \label{lem:configuration-models}
Let $(f_j)_{j\geq 1}$ be a probability mass function on $\{1,2,,\ldots\}$.
Assume $f_1< 1$ and $1 - \sum_{k=1}^jf_k\sim Lj^{-\alpha_1}$ as $j\to\infty$, for some
$\alpha_1\in(0,1)$ and $c>0$. Then the equation
\begin{equation}
  \label{eq:9}
  \sum_{j\geq 1} f_j \sum_{k=1}^{j-1} \frac{\alpha}{k - \alpha}=1.
\end{equation}
has a unique solution $\alpha_0$ in $(0,1)$. Moreover, if
$D_{\max,t}=A N_t^{\frac{1-\alpha_0}{(1+\alpha_0)(1-\alpha_1)}}$ for
some $A > 0$, then there exists $\tau_{*} > 0$ (which can be made
explicit; see the proof for the exact expression) such that the iid
configuration model with constrained maximum degree
\eqref{eq:iidoncfigurationconstrained} satisfies the two conditions
\eqref{eq:assumpt1_cond1} of Assumption \ref{ass:concent:degdist} (in
probability as $t\to \infty$).
\end{lemma}

In general, $\alpha_1\neq \alpha_0$, and the sparsity parameter $\alpha_0$ does not correspond to the exponent of the tail of the degree distribution. An exception is when $(f_j)$ is the Karlin-Rouault distribution \eqref{eq:KarlinRouaultdist}. In this case, $\alpha_1=\alpha_0$.

\paragraph*{Sparse heterogeneous random graphs}

Degree corrected sparse graph models will typically satisfy
\cref{ass:concent:degdist} if the sequence of degrees is chosen accordingly, see
for instance \cite{ouadah20}. We illustrate this point in the next lemma in the
case of the simpler degree corrected Erdös--Rényi graph. The proof of \cref{lem:z9f-xsj8-bpx} is
given in \cite[\cref{supp-sec:yyv-dh1c-y6a}]{caron:naulet:rousseau:supplement}.

\begin{lemma}
  \label{lem:z9f-xsj8-bpx}
  Consider the degree corrected Erdös--Rényi graph where the probability of an
  edge between two vertices $i,j$ is given by:
  \begin{equation*}
    P(Z_{ij}=1)  = \theta_i\theta_j p_N, \quad p_N = p_0N^{-2\alpha_0/(1 + \alpha_0)}.
  \end{equation*}
  Define $\bar{\theta} \coloneqq N^{-1}\sum_{i=1}^N\theta_i$, and suppose $(\theta_i)_{i\leq N}$ is
  chosen so that $p_N(N\bar{\theta})^2 = c_0N^{2/(1+\alpha_0)}$ for some constant $c_0$,
  so that for all $j=1,\dots,N$
  \begin{equation*}
    p(j)  \coloneqq \frac{ \sum_{u=1}^N 1_{\lceil\theta_u Np_N \bar \theta \rceil =j} }{ N }, \quad \text{with} \quad  \sum_{j=2}^N p(j) \sum_{k=1}^{j-1} \frac{ \alpha_0 }{ k - \alpha_0} = 1.
  \end{equation*}
  Then \cref{ass:concent:degdist} is satisfied with $\tau_{*} = c_0^{-(1+\alpha_0)/(2-2\alpha_0)}$.
\end{lemma}

\paragraph*{Edge-exchangeable models and extensions}

Edge-exchangeable (multi)graphs \cite{Crane2018} are natural candidates for
checking \cref{ass:concent:degdist} due to their close proximity with graphex
processes. We show that, however, they behave rather differently and do not
satisfy the \cref{ass:concent:degdist}. To illustrate this, let consider the
so--called \textit{Hollywood process} \cite[Section~5]{Crane2018}. This process
generates almost-surely sequence of multigraphs $(\mathcal{G}_t)_{t> 0}$ whose
degree distributions satisfy for some $\alpha \in (0,1)$
\begin{equation*}
  \frac{N_{t,j}}{N_t} \sim  \frac{\alpha j^{-(1+\alpha)}}{\Gamma(1-\alpha)},\quad t\to\infty.
\end{equation*}
This power-law behaviour of the degree distribution is identical to those of
the model of \cite{Caron2017} (in this respect, see also \cref{rmk:3}), so that one might expect the Hollywood process to
satisfy the \cref{ass:concent:degdist} as well. In fact they don't, because the
Hollywood process generates multigraphs with way more multi-edges. In
particular, from \cite{Crane2018} we see that the number of multi-edges for the
Hollywood process is almost-surely
 $ \nedg \asymp \nver^{1/\alpha},\quad t \to \infty,$
while we need $\nedg \asymp \nver^{2/(1+\alpha)}$ in order to satisfy
\cref{ass:concent:degdist}. This begs the question of what happens if
we consider the simple random graph obtained from the Hollywood process by
merging the multiedges. It is inferred in \cite{Crane2018} that the power-law
behaviour may be preserved in this case and later \cite{Janson2018} has studied
the property of such graphs, yet the analysis is tricky and it is not obvious
if whether or not they can satisfy our assumptions.

Interestingly, \cite{di2017non} proposes a non-exchangeable partition model that
can be used as model for random multigraphs, in the same way that exchangeable
partitions are used in edge-exchangeable models. The non-exchangeability brings
additional flexibility. In particular, replacing the two-parameters Poisson
Dirichlet process used by \cite{Crane2018} to construct the Hollywood process
by the random partition model of \cite[Section~3]{di2017non} adds an extra
parameter $\xi$ such that
$D^*_t\asymp (N_t)^{(1+\xi)/(\sigma+\xi)},\quad t\to \infty$,
while the shape of the degree distribution remains preserved. Taking $\xi=1$ gives
a model with the same behaviour than the model of \cite{Caron2017} and such
model may thus satisfy \cref{ass:concent:degdist}.

\section{The special case of multigraphex processes}
\label{sec:multigraphex}

In this section we study a wide class of random multigraphs, called multigraphex, and show that the assumptions considered in the previous section are satisfied for this class. Multigraphex are extensions of graphex models~\cite{CaronFox2014,Veitch2019,Borgs2018,caron:rousseau:18} to multigraphs. The first multigraphex model was introduced by \cite{CaronFox2014}; the term multigraphex was coined by  \cite{Borgs2019}, who provided a general definition and made the connection between multigraphex and some sparse configuration models.

This Section derives additional properties of multigraphex. In particular it shows that, multigraphex models satisfy \cref{ass:concent:degdist}. Additionally, an interesting feature of these graphs is that most nodes have ties with multiplicity 1, see \cref{lemma:Ntilde}; hence the sparsity property and power-law properties of such multigraph can be directly derived from the properties of graphex models~\cite{caron:rousseau:18}, see \cref{th:asympmultigraphex} and \cref{th:sparsity1}.

\subsection{Definition of multigraphex processes}

A multigraphex process is a growing family of random multigraphs
$(\mathcal G_t)_{t\geq 0}$, indexed by the size $t\geq 0$, where $\mathcal G_t$ is a
multigraph with no isolated vertex (that is, no vertex of degree zero). Ignoring
terms corresponding to isolated edges and (multi-)stars, a multigraphex process
is characterized by a measurable function
$\multigraphon:\mathbb R_+^2\times \NNInts \rightarrow[0,1]$, called a
\textit{multigraphon}, and satisfying
$\multigraphon(x,y,k)=\multigraphon(y,x,k)$, $\sum_{k=0}^\infty \multigraphon(x,y,k)=1$
as well as some integrability conditions. We refer to \cite[Definition
4]{Borgs2019} for details.
To generate a  multigraphex process, one first
draws a unit rate Poisson process
$M \coloneqq \Set{(\vartheta_i,\theta_i) \given i \in \Nats}$ on $\NNReals^2$.
Conditional on $M$, generate an infinite symmetric array
$(\tilde{n}_{i,j})_{(i,j)\in \Nats^2}$ such that for $i \leq j$ the variables
$\tilde{n}_{i,j}$ are independent with distribution on $\NNInts$ given by
\begin{equation*}
  \Pr(\tilde n_{ij}=k\mid M)=\multigraphon(\vartheta_i,\vartheta_j,k),\qquad k \in \NNInts.
\end{equation*}
The infinite array $(\tilde{n}_{i,j})_{(i,j)\in \Nats^2}$ can be seen as the
adjacency matrix of an infinite multigraph $\mathcal{G}$. Then for $t>0$, $\mathcal{G}_t$ is obtained from
$\mathcal{G}$ by keeping only the nodes $i \in \Nats$ such that $\theta_i \leq t$ and such
that they have at least one connection with a node $\theta_j \leq t$ (with possibly
$j=i$). We note that this construction implies that each $\mathcal{G}_t$ has no
isolated vertex. Also, the statistics of interest for us can be easily seen to
satisfy
\begin{equation*}
  D_{t,i} = \sum_k\tilde{n}_{i,k}\1{\theta_{k}\leq t},\quad%
  N_{t,j} = \sum_{i}\1{\theta_{i}\leq t}\1{D_{t,i}=j},\quad%
  \nedg = \sum_{i\geq 1} D_{t,i}\1{\theta_i\leq t}.
\end{equation*}

\begin{remark}
  \label{rmk:4}
  Most of the recent attention on multigraphex processes has been drawn on the
  situation where the multigraphon function is of the form
  \begin{equation}
    \label{eq:poissonmultigraphonrank1}
    \multigraphon(x,y,k)=%
    \begin{cases}
      \Poissonpdf(k;2\bar{\rho}^{-1}(x)\bar{\rho}^{-1}(y) ) & x\neq y\\
      \Poissonpdf(k;\bar{\rho}^{-1}(x)^2) & x = y,
    \end{cases}
  \end{equation}
  with $\Poissonpdf(\cdot;\lambda)$ the probability density function of the $\poiDist(\lambda)$
  distribution, and
  with $\bar{\rho}^{-1}$ the generalized inverse of the tail of a Lévy intensity
  $\bar{\rho}(x) = \int_x^{\infty}\rho(\intd w)$, where $\rho$ is some Lévy measure on
  $\NNReals$. Depending on the properties of the tail L\'evy intensity, one can
  obtain sparse and dense multigraph processes, potentially with power-law
  degree distributions. \cite{Borgs2019} recently showed that this class of
  models is the limit of some configuration and preferential attachment models,
  underlining the relevance of such models for statistical network modeling. It
  is also worth mentioning that, the model of \cite{Caron2017}, which is also
  summarized in \cref{sec:conc-likel-funct}, is easily seen to be a special case
  of multigraphex processes with $W_m$ as in \cref{eq:poissonmultigraphonrank1}
  with $\rho$ the Lévy measure defined in \cref{eq:GGLevyMeasure}. The multigraphs
  are sparse for $\sigma_0\geq 0$ and dense otherwise, with power-law degree
  distribution if $\sigma_0>0$.
\end{remark}

\subsection{Assumptions}

We start by assuming some integrability conditions on functions related to
$W_m$. These integrability conditions are sufficient (though not necessary) for
the number of edges and vertices to be almost surely finite, and the
multigraphex process well defined~\cite[Definition 4]{Borgs2019}. In particular,
we define for $r \geq 0$ the $r$'s multigraphon moment function as
$$
\overline W_r(x,y)=\sum_{k=1}^{\infty}k^r \multigraphon(x,y,k),
$$
and we assume the following.
\begin{assumption}
  \label{ass:integgraphon}
  $W_m$ is such that
  \begin{equation*}
    \int_0^\infty \overline W_2(x,x)dx<\infty,\ %
    \mathrm{and},\quad \int_{\mathbb R_+^2} \overline W_2(x,y)dxdy<\infty.
  \end{equation*}
\end{assumption}

In many occasions it is of interest to consider the simple graph obtained from a
multigraphex process by merging the multiedges. It is easily seen that a simple
graph obtained by this mechanism is a ``standard'' graphex process in the sense
of \cite{Caron2017,Veitch2015}, with graphon function given by
\begin{align}
  \label{eq:defW}
  W(x,y) \coloneqq 1-\multigraphon(x,y,0)
\end{align}
In particular, it is a consequence of  \cref{lemma:Ntilde} below that many
properties of the multigraph can be deduced from their counterparts in the
simple graph model with graphon $W$. It is known from
\cite{Veitch2015,caron:rousseau:18} that a key quantity to characterize the
behaviour of the simple graph is the marginal graphon function%
\begin{equation*}
  \mu(x)\coloneqq\int_0^\infty W(x,y)dy=\int_0^\infty (1-W_m(x,y,0))dy.
\end{equation*}

We observe that the \cref{ass:integgraphon} implies that $\mu$ is
an integrable function.

The following assumption, used by \cite{caron:rousseau:18}, characterises the
behaviour of $\mu$ at infinity or, equivalently, of its generalised inverse
$\mu^{-1}$ at 0. We require $\mu^{-1}$ to behave approximately as a power function
$x^{-\alpha_0}$ around 0, for some $\alpha_0\in[0,1]$. As demonstrated in the simple graph
case in \cite{caron:rousseau:18}, this power law behaviour drives the
asymptotics of key statistics of the graph such as the number of vertices or the
degree distribution. \cref{th:asympmultigraphex} establishes the same result
in the multigraph context.

\begin{assumption}
  \label{assumpt:1}
  Assume $\mu$ is non-increasing, right-continuous, with generalised inverse
  $\mu^{-1}(x)=\inf\Set{ y>0 \given \mu(y)\leq x}$, such that
  \begin{equation*}
    \mu^{-1}(x)\sim\ell_1(1/x)x^{-\alpha_0}\text{ as }x\rightarrow0,
  \end{equation*}
  where $\alpha_0\in\lbrack0,1]$ and $\ell_1$ is a slowly varying function at
  infinity.
\end{assumption}

Many examples of (simple) graphon functions are studied  in\cite{caron:rousseau:18}
which satisfy \cref{assumpt:1}.
In particular, we show in \cref{thm:2} that the GGP model of \cite{Caron2017} satisfies
\cref{assumpt:1} with $\ell_1$ a constant function.

The following is a technical assumption needed in order to obtain the
almost sure results. A similar assumption has been made by~\cite{Veitch2015} and
\cite{caron:rousseau:18} for the analysis of graphex models. In particular it
permits to obtain sufficiently sharp bounds on the variance of the key
statistics of the multigraphs. It is slightly stronger than Assumption 2 in \cite{caron:rousseau:18}.
\begin{assumption}\label{assumpt:2} Let $\nu(x,y)\coloneqq\int_0^\infty W(x,z)W(y,z)dz$,
and assume that there exists a positive, locally bounded function $\ell_3$ such
that $\ell_3(x)$ converges to a strictly positive constant as $x\to\infty$, and
for all $x,y > 0$
\begin{align*}
  \nu(x,y)\leq \ell_3(x)\ell_3(y)\mu(x)^{a}\mu(y)^a,\quad%
  \begin{cases}
    a>\max\left(\frac{1}{2},\alpha_0\right) & \mathrm{if}\ \alpha_0\in[0,1),\\
    a=1 & \mathrm{if}\ \alpha_0=1.
  \end{cases}
\end{align*} 
\end{assumption}

In comparison with \cite{caron:rousseau:18} which only deals with the simple
graph models, in order to study multigraphex we
need to consider assumptions on the distribution of the number of multiedges
given that there is an edge. This is provided by the following assumption.

\begin{assumption}
  \label{ass:Wm}
  Assume that, for all $x,y > 0$,
  \begin{align*}
    W_m(x,y,1)\geq W(x,y)(1-W(x,y)).
  \end{align*}
\end{assumption}

Clearly the \cref{ass:Wm} is satisfied if $W_m$ is of the form of
\cref{eq:poissonmultigraphonrank1}, and hence it is satisfied by the important
class of graphex processes considered in \cite{Borgs2019} and also by the model
of \cite{Caron2017}. Interestingly, the \cref{ass:Wm} is also satisfied if $W_m$
is of the form of \cref{eq:poissonmultigraphonrank1} but with the Poisson
distribution replaced with a Binomial, a Geometric, or even a Negative Binomial
distribution.

The \cref{ass:integgraphon,assumpt:1,assumpt:2,ass:Wm} allow us to derive the
first order asymptotic behaviour of $\nver, \nedg, N_{\TrueSize,j}$. However to
check if the multigraphex satisfies \cref{ass:concent:degdist} we need to
consider a strengthened version of \cref{assumpt:1}.
\begin{assumption}
  \label{ass:tails}
There exist $\beta, c_0  >0$   such that
  \begin{equation*}
    |\mu^{-1}(x) - x^{ -\alpha_0} c_0 | = O( x^{\beta-\alpha_0}) , \quad \text{when } x \rightarrow 0.
  \end{equation*}
 \end{assumption}

 \Cref{ass:tails} is a second order assumption on the tail behaviour of $\mu^{-1}$
 similar to the assumptions found in extreme value theory, see for instance
 \cite{carpentier2015adaptive,boucheron2015tail}. The first-order assumption
 \cref{assumpt:1} is sufficient to show that $\qMLEspars \to \alpha_0$ almost-surely.
 Unfortunately, to ensure the concentration of the other parameters
 $(\qMLEtau,\qMLEsize)$, we need to establish that
 $\log(t)|\hat{\alpha}_t - \alpha_0| = o(1)$, which cannot be guaranteed if we only assume
 $\mu^{-1}(x) \sim x^{-\alpha_0}\ell_1(1/x)$; whence the second order assumption.

\subsection{Sparsity and Power-law properties of multigraphex processes}

Here we derive sparsity and power-law properties of multigraphex
processes. These results are the counterpart of the results derived by
\cite{Caron2017} for (simple) graphex processes.

The results follow from the following Proposition which shows that the number of
nodes connected to at least one edge of multiplicity larger than 1 is
negligible.

\begin{proposition}
  \label{lemma:Ntilde}
  Let $(\mathcal G_t)$ be a multigraphex with multigraphon function $W_m$. Let
  $\widetilde{\mathcal G}_t$ denote the subgraph of $\mathcal G_t$ obtained by
  removing all edges of multiplicity 1, and all nodes whose ties all have
  multiplicity 1. Denote $\widetilde N_t$ the number of nodes in
  $\widetilde{\mathcal G}_t$. Suppose
  \cref{ass:integgraphon,assumpt:1,assumpt:2,ass:Wm} hold for some $\alpha_0>0$. We
  have, for any $\delta>0$,
  $$
  \widetilde N_t =o(t^{1+\alpha_0/2+\delta})=o(N_t)
  $$
  almost surely as $t\to\infty$.
\end{proposition}

The proof of \cref{lemma:Ntilde} is given in \cref{sec:proof-crefl}. As a
consequence of \cref{lemma:Ntilde}, most of analyses involving multigraphexes
can be reduced to the analysis of a classical graphex process, \textit{i.e.} a
simple graph analysis. This turns out to be convenient since asymptotic
behaviour of simple graphex processes has been extensively studied in
\cite{caron:rousseau:18}. In particular, from their result, we can establish the
following theorem for multigraphex processes.

\begin{theorem}[Asymptotic properties of the multigraphex process]
  \label{th:asympmultigraphex}
  Suppose \cref{ass:integgraphon} holds. All the results below hold in mean and
  almost surely as $t\to\infty$. We have
  \begin{equation*}
    D_{t}^{\ast}\sim t^{2}\int_{\mathbb{R}_{+}^{2}}\oW_1(x,y)dxdy,\quad%
    N_{t}\sim%
    \begin{cases}
      t^{1+\alpha_{0}}\Gamma(1-\alpha_{0})\ell_{1}(t) &\mathrm{if}\ \alpha_{0}\in \lbrack0,1)\\
      t^{2}\widetilde{\ell}_{1}(t) & \mathrm{if}\ \alpha_0 =1,
    \end{cases}
\end{equation*}
where the result on $\nver$ holds under the extra \cref{assumpt:1,assumpt:2} for
some $\alpha_0\in[0,1]$ and some slowly varying function $\ell_1$, with
$ \widetilde{\ell}_{1}(t)=\int_{t}^{\infty}y^{-1}\ell_{1}(y)dy$. If $\alpha_{0}=0$, then for any
$j\geq1$ we have
$ N_{tj}=o(t\ell_{1}(t)). $ Assume now \cref{ass:Wm} also hold. Then, for any
$j\geq1$,%
\begin{equation*}
  N_{t,j}\sim\frac{\alpha_{0}\Gamma(j-\alpha_{0})}{j!}t^{1+\alpha_{0}}\ell
  _{1}(t).
\end{equation*}
Finally, if $\alpha_{0}=1$, then  for any $j\geq1$  we have
\begin{equation*}
  N_{t,j}%
  \sim
  \begin{cases}
    t^{2}\widetilde{\ell}_{1}(t) &\mathrm{if}\ j= 1,\\
    o(t^{2}\widetilde{\ell}_{1}(t)) &\mathrm{if}\ j\geq 2.
  \end{cases}
\end{equation*}
\end{theorem}

The proof of \cref{th:asympmultigraphex} is given in
\cite[\cref{supp-sec:proof-theorem-asympmultigraphex}]{caron:naulet:rousseau:supplement}. The following is a straightforward
corollary showing giving the power-law and sparsity properties of the
multigraph.
\begin{corollary}[Sparsity and power-law degree distribution]
  \label{th:sparsity1}
  Suppose \cref{ass:integgraphon,assumpt:1,assumpt:2,ass:Wm} hold. The multigraph is dense if $\alpha_0=0$ and $\lim_{t\rightarrow\infty} \ell_1(t)=C<\infty$, as
$D_{t}^{\ast}/N_{t}^2\rightarrow C^2\int_{\mathbb{R}_{+}^{2}}\oW_1(x,y)dxdy $ almost surely. Otherwise, if $\alpha_0>0$ or $\alpha_0=0$ and $\lim_t \ell_1(t)=\infty$, the multigraph is sparse, as
$D_{t}^{\ast}/N_{t}^2\rightarrow 0$. Additionally, for $\alpha_0\in [0,1)$, for any $j=1,2,\ldots$,
\begin{align*}
  \frac{N_{t,j}}{N_t}\rightarrow \frac{\alpha_0 \Gamma(j-\alpha_0)}{j!\Gamma(1-\alpha_0)}
\end{align*}
almost surely as $t\to\infty$. If $\alpha_0>0$, this corresponds to a degree distribution with a power-law behaviour as, for $j$ large
$$
\frac{\alpha_0 \Gamma(j-\alpha_0)}{j!\Gamma(1-\alpha_0)}\sim \frac{\alpha_0}{\Gamma(1-\alpha_0)j^{1+\alpha_0}}.
$$
For $\alpha_0=1$,
$N_{t,1}/N_t\rightarrow 1$ and $N_{t,j}/N_t\rightarrow 0\text{ for $j\geq 2$}
$, hence the nodes of degree 1 dominate in the graph.
\end{corollary}

\subsection{Multigraphex processes satisfy \texorpdfstring{\cref{ass:concent:degdist}}{Assumption \ref{ass:concent:degdist}}}
\label{sec:rates-conv-mle}

Here we prove that sequence of multigraphs issued from certain multigraphex
processes satisfy the \cref{ass:concent:degdist}, and thus for those processes, \cref{thm:concent:1} and  \cref{posterior:cons:1} hold.


\begin{theorem}
  \label{thm:2}
  Let $(\mathcal{G}_t)_{t>0}$ be generated according to a multigraphex process
  with multigraphon function $W_m$ satisfying all the
  \cref{assumpt:1,assumpt:2,ass:Wm,ass:tails}. If in addition we have
  $\int_{\NNReals}\mu(x)^2\,\intd x < \infty$, then \cref{ass:concent:degdist} holds in
  probability with
  \begin{equation*}
    \tau_{*}%
    = \Bigg\{%
    \frac{2\alpha_0 c_0 \Gamma(1-\alpha_0)}{\big(2 \int_{\NNReals^2}
      \bar{W}_1(x,y)\intd x\intd y \big)^{(1+\alpha_0)/2}}
    \Bigg\}^{1/(1-\alpha_0)}.
  \end{equation*}
\end{theorem}

The proof of \cref{thm:2} is given in \cref{sec:proof-thm:2-1}.

\begin{remark}
  \label{rmk:3}
  The fact that $\ell_1$ converges to a constant in the previous theorem seems
  necessary to ensure concentration of the MLE. Indeed, by analysis of the proof
  of the theorem, we infer that if $\ell_1$ converges to zero or diverge, then
  either $\qMLEtau \to 0$ and $\qMLEsize/t \to 0$ or $\qMLEtau \to \infty$ and
  $\qMLEsize/t \to \infty$ in probability. The difficulty comes in particular from the
  fact that the size of the network is inferred.%
\end{remark}

\begin{remark}
  \label{rmk:2}
  The results shown in \cref{thm:concent:1} suggest that $\hat{\alpha}_t$ can also be used as an alternative to
  \cite{Naulet2017,caron:rousseau:18} for producing an estimator for the
  tail-index of graphex processes.
\end{remark}

%
%

\section{Proofs of the main results}
\label{sec:proofs-main-results}

\subsection{Approximation of the likelihood: Theorem \ref{lem:3}}
\label{sec:proof-lem:3}

A key step to analyse the likelihood which is only available in the form of an
  integral is to obtain
  an  approximation of the integral in \cref{eq:likelihood}.
In \cref{lem:3}, we obtain an upper-bound on the log-likelihood which is valid
for all parameters $\params$, and an asymptotic equivalence which is valid
uniformly over all parameters in the bounded sets
\begin{equation*}
  \mathcal{S}_K%
  \coloneqq \Set*{ \params = (\spars,\tau,\size) \given -K \leq \spars < 1,\, K^{-1} \leq
    \tau \leq K,\, 0 < \size \leq K \sqrt{2\nedg} },
\end{equation*}
with $K>0$ an arbitrarily large number.
\begin{theorem}
  \label{lem:3}
  For every $\spars \in (-\infty,1)$, every $\tau > 0$ and every $\size > 0$,
  let $\params = (\spars,\tau,\size)$, and let $\zeta(\params)$ be a
  non-negative real solution to $\partial_z\mathcal{A}(\params;z) = 0$. Then,
  \begin{enumerate}
    \item\label{item:lem:3:1} For every $\params$, $\zeta(\params)$ exists uniquely.
    \item\label{item:lem:3:2} There exists $\mathcal{E}_t \in \Reals$ such that
    for every $\params$,%
    \begin{equation*}
      \loglik(\params)%
      \leq%
      \nver \log(\size) + \sparsfunc(\spars)%
      - \nedg \mathcal{A}(\params; \zeta(\params)) + \mathcal{E}_t,
    \end{equation*}
    \item\label{item:lem:3:3} For the same $\mathcal{E}_t$ as in
    \cref{item:lem:3:2}, for all $K > 1$, as $\nedg \to \infty$,
    \begin{equation*}
      \sup_{\params\in \mathcal{S}_K}\Big|\loglik(\params)%
      - \nver \log (s) - \sparsfunc(\spars) + \nedg
      \mathcal{A}(\params;\zeta(\params))%
      - \mathcal{E}_t%
      + \frac{\log(2)}{2}\Big|  =  o(1).
    \end{equation*}
  \end{enumerate}
\end{theorem}

We can see from
\cref{item:lem:3:3} that the upper-bound of \cref{item:lem:3:2} is sharp up to
$\log(2)/2$ for all $\params$ in $\mathcal{S}_K$. Interestingly, we note that
\cref{item:lem:3:3} gives an upper bound on $\loglik(\params)$ which is valid
for all $\params \in \mathcal{S}$, while the asymptotic equivalence only holds
uniformly on $\mathcal{S}_K$. This is not a problem, as to understand the
maximum likelihood estimator we only need to lower bound $\loglik(\phi)$ for $\phi$ in
a neighborhood of the maximizer. Indeed, by \cref{lem:3}, we have for any
$\params \in \mathcal{S}$ that 
$\loglik(\params) - \loglik(\qMLEparams) \leq \qloglik(\params) -
\qloglik(\qMLEparams) + \frac{\log(2)}{2} + o(1) \leq 0$.
$\params$ is away from $\qMLEparams$, where the latter is proved in \cref{thm:concent:1}. Hence, this will imply that any maximizer
of $\loglik$ must be close to $\qMLEparams$.

\begin{proof}[Proof of \cref{lem:3}]
Here we prove only \cref{item:lem:3:1,item:lem:3:2}. The proof of the
\cref{item:lem:3:3} is similar to the proof of \cref{item:lem:3:2} but with
extra cares to handle the asymptotic equivalence instead of simply getting an
upper bound. For this reason, the proof of the \cref{item:lem:3:3} is deferred
to \cite[\cref{supp-sec:proof-lem:6}]{caron:naulet:rousseau:supplement}. The proof of \cref{item:lem:3:1} is a
consequence of the \cref{pro:4} below. We now prove \cref{item:lem:3:2}. Set
\begin{equation*}
  I(\params) \coloneqq%
  \int_{\NNReals^2}%
  e^{-(x+y)^{2}}\,
  \frac{%
    x^{\nedg - \spars\nver -1} e^{-\tau x}%
  }{%
    \Gamma(\nedg - \spars\nver)%
  }%
  \tiltstab(\intd y)\intd x.
\end{equation*}
Then, by \cref{eq:likelihood}, there is a constant $\mathcal{E}_t \in \Reals$ such
that $\loglik$ can be written as
\begin{align}
  \label{eq:80}
  \loglik(\params )%
  &= \nver \log \size + \sum_{i=1}^{\nver}\log\frac{\Gamma(\deg{i} -
    \spars)}{\Gamma(1-\spars)} + \log I(\params) + \mathcal{E}_t \nonumber\\
  &= \nver \log \size + \sum_{j\geq 2}\sum_{i=1}^{\nver}\bm{1}_{\deg{i}=j}\log\frac{\Gamma(j -
    \spars)}{\Gamma(1-\spars)} + \log I(\params) + \mathcal{E}_t \nonumber\\
  &= \nver \log \size + \sparsfunc(\spars) + \log I(\params) + \mathcal{E}_t.
\end{align}
Hence to approximate $\loglik(\params )$ we need only approximate $I(\params)$.

 The first step is to express $I(\params) = (2 \sqrt{\pi})^{-1} \int_{\Gamma} \exp\{-\nedg
\mathcal{A}(\params;z) \}\,\intd z$ where $\Gamma$ denotes the line in the
complex plane that goes from $(\tau, + i \infty)$ to $(\tau, -i\infty)$. This is done  in
\cref{pro:4} by noting that
$
 \int_{\Reals}\exp\{ -\frac{\xi^2}{4} + i(x+y)\xi
  \}\, \intd \xi
  = 2\sqrt{\pi}\exp\{ -(x+y)^2 \}.
$
 This trick allows in particular to separate $x$ and $y$ inthe integral. We then analyse $I(\params)$ using a saddlepoint approximation: first, by the \cref{pro:4},
the saddlepoint $\zeta(\params)>\tau$. Thus, we can deform the contour of integration so that it
goes through $\zeta(\params)$. For every $L > 0$ let $\mathcal{C}_L$ be the
contour that goes from $(\tau,+iL)$ to $(\tau,-iL)$, then from $(\tau,-iL)$ to
$(\zeta(\params),-iL)$, and then from $(\zeta(\params),-iL)$ to
$(\zeta(\params),+iL)$, and finally from $(\zeta(\params),+iL)$ to
$(\tau,+iL)$. Then $\mathcal{C}_L$ is a rectifiable and closed path, and
$z \mapsto \mathcal{A}(\params;z)$ is complex-analytic in the region delimited
by $\mathcal{C}_L$. Hence Cauchy's integral theorem implies that
\begin{equation*}
  \int_{\mathcal{C}_L}\exp\{-\nedg \mathcal{A}(\params;z) \}\,\intd z = 0.
\end{equation*}
As $L \rightarrow \infty$, the integrals of
$z\mapsto \exp\{-\nedg \mathcal{A}(\params;z) \}$ from $(\tau,-iL)$ to
$(\zeta(\params),-iL)$ and from $(\zeta(\params),+iL)$ to $(\tau,+iL)$ vanish
and thus we obtain from the previous display that,
\begin{align*}
  I(\params)%
  &= - \frac{1}{2 \sqrt{\pi}}
    \int_{(\zeta(\params),-i\infty)}^{(\zeta(\params),+i\infty)}%
    \exp\{-\nedg \mathcal{A}(\params;z)\}\,\intd z\\
  &= \frac{1}{2 \sqrt{\pi}} \int_{\Reals}\exp\{-\nedg \mathcal{A}(\params;
    \zeta(\params) - iu) \}\,\intd u.
\end{align*}

Now let $R_1(\params;z)$ denote the second-order remainder of the Taylor
expansion of $z \mapsto \mathcal{A}(\params;z)$ near $\zeta(\params)$, that is we have,
\begin{equation}
  \label{eq:1}
  \mathcal{A}(\params;\zeta(\params) + z)%
  = \mathcal{A}(\params; \zeta(\params) )%
  + \frac{z^2}{2} \partial_z^2\mathcal{A}(\params;\zeta(\params) ) + R_1(\params;z),
\end{equation}
and then,
\begin{equation}
  \label{eq:6}
  I(\params) = \frac{\exp \big\{- \nedg \mathcal{A}(\params; \zeta(\params))
    \big\}}{\sqrt{- 2\nedg \partial_z^2\mathcal{A}(\params; \zeta(\params) ) }}\Big(%
  1 + \Delta(\params) \Big),
\end{equation}
where,
\begin{equation}
    \label{eq:8}
    \Delta(\params)%
    \coloneqq%
    -1%
    + \sqrt{%
      \frac{%
        - \nedg\partial_z^2\mathcal{A}(\params; \zeta(\params))%
      }{2\pi}}%
    \int_{\Reals} \exp\Big\{\frac{\nedg}{2}\partial_z^2\mathcal{A}(\params; \zeta(\params))
    z^2 -\nedg R_1(\params;-iz) \Big\}\,\intd z.
\end{equation}

It is easily seen that $\partial_z^2\mathcal{A}(\params;\zeta(\params))$ is pure
real and always (strictly) negative, so that \cref{eq:6,eq:8} are indeed
well-defined.  Using \cref{pro:5} combined with
 \cref{eq:8}, we find that
\begin{equation*}
  |1 + \Delta(\params)|%
  \leq \sqrt{%
    \frac{%
      - \nedg\partial_z^2\mathcal{A}(\params; \zeta(\params))%
    }{2\pi}}%
  \int_{\Reals} e^{-z^2/4}\,\intd z%
  = \sqrt{-2\nedg \partial_z^2\mathcal{A}(\params;\zeta(\params))}.
\end{equation*}
The last display and \cref{eq:6} give the bound which established the
\cref{item:lem:3:2},
\begin{equation*}
  \log I(\params)%
  \leq - \nedg \mathcal{A}(\params;\zeta(\params)).
\end{equation*}

\end{proof}

\begin{lemma}
  \label{pro:4}
  For every $\params$, let $(\params,z) \mapsto \mathcal{A}(\params;z)$ be the
  function defined in \cref{eq:2}. Then, for every $\params$, we have
  \begin{equation} \label{eq:Iphi}
    I(\params) =%
    \frac{1}{2\sqrt{\pi}} \int_{\Reals} \exp\{ - \nedg
    \mathcal{A}(\params;\tau - i\xi )\}\,\intd\xi.
  \end{equation}
  In addition, there exists a unique non-negative solution $\zeta(\params)$ for
  the equation $\partial_z \mathcal{A}(\params;z) = 0$, which satisfies
  $\zeta(\params) > \tau$ and
  \begin{equation}
    \label{eq:7}
    \zeta(\params)^2%
    = \tau \zeta(\params)%
    + 2\size \zeta(\params)^{\spars} + 2\nedg \Big(1 - \frac{\spars\nver}{\nedg}\Big).
  \end{equation}
\end{lemma}

\begin{lemma}
  \label{pro:5}
  For every $\params$ and for every $z \in \Reals$,
  \begin{equation*}
    \Re\Big(\frac{\nedg}{2}\partial_z^2\mathcal{A}(\params; \zeta(\params)) z^2
    -\nedg R_1(\params;-iz) \Big)%
    \leq - \frac{z^2}{4},
  \end{equation*}
  where $\Re(z)$ denotes the real part of $z \in \Complex$.
\end{lemma}

The proofs of \cref{pro:4,pro:5} as well as the remaining parts of the proof of
the \cref{lem:3} can be found in \cite[\cref{supp-sec:proofs-relat-asympt}]{caron:naulet:rousseau:supplement}.

\subsection{Proof of \texorpdfstring{\cref{thm:concent:1}}{Theorem \ref{thm:concent:1}}}
\label{sec:concent:proofs-conc-likel-funct}

A first difficulty in the proof of \cref{thm:concent:1} is to show that there
exists a unique MLE. Indeed the function $\qloglik$ is a complicated function of
$\params$. We first introduce a new parameterization which simplifies
its analysis. Define
\begin{equation*}
  \varepsilon(\params)%
  \coloneqq 1 - \frac{2\size \zeta(\params)^{\spars} + 2\nedg(1-\spars\nver/\nedg)}{\zeta(\params)^2},\qquad%
  u(\params) \coloneqq \frac{2\size
    \zeta(\params)^{\spars}}{2\nedg(1 - \spars\nver/\nedg)}.
\end{equation*}
Observe that $\varepsilon(\params) \in (0,1)$ and $u(\params) > 0$. We also define for
convenience $\beta_{\spars} \coloneqq 1 - \spars \nver/\nedg$. Our goal is to
re-express the likelihood in term of the new parameterization $(\spars,\varepsilon,u)$. By
\cref{eq:7}, $\tau = \zeta(\params)\cdot \varepsilon(\params)$; therefore we can also rewrite
$\zeta(\params) - \tau = \zeta(\params)\cdot(1 - \varepsilon(\params))$ so that
\begin{equation*}
  \zeta(\params)^2%
  = 2\nedg \beta_{\spars}%
  \frac{1 + u(\params)}{1 - \varepsilon(\params)}.
\end{equation*}
Define $f(\spars,\varepsilon) \coloneqq (1 - \varepsilon^{\spars})/\spars$ if $\spars\ne 0$, and
$f(\spars,\varepsilon) = \log(1/\varepsilon)$ if $\spars = 0$. Then, we have
$\size \LTS(\spars,\tau;\zeta(\params)) = \nedg \beta_{\spars} u(\params) f(\spars,\varepsilon(\params))$
for all $\spars \in (-\infty,1)$ and
$\size \LTS(\spars,\tau;\zeta(\params)) = \nedg \beta_{\spars}u(\params)g(\spars,\varepsilon(\params)) + \frac{\nedg}{2}(1-\varepsilon(\params)) \beta_{\spars} u(\params)$,
where $g(\spars,\varepsilon) \coloneqq f(\spars,\varepsilon) - \frac{1 -\varepsilon}{2}$. Therefore,
\begin{align*}
  \qloglik(\params)%
  &= \nver \log u(\params)%
    + \nver \log\nedg%
    + \nver \log \beta_{\spars}%
    +\sparsfunc(\spars)%
    + \frac{\nedg}{2}\beta_{\spars}(1 + u(\params))(1-\varepsilon(\params))\\
  &\quad%
    -\frac{\nedg}{2}\log(2\nedg)%
    - \frac{\nedg}{2}\log \beta_{\spars}%
    + \frac{\nedg}{2}\log(1 - \varepsilon(\params))%
    - \frac{\nedg}{2}\log(1 + u(\params))\\
  &\quad%
    - \nedg \beta_{\spars}u(\params) g(\spars,\varepsilon(\params))%
    - \frac{\nedg}{2}(1 - \varepsilon(\params))\beta_{\spars}u(\params)%
    + \frac{\log(2)}{2}.
\end{align*}
This leads us to define,
\begin{multline*}
  \qloglik^{*}(\spars,\varepsilon,u)%
  \coloneqq%
  \nver \log \beta_{\spars}%
  - \frac{\nedg}{2}\log \beta_{\spars}%
  + \sparsfunc(\spars)%
  + \nver \log u%
  + \frac{\nedg}{2}(1-\varepsilon)\beta_{\spars}\\%
  + \frac{\nedg}{2}\log(1-\varepsilon)%
  - \frac{\nedg}{2}\log(1+u)%
  - \nedg \beta_{\spars}ug(\spars,\varepsilon).%
\end{multline*}
which satisfies  $\qloglik(\params) -
\qloglik(\params') = \qloglik^{*}(\spars,\varepsilon(\params),u(\params)) -
\qloglik^{*}(\spars',\varepsilon(\params'),u(\params'))$ for any pair
$(\params,\params') \in \mathcal{S}^2$. In many occasions, we will decompose
$\qloglik^{*}(\spars,\varepsilon,u)$ as
$\qloglik^{*}(\spars,\varepsilon,u) = H(\spars,\varepsilon,u) + K(\spars)$, where
\begin{gather*}
  H(\spars,\varepsilon,u)
  \coloneqq \nver \log u
  - \frac{\nedg}{2}\log(1 + u)
  - \nedg\beta_{\spars}g(\spars,\varepsilon)u
  + \frac{\nedg}{2}\log(1- \varepsilon)
  + \frac{\nedg}{2}(1-\varepsilon)\beta_{\spars},\\
  K(\spars)%
  \coloneqq \nver \log \beta_{\spars}%
  - \frac{\nedg}{2}\log \beta_{\spars} + \sparsfunc(\spars).
\end{gather*}

From that we analyze
$\Psi(\spars) \coloneqq \sup\Set{\qloglik^{*}(\spars,\varepsilon,u) \given \varepsilon \in (0,1),\, u > 0 }$,
proving first that
$\Psi(\spars) =\qloglik^{*}(\spars,\tilde{\varepsilon}(\spars),\tilde{u}(\spars))$ where
$(\tilde{\varepsilon}(\spars),\tilde{u}(\spars))$ is unique and that under
\cref{ass:concent:degdist}, $\Psi$ has a unique maximizer $\qMLEspars$, which
satisfies $\lim_t\qMLEspars = \alpha_0$.

We first show in \cref{pro:concent:4} that $\Psi$ is small for $\sigma \leq -C$ or $\sigma >c_2$ for any $C>0$ and if $c_2<1$ is large enough. More precisely, for all $C>0$, $c_0 \in (0,1)$, there exists $c_2\in (0,1)$ and $K>0$ such that
$\Psi(\spars) - \Psi(c_0) \leq -K \log \nedg $ over $(-\infty , -C) \cup [c_2, 1)$, which implies
that the maximizer must be within $[-C,c_2]$.

Moreover from \cref{pro:concent:5}, for any $\spars \in [-C,c_2]$, there exists
a unique $(\tilde{\varepsilon}(\spars),\tilde{u}(\spars))$ such that
$\Psi(\spars)
=\qloglik^{*}(\spars,\tilde{\varepsilon}(\spars),\tilde{u}(\spars))$, also
$\Psi'(c_2)<0$ and there exists $c_1>0$ such that $\Psi'(\spars) >0 $ on
$[-C,c_1]$ so that $\Psi(\spars) \leq \Psi(c_1) $ if $\spars \leq c_1$ and any
maximizer of $\Psi$ must belong to $[c_1,c_2]$.  From \cref{pro:concent:11}
$\Psi''(\spars) <0$ on $[c_1,c_2]$, and since $\Psi'(c_1)>0, \Psi'(c_2)<0$, there is a
unique $\qMLEspars \in (c_1,c_2)$ which maximizes $\Psi$. Hence
$(\qMLEspars, \tilde{\varepsilon}(\qMLEspars),\tilde{u}(\qMLEspars))$ maximizes
$\qloglik^{*}$ and it must be the unique maximizer.

We now study $\qMLEparams$ and we prove \cref{thm:concent:1:item:1}. Indeed, we
established that $\qloglik^{*}$ has a unique maximizer
$(\hat{\sigma}_t,\hat{\varepsilon}_t,\hat{u}_t)$ with
$\hat{\varepsilon}_t \coloneqq \tilde{\varepsilon}(\hat{\sigma}_t)$ and
$\hat{u}_t = \tilde{u}(\hat{\sigma}_t)$, yet it remains to establish that this
implies that $\qloglik$ also has a unique minimizer. We proceed by obtaining
asymptotic expressions for $(\hat{\sigma}_t,\hat{\varepsilon}_t,\hat{u}_t)$ and showing that
the map $\phi \mapsto (\sigma,\varepsilon(\phi),u(\phi))$ is locally invertible. Note that if
$\spars \in [c_1,c_2] $ then
$\partial_{\spars}\log f(\spars,\varepsilon) = -\frac{1 + O(\varepsilon^{\spars}\log\varepsilon^{-1})}{\spars}$.
Hence, by \eqref{eq:concent:103} in \cref{pro:concent:5}, it comes that
$\qMLEspars$ must satisfy
\begin{align*}
  - \sparsfunc'(\qMLEspars)%
  &= \frac{\nver}{\qMLEspars}\Big(1 + O\Big(\frac{\nver}{\nedg}\Big)
    \Big)\Big(1 + O\Big(\qMLEeps^{\qMLEspars}\log \frac{1}{\qMLEeps}\Big)\Big)
    + O\Big( \frac{\nver^2}{\nedg}\Big)\\
  &= \frac{\nver}{\qMLEspars}\Big(1 + O\Big(\frac{\nver}{\nedg} \bigvee
    \qMLEeps^{\qMLEspars}\log \frac{1}{\qMLEeps}   \Big) \Big).
\end{align*}
Moreover using \cite[\cref{supp-pro:concent:3}]{caron:naulet:rousseau:supplement}, since
$\qMLEu = \bar{u}(\qMLEspars,\qMLEeps)$ and
$\frac{1}{2} + \beta_{\spars}g(\qMLEspars,\qMLEeps) = \frac{1}{2} + \frac{1-\qMLEeps^{\qMLEspars}}{\qMLEspars} - \frac{1-\qMLEeps}{2} + O(\nver/\nedg) = \frac{1}{\qMLEspars}\big(1 + O(\frac{\nver}{\nedg}\bigvee \qMLEeps^{\qMLEspars})\big)$,
then,
\begin{align}\label{hatu}
  \qMLEu = \frac{\qMLEspars\nver}{\nedg}\Big(1 +
  O\Big(\frac{\nver}{\nedg}\bigvee \qMLEeps^{\qMLEspars} \Big)\Big)=o(1)
\end{align}
Finally since $\partial_{\varepsilon}H(\qMLEspars,\qMLEeps,\qMLEu) = 0$, we deduce that,
\begin{align*}
  0
  &=\nedg\qMLEu\Big(1 - \frac{\qMLEspars \nver}{\nedg}\Big)
    \Big(\qMLEeps^{-1+\qMLEspars} - \frac{1}{2}\Big)%
    - \frac{\nedg}{2}\frac{1}{1- \qMLEeps} - \frac{\nedg}{2}\Big(1 -
    \frac{\qMLEspars\nver}{\nedg}\Big)\\
  &= \nedg\Big\{ \qMLEu\Big(1 - \frac{\qMLEspars \nver}{\nedg}\Big)%
    \Big(\qMLEeps^{-1+\qMLEspars} - \frac{1}{2}\Big)%
    - 1 + O\Big( \frac{\nver}{\nedg}\bigvee \qMLEeps\Big) \Big\}.
\end{align*}
That is, using \eqref{hatu},
\begin{align*}
  \qMLEeps^{-1+\qMLEspars}%
  &= \frac{1}{2} + \frac{1 + O\big(\frac{\nver}{\nedg}\bigvee \qMLEeps
    \big)}{\qMLEu}%
    = \frac{\nedg}{\qMLEspars \nver}\Big(1 + O\Big( \frac{\nver}{\nedg}\bigvee
    \qMLEeps^{\qMLEspars}\Big) \Big).
\end{align*}
From the previous, we also obtain
$\qMLEeps \asymp \big(\frac{\nver}{\nedg}\big)^{1/(1-\qMLEspars)}$. Then, under
\cref{ass:concent:degdist}, since $x \rightarrow x\sparsfunc'(x)$ is decreasing and since
its derivative is bounded from below by $-C \nver$ for some
$C>0$,
\begin{equation} \label{qMLspars}
|\hat\alpha_t - \qMLEspars| = O\Big(\frac{\nver}{\nedg} \bigvee \qMLEeps^{\qMLEspars}\log \frac{1}{\qMLEeps} \Big)
\end{equation}
and
\begin{equation}
  \label{qMLeps}
  \qMLEeps \asymp 1 / \sqrt{\nedg}, \quad \text{and} \quad
  \qMLEeps^{\qMLEspars} \asymp \frac{\sqrt{\nedg}}{\nver}.
\end{equation}
This finished to establish the asymptotic expressions of
$(\hat{\sigma}_t,\hat{\varepsilon}_t,\hat{u}_t)$. To prove that $\qMLEparams$ exists and is
unique we prove in \cref{pro:concent:16} that
\begin{equation}
  \label{eq:14}
  \sup_{(\spars, \epsilon, u) \notin \varphi(U_t(K'))} \qloglik^{*}(\spars, \epsilon, u) -  \sup_{(\spars, \epsilon, u) \in \mathbb R_+\times (0,1) \times \mathbb R_+} \qloglik^{*} (\spars, \epsilon, u)  \leq - K \log \nedg,
\end{equation}
where
\begin{equation*}
  U_t(C) \coloneqq \Set*{\params \given%
    \textstyle%
    |\spars - \qMLEspars|^2 \leq \frac{C\log(\nedg)}{\nver},\ %
    |\varepsilon(\params) - \qMLEeps|^2 \leq \frac{C\log(\nedg)}{(\nedg)^{3/2}},\ %
    |u(\params) - \qMLEu|^2 \leq \frac{C \nver \log(\nedg)}{\nedgsq}%
  },
\end{equation*}
and that on $U_t$ the function
$\params \rightarrow \varphi(\params)=(\spars, \varepsilon(\params) , u(\params))$ is invertible. This
guarantees that
$\qMLEparams = \varphi^{-1}( (\qMLEspars, \tilde{\varepsilon}(\qMLEspars),\tilde{u}(\qMLEspars))) \in U_t$.
Hence the \ref{thm:concent:1:item:1}  in \cref{thm:concent:1:item:1} is proved.

We now prove points \ref{thm:concent:1:item:2} of \cref{thm:concent:1}. The starting point is the
\cref{eq:14}. Moreover, from part (ii) of \cref{pro:concent:16} the function
$\varphi(\params) = ( \spars, \varepsilon(\phi), u(\phi))$ has a continuously differentiable inverse
whose gradient is given by \eqref{eq:concent:35}. Hence by doing a Taylor
expansion of $\varphi^{-1}$, if $\params \in U_t$, we have,
\begin{equation*}
  |\tau - \qMLEtau |%
  \lesssim  \sqrt{2\nedg}|\varepsilon - \qMLEeps|%
    \lesssim \frac{  \sqrt{ \log \nedg} }{ (\nedg)^{1/4}},%
\end{equation*}
and,
\begin{equation*}
  | \size - \qMLEsize |
  \lesssim
    c_8 \sqrt{2\nedg}\log(2\nedg)|\spars- \qMLEspars|
    + c_7\frac{(2\nedg)^{3/2}}{\nver}|u(\params) - \qMLEu|%
    \lesssim \sqrt{2\nedg}  \frac{ \log \nedg }{ \sqrt{\nver}},
\end{equation*}

It remains to prove that \ref{thm:concent:1:item:3} of \cref{thm:concent:1} holds. But, from the expression of
$\qMLEspars, \qMLEeps, \qMLEu$, we obtain
\begin{equation}\label{qMLE:equiv}
  \qMLEtau%
  = \sqrt{2\nedg}\Big(\frac{\qMLEspars \nver}{\nedg}
  \Big)^{1/(1-\qMLEspars)}\Big(1 + O\Big(\frac{\nver}{\nedg}\bigvee
  \frac{\sqrt{\nedg}}{\nver} \Big)\Big),
  \quad
  \qMLEsize%
  = \frac{\qMLEspars \nver}{(\sqrt{2\nedg})^{\qMLEspars}}%
  \Big(1 + O\Big(\frac{\nver}{\nedg}\bigvee
  \frac{\sqrt{\nedg}}{\nver} \Big)\Big).
\end{equation}
which terminates the proof of \cref{thm:concent:1}.

\begin{lemma}
  \label{pro:concent:4}
  Under \cref{ass:concent:degdist}, for all $c_0 \in (0,1)$ fixed, we have:
  \begin{itemize}
\item  For all $C > 0$   there exists
  $B,t_0 > 0$ such that for all $\spars \leq -C$, and for all $t > t_0$,
  \begin{equation}
    \label{eq:concent:63}
    \Psi(\spars) - \Psi(c_0)%
    \leq -B\nver \log \frac{\nedg}{\nver}.
  \end{equation}
  \item  For all $K > 0$ there exists $c_2 > 0$
  \begin{equation}
    \label{pro:concent:9}
    \Psi(\spars) - \Psi(c_0)\leq -K \nver \quad \forall \spars \in (c_2,1).
  \end{equation}
  \end{itemize}
\end{lemma}

\begin{lemma}
  \label{pro:concent:5}
  Let $-C \leq \spars < c_2 < 1$. Then, if $\nver = o(\nedg)$, the equation
  $(\varepsilon,u) \mapsto H(\spars,\varepsilon,u)$ has a unique maximizer
  $(\tilde{\varepsilon}(\spars),\tilde{u}(\spars))$. Furthermore
  $\tilde{\varepsilon}(\spars) \leq \frac{3C\nver}{\nedg}$ ,
  $\tilde{u}(\spars) = \frac{\nver}{\nedg}\frac{1 +
    O(\nver/\nedg)}{f(\spars,\tilde{\varepsilon}(\spars))}$ and
         \begin{equation}
    \label{eq:concent:103}
    \Psi'(\spars)%
    = \sparsfunc'(\spars)%
    -\nver\Big(1+O\Big(\frac{\nver}{\nedg}\Big)\Big)\frac{\partial_{\spars}f(\spars,\tilde{\varepsilon}(\spars))}{f(\spars,\tilde{\varepsilon}(\spars))}%
    + O\Big(\frac{\nver^2}{\nedg}\Big), \quad \spars \in [-C,c_2].
  \end{equation}
In addition, under \cref{ass:concent:degdist}, there exists $0 < c_1 < c_2$
  such that $\Psi'(\spars) > 0$ for all $- C \leq \spars \leq c_1$, and for
  $c_2$ sufficiently large $\Psi'(c_2) < 0$.
\end{lemma}

\begin{lemma}
  \label{pro:concent:11}
  For all $\spars \in [c_1,c_2]$,
  $$
    \Psi''(\spars)%
    =\sparsfunc''(\spars) - \frac{\nver(1+o(1))}{\spars^2} < 0.%
$$
\end{lemma}

\begin{lemma}
  \label{pro:concent:16}
    Under \cref{ass:concent:degdist}, we have :
    \begin{enumerate}
    \item\label{item:pro:concent:16:1} For all $K>0$ there exists $K'>0$ such that
    \begin{equation*}
     \sup_{(\spars, \epsilon, u) \in \varphi(U_t(K'))^c} \qloglik^{*} (\spars, \epsilon, u)- \sup\qloglik^{*} \leq -K\log \nedg.
   \end{equation*}
 \item\label{item:pro:concent:16:2}  The map $\varphi$ is invertible on
  $U_t(C)$, for any $C>0$ if $\TrueSize$ is large enough. Consequently, there is a unique $\qMLEparams =
  (\qMLEspars,\qMLEtau,\qMLEsize) \in U_t$, and thus in $\mathcal{S}$ which
  maximizes $\qloglik$ and the Jacobian is given by
    \begin{equation}
     \label{eq:concent:35}
     J(\params)^{-1}
     \sim
     \begin{pmatrix}
       1 & 0 & 0\\
      0  &\sqrt{2\nedg}        &  0 \\
       c_8 \sqrt{2\nedg}  \log( 2 \nedg)
       & 0
       & c_7 \frac{(2\nedg)^{3/2}}{ \nver}
     \end{pmatrix}.
   \end{equation}
  \end{enumerate}
\end{lemma}

\cref{pro:concent:4,pro:concent:5,pro:concent:11,pro:concent:16} are proved in \cite[\cref{supp-sec:concent:case-where-spars}]{caron:naulet:rousseau:supplement}.

\subsection{Proof of \texorpdfstring{\cref{posterior:cons:1}}{Theorem \ref{posterior:cons:1}}}
\label{sec:posterior:cons:1}

Set $\phi_\ast = (\alpha_0, \KLtau, \KLsize)$ with $\KLsize = \sqrt{2\nedg}\tau_\ast^{1-\alpha_0}/2$
and $\KLtau$ as defined in \cref{thm:concent:1}, and define
\begin{equation*}
  B_t(M)%
  = \Set*{  \params \given  \left| \frac{\size}{\KLsize}  - \frac{\qMLEsize}{\KLsize} \right|\leq \frac{ M }{  \KLsize^{(1 + \qMLEsize)/2} }; \, |\tau - \qMLEtau| \leq \frac{ M }{ \sqrt{\KLsize} }; \, | \spars - \qMLEspars| \leq  \frac{M  }{\KLsize^{(1 + \qMLEsize)/2}} }.
\end{equation*}
Then under \cref{ass:concent:degdist}, for all $\epsilon>0$,
$ |\qMLEspars - \alpha_0| +|\qMLEtau - \KLtau| + |\qMLEsize/\KLsize - 1 | \leq \epsilon/2$
if $t$ is large enough.
Hence for all $M>0$, when $t$ is large enough,
$$\Set*{ |\spars - \alpha_0| +|\tau - \KLtau| + \Big|\frac{\size}{\KLsize} - 1 \Big| > \epsilon } \subset B_t(M\log \KLsize)^c .$$
Moreover, using \cref{lem:3}
\begin{align} \label{PiBtc}
\Pi ( B_t(M\log^2 \KLsize)^c \mid \mathcal G_t) & = \frac{ \int_{ B_t(M\log^2 \KLsize)^c } e^{L_t(\params)-L_t(\qMLEparams)} d\pi(\phi) }{ \int_{\mathcal S} e^{L_t(\phi)-L_t(\qMLEparams)} d\pi(\phi) } \nonumber\\
 & \leq
\frac{2 \int_{B_t(M\log^2 \KLsize)^c} e^{\qloglik(\params)-\qloglik(\qMLEparams)} d\pi(\phi) }{ \int_{B_t(1/\log \KLsize)} e^{\qloglik(\params)-\qloglik(\qMLEparams)} d\pi(\phi) }
\end{align}
Using \cite[\cref{supp-lem:Sigmat}]{caron:naulet:rousseau:supplement}, on $B_t(1/\log \KLsize) $,
  $\Sigma_t (\spars, \tau, \size/\KLsize)^{-1}  \lesssim \text{diag}\left( \KLsize^{1 + \alpha_0}\log^2 \KLsize,  \KLsize, \KLsize^{1 + \alpha_0} \right)$
so by a Taylor expansion, for all $\phi \in B_t(1/\log \KLsize)$ and since $\nabla \qloglik(\qMLEparams)=0$
\begin{align*}
  &\qloglik(\params)-\qloglik(\qMLEparams)\\
  &\quad\gtrsim - (\spars - \qMLEspars)^2\KLsize^{1 + \alpha_0} \log^2 \KLsize- (\tau - \qMLEtau )^2\KLsize - \left(\frac{ \size - \qMLEsize}{\KLsize}\right)^2 \KLsize^{1+\alpha_0} \gtrsim - C
\end{align*}
 for some $C > 0$. Thus the denominator of the right hand side of \cref{PiBtc} is bounded from below by $e^{-C}\Pi(B_t(1/\log \KLsize)$. We use  part 2 of \cref{thm:concent:1}  to bound the numerator  by $(\nedg)^{-M'}$ with  $M'>0$ arbitrarily large by choosing $M>0$ large enough and using \cref{cond:Pis},
\begin{align*}
  \Pi \big( B_t(M\log^2 \KLsize)^c \mid \mathcal G_t\big)%
  &\leq \frac{e^C (\nedg)^{-M'} }{ \Pi(B_t(1/\log \KLsize))} = o(1)
\end{align*}

We now prove \eqref{eq:BvM}. On $ B_t(M\sqrt{\log \KLsize})$, using the notation
from \cite[\cref{supp-sec:Sigmat}]{caron:naulet:rousseau:supplement} in particular $\qloglik^*(\spars, \tau, u ) = \qloglik(\spars, \tau,\KLsize u)$ and $\params_u = (\spars, \tau, u )$, and by a Taylor expansion,
\begin{align*}
  L_t(\params)-L_t(\qMLEparams)
  &=\qloglik(\params)-\qloglik(\qMLEparams) + o(1) \eqqcolon Q_t^*(\params_u)-Q_t^*(\hat\params_{\TrueSize,u})+o(1)\\
  &=\frac{ -(\params_u - \hat\params_{\TrueSize,u})^T \Sigma_t( \bar{\params}_u)^{-1}(\params_u - \hat\params_{\TrueSize,u})}{2} + o(1),
\end{align*}
where $\bar{\phi}_u = \gamma \phi_u + (1-\gamma)\hat{\phi}_{\TrueSize,u}$ for some $\gamma \in (0,1)$.
From \cite[\cref{supp-lem:Sigmat}]{caron:naulet:rousseau:supplement}, the matrix $\Sigma_t(\bar{\params}_u)^{-1}$ is positive
definite and satisfies
$\Sigma_t(\bar{\params}_u)^{-1} = \Sigma_t(\hat\params_{\TrueSize,u})^{-1} (1 + o(1)),$
so that on $ B_t(M\log^2 \KLsize)$%
$$
L_t(\params)-L_t(\qMLEparams)  =\frac{ -(\params_u - \hat\params_{\TrueSize,u})^T \Sigma_t(\hat\params_{\TrueSize,u})^{-1}(\params_u - \hat\params_{\TrueSize,u})}{2}(1 +o(1)).
$$
Now  denote $\tilde \pi_u$  (resp. $\tilde \pi_u( \params \mid \mathcal G_t)$) the prior  (resp. posterior) density of $\params_u$,
$\tilde \pi_u( \params_u \mid \mathcal G_t) = \KLsize\pi_{\TrueSize}( \params\mid \mathcal G_t) $ with $pi_{\TrueSize}( \params\mid \mathcal G_t) $ the posterior density of $\phi$; and  keeping
the notation $B_t(M\log^2 \KLsize)$ for its transformation by the function
$\params \mapsto \params_u$,%
\begin{multline*}
  \int_{B_t(M\log^2 \KLsize)} e^{\frac{ -(\params_u - \hat\params_{\TrueSize,u})^T \Sigma_t(\hat\params_{\TrueSize,u})^{-1}(\params_u - \hat\params_{\TrueSize,u})}{2}(1 +o(1))} \tilde \pi_u(\params_u) d\params_u \\
  = \tilde \pi_u(\hat\params_{\TrueSize,u}))(1+o(1))\int_{B_t(M\log^2 \KLsize)} e^{\frac{ -(\params_u - \hat\params_{\TrueSize,u})^T \Sigma_t(\hat\params_{\TrueSize,u})^{-1}(\params_u - \hat\params_{\TrueSize,u})}{2}}d\params_u.
\end{multline*}
Moreover using the lower bound on $\Sigma_t(\hat\params_{\TrueSize,u})^{-1}$ from \cite[\cref{supp-sec:Sigmat}]{caron:naulet:rousseau:supplement}
$$ (\params_u - \hat\params_{\TrueSize,u})^T \Sigma_t(\hat\params_{\TrueSize,u})^{-1}(\params_u - \hat\params_{\TrueSize,u}) \gtrsim ( \spars - \qMLEspars)^2 \KLsize^{1 + \alpha_0} + (\tau - \qMLEtau)^2 \KLsize + (\KLsize - \qMLEsize)^2 \KLsize^{\alpha_0-1} /\log^2 \KLsize .$$
Hence there exists $c>0$ such that writing $\tilde B_t = \{ \Sigma_t(\hat\params_{\TrueSize,u})^{-1/2}(\params_u - \hat\params_{\TrueSize,u}); \, \params_u \in  B_t(M\log^2 \KLsize)\}$,
\begin{multline*}
  \int_{B_t(M\log^2 \KLsize)} e^{\frac{ -(\params_u - \hat\params_{\TrueSize,u})^T \Sigma_t(\hat\params_{\TrueSize,u})^{-1}(\params_u - \hat\params_{\TrueSize,u})}{2}(1 +o(1))} \tilde \pi_u(\params_u) d\params_u \\
  =(1 +o(1))\tilde \pi_u(\hat\params_{\TrueSize,u})) \int_{B_t(M\log \KLsize)} e^{\frac{ -(\params_u - \hat\params_{\TrueSize,u})^T \Sigma_t(\hat\params_{\TrueSize,u})^{-1}(\params_u - \hat\params_{\TrueSize,u})}{2}(1 +o(1))} d\params_u \\
=  (1 +o(1))\tilde \pi_u(\hat\params_{\TrueSize,u})) | \Sigma_t(\hat\params_{\TrueSize,u}) |^{1/2} \left[ \int_{\mathbb R^3} e^{\frac{ -\|z\|^2}{2}} dz - \int_{\tilde B_t^c} e^{\frac{ -\|z\|^2}{2}}  dz\right] \\
\end{multline*}
  and using $ \int_{\tilde B_t^c} e^{\frac{ -\|z\|^2}{2}}  dz \leq (2\pi)^{3/2} Pr ( \|\mathcal X^2(3)\|^2 > M  ) =o(1) $ by choosing $M$ large, we obtain that
\begin{align*}
\tilde \pi_u( \params_u \mid \mathcal G_t) &= \frac{ (1 + o(1) )\tilde \pi_u(\params_u)  e^{\frac{ -(\params_u - \hat\params_{\TrueSize,u})^T \Sigma_t(\hat\params_{\TrueSize,u})^{-1}(\params_u - \hat\params_{\TrueSize,u})}{2}(1 +o(1))} }{ \int_{B_t(M\log^2 \KLsize)} e^{\frac{ -(\params_u - \hat\params_{\TrueSize,u})^T \Sigma_t(\hat\params_{\TrueSize,u})^{-1}(\params_u - \hat\params_{\TrueSize,u})}{2}(1 +o(1))} \tilde \pi_u(\params_u) d\params_u } \\
& = \frac{ (1 + o(1) )\tilde \pi_u(\KLparams)  e^{\frac{ -(\params_u - \hat\params_{\TrueSize,u})^T \Sigma_t(\hat\params_{\TrueSize,u})^{-1}(\params_u - \hat\params_{\TrueSize,u})}{2}(1 +o(1))} }{ \tilde \pi_u(\hat\params_{\TrueSize,u})|\Sigma_t(\hat\params_{\TrueSize,u})|^{1/2}(2\pi)^{3/2} } \\
& = \frac{ (1 + o(1) )  e^{\frac{ -(\params_u - \hat\params_{\TrueSize,u})^T \Sigma_t(\hat\params_{\TrueSize,u})^{-1}(\params_u - \hat\params_{\TrueSize,u})}{2}(1 +o(1))} }{ | \Sigma_t(\hat\params_{\TrueSize,u}) |^{1/2} (2\pi)^{3/2} },
\end{align*}
uniformly  over $\params \in  B_t(M\log^2 \KLsize)$, where we have used the consistency of $\hat\params_{\TrueSize,u}$. This implies that
\begin{align*}
\| \Pi_t - \mathcal{N}(0,1)^{\otimes 3}\|_{TV} \leq o(1) + \Pi ( B_t(M \log^2 \KLsize)^c  \mid \mathcal G_t) = o(1) .
\end{align*}

\subsection{Proof of \texorpdfstring{\cref{posterior:cons:2}}{Theorem \ref{posterior:cons:2}}}
\label{sec:posterior:cons:2}

Following from the proof of \cref{pro:concent:9}, for any $\epsilon>0$ fixed and $\spars >\epsilon$,
$$ \Psi(\spars) - \Psi(\epsilon) \leq \sum_{j\geq 2} \ndeg{j} \sum_{k=1}^{j-1}\log \frac{ k-\spars }{ k-\epsilon } + O(N_t)  \leq -(\spars - \epsilon)  \sum_{j\geq 2} \ndeg{j} \sum_{k=1}^{j-1} \frac{ 1 }{ k-\epsilon }+ O(N_t)  $$
Moreover
$$ \sum_{k=1}^{j-1} \frac{ 1 }{ k-\epsilon } \geq  \sum_{k=1}^{j-1} \frac{ 1 }{ k}  \geq \log j + O(1) $$
Hence, when $t$ is large enough and under \cref{ass:dense},
$$ \Psi(\spars) - \Psi(\epsilon) \leq -(\spars - \epsilon)  \sum_{j\geq 2} \ndeg{j} \log j+ O(N_t) \leq - c_2 (\spars - \epsilon) \nver \log \nedg + O(\nver) .$$
which thus leads to : for all $\params$ such that $\spars > 2 \epsilon$
\begin{equation*}
\loglik(\spars, \tau, \size) - \sup_{\tau, \size}\loglik(\epsilon, \tau, \size)\leq - c_2 (\spars - \epsilon) \nver \log \nedg +O(\nver)
\end{equation*}

Using the computations of the derivatives of $\qloglik$ in the proof of \cite[\cref{supp-lem:Sigmat}]{caron:naulet:rousseau:supplement}, $\sup_{\tau, \size}\loglik(\epsilon, \sigma, \tau)$ is attained at
$$\hat \tau_\epsilon = \left( \frac{ 2 \hat \size_\epsilon }{ \sqrt{2\nedg}} \right)^{1/(1-\epsilon)}(1 + o(1)) , \quad \hat \size_\epsilon = \frac{ \nver \epsilon }{ (1 -\epsilon/2) (2\nedg)^{\epsilon/2}}(1 + o(1))$$
and choosing
\begin{equation*}
\params \in B_{\TrueSize} = \left\{ | \spars - \epsilon| + |\tau - \hat \tau_\epsilon| \leq \frac{ \epsilon }{ \nver \log \nedg }, \,  |\size - \hat\size_\epsilon | \leq \frac{ \epsilon}{ \nver} \right\}
\end{equation*}
together with  $\nabla \qloglik$ in \cite[\cref{supp-sec:Sigmat}]{caron:naulet:rousseau:supplement},
$$\left| \loglik(\spars, \tau, \size) - \sup_{\tau, \size}\loglik(\epsilon, \tau, \size) \right|  \leq C\epsilon $$
for some positive constant  $C$. Hence
$$ \int_{B_{\TrueSize}}e^{\loglik(\spars, \tau, \size) - \Psi(\epsilon) }\pi(\params)d\params \geq e^{-\epsilon C} \Pi(B_{\TrueSize})$$
and from \cref{cond:Pis2},
\begin{align*}
  \Pi(B_{\TrueSize})
  &\geq \pi( \epsilon, \hat \tau_\epsilon)\left(\frac{ \epsilon }{ \nver \log \nedg } \right)^2(1+o(1)) \Pi\left( \left|\frac{\size}{\sqrt{2\nedg}}  - \frac{ \nver \epsilon }{ (1 -\epsilon/2) (2\nedg)^{(1+\epsilon)/2}}(1 + o(1)) \right| \leq \frac{ \epsilon}{ \nver} \right) \\
  &\geq e^{-\frac{ c_2 \nver}{2}},
\end{align*}
since if $\epsilon$ is small enough and when $\TrueSize$ is large enough
$$ \frac{ 1 }{(\nedg)^{c_3}} \leq  \frac{ \nver \epsilon }{ (1 -\epsilon/2) (2\nedg)^{(1+\epsilon)/2}}  \leq 1, \quad \text{and } \quad o(1) = \hat \tau_\epsilon \geq \frac{ \epsilon }{ 2} \frac{ 1 }{ (2\nedg)^{\epsilon)/2}}.$$
We finally obtain
$$\Pi( \sigma > 2 \epsilon \mid \mathcal G_{\TrueSize}) \leq e^{ - \frac{c_2 (\spars - \epsilon) \nver \log \nedg }{2} }=o(1).$$

\subsection{Proof of \texorpdfstring{\cref{thm:well-specified}}{Theorem \ref{thm:well-specified}}}
\label{sec:proof-crefthm:w-spec}

Since the GGP model of \cite{Caron2017} is a special case of multigraphex
processes, we establish the concentration result in the well-specified case by
using the concentration result for more general multigraphex processes of
\cref{thm:2}. To do so, in a first time we check that the GGP model satisfies
all the \cref{assumpt:1,assumpt:2,ass:Wm,ass:tails} and we characterize what are
the constants involved in those assumptions in term of the true parameters. In
particular we show that the GGP model satisfies the \cref{ass:concent:degdist}
by the general result of \cref{thm:2}, with $\alpha_0 = \sigma_0$, $\tau_{*} = \tau_0$ and
$\tau_{*}^{1-\alpha_0} = 2$. Hence, by the \cref{thm:concent:1} $(\qMLEspars,\qMLEtau)$
is a consistent estimator of $(\sigma_0,\tau_0)$ and $\qMLEsize/t \to 1$ in probability as well. In a
second time, we leverage the well-specification of the model to obtain the exact
rates of convergence via adaptation of standard techniques involving asymptotic
expansion of the log-likelihood.

\paragraph*{The GGP model satisfies \cref{assumpt:1,assumpt:2,ass:Wm,ass:tails}}

Recall that we consider only $\sigma_0 \in (0,1)$. From the definition of the GGP model
(see \cref{sec:multigraphex}), we obtain that the model is a multigraphex
process with marginal graphon function
\begin{equation}
  \label{eq:26}
  \mu(x)%
  = \int_{\NNReals}\Big(1 - e^{-2
    \bar{\rho}^{-1}(x)\bar{\rho}^{-1}(y)}\Big)\intd y,
\end{equation}
where $\bar{\rho}$ is the tail of the Lévy intensity defined in
\cref{eq:GGLevyMeasure}. It has been established in
\cite[Proposition~16]{caron:rousseau:18} that $\mu$ in \cref{eq:26} satisfies
the \cref{assumpt:1,assumpt:2} with $\alpha_0 = \sigma_0$ and $a=1$. By
\cref{eq:poissonmultigraphonrank1}, we also have that when $x\ne y$,
\begin{align*}
  W(x,y)(1 - W(x,y))%
  &= \Big(1 - e^{-2
    \bar{\rho}^{-1}(x)\bar{\rho}^{-1}(y)}\Big) e^{-2
    \bar{\rho}^{-1}(x)\bar{\rho}^{-1}(y)}\\
  &\leq 2\bar{\rho}^{-1}(x,y)\bar{\rho}^{-1}(y) e^{-2
    \bar{\rho}^{-1}(x)\bar{\rho}^{-1}(y)}= W_m(x,y,1).
\end{align*}
A similar computation in the case $x=y$ gives that the GGP model also satisfies
the \cref{ass:Wm}. It remains to establish the validity of \cref{ass:tails}. The
proof is very similar to \cite[Proposition~16]{caron:rousseau:18} but with extra
cares to get the second order terms. They in particular show that
$\mu^{-1}(x) = \bar{\rho}(\psi^{-1}(x)/2)$ where
$\psi(t) \coloneqq \int_0^{\infty}(1 - e^{-wt})\rho(\intd w)$. Then,
$\psi(t) = t \int_0^{\infty} w \rho(\intd w) + O(t^2)$ as $t\to 0$, namely
$\psi(t) = t \cdot \tau_0^{-1+\sigma_0} + O(t^2)$. We deduce that as $t\to 0$,
\begin{equation}
  \label{eq:66}
  \psi^{-1}(t) = t \cdot \tau_0^{1-\sigma_0} + O(t^2).
\end{equation}
On the other hand,
\begin{align*}
  \bar{\rho}(x)%
  &= \int_x^{\infty}\frac{w^{-1-\sigma_0}e^{-\tau_0 w}\intd w}{\Gamma(1-\sigma_0)}= \frac{x^{-\sigma_0}}{\Gamma(1-\sigma_0)}\Big\{\frac{1}{\sigma_0} -
    \int_1^{\infty}w^{-1-\sigma_0}(1 - e^{-\tau_0x w})\intd w \Big\}.
\end{align*}
But, for $x\to 0$,
\begin{align*}
  0
  &\leq \int_1^{\infty}w^{-1-\sigma_0}(1 - e^{-\tau_0x w})\intd w \leq \int_{1}^{1/x} x \tau_0 w^{-\sigma_0}\intd w%
    + \int_{1/x}^{\infty} w^{-1-\sigma_0}\intd w\\
  &\leq \frac{\tau_0x^{\sigma_0}}{1 - \sigma_0}%
    + \frac{x^{\sigma_0}}{\sigma_0}.
\end{align*}
Therefore, as $x \to 0$,
\begin{align}
  \label{eq:50}
  \bar{\rho}(x)%
  &= \frac{x^{-\sigma_0}}{\sigma_0\Gamma(1-\sigma_0)} + O(1).
\end{align}
Combining \cref{eq:66,eq:50}, we find as $x \to 0$,
\begin{equation*}
  \mu^{-1}(x) - \frac{2^{\sigma_0}}{\sigma_0
    \tau_0^{\sigma_0(1-\sigma_0)}\Gamma(1-\sigma_0)}x^{-\sigma_0} = O(1).
\end{equation*}
Then \cref{ass:tails} is satisfied by the GGP model with $\alpha_0 = \sigma_0$,
$\beta = \sigma_0 > 0$ and
\begin{equation}
  \label{eq:68}
  c_0%
  = \frac{2^{\sigma_0}}{\sigma_0
    \tau_0^{\sigma_0(1-\sigma_0)}\Gamma(1-\sigma_0)}.
\end{equation}
We note that the previous also implies that \cref{assumpt:1} is satisfied. We
now show that the previous implies $\tau_{*} = \tau_0$. The expression for
$\tau_{*}$ when $(\mathcal{G}_t)_{t\geq 0}$ is a multigraphex process is given
in the \cref{thm:2}. The only term that remains to compute is
\begin{equation}
\label{eq:69}
  \int_{\NNReals}\bar{W}_1(x,y)\intd x \intd y%
  \notag
 = 2 \int_{\NNReals^2}\bar{\rho}^{-1}(x)\bar{\rho}^{-1}(y) \intd x \intd y = 2 \Big\{ \int_{\NNReals} \frac{w^{-\sigma_0}e^{-\tau_0 w}\intd
    w}{\Gamma(1-\sigma_0)} \Big\}^2  = 2 \tau_0^{-2+2\sigma_0}.
\end{equation}

\paragraph*{Rates of convergence via asymptotic expansion of the log-likelihood}

We start by getting a slow, but polynomial, rate of convergence. Indeed, we have
from \cref{eq:16}, $|\qMLEspars - \sigma_0| = O_p(t^{-\eta})$ for some $\eta>0$. By
\cref{eq:68}, \cref{eq:69} and \cref{thm:2},
\begin{align*}
  \tau_{*}%
  &= \Bigg\{ \frac{2^{1+\sigma_0} \tau_0^{-\sigma_0(1-\sigma_0)} }{\big(4
    \tau_0^{-2+2\sigma_0} \big)^{(1+\sigma_0)/2} } \Bigg\}^{1/(1-\sigma_0)} = \tau_0.
\end{align*}
and  using \cite[\Cref{supp-sec:asympt-prop-mult}]{caron:naulet:rousseau:supplement}, for some $\eta>0$,  $$ \left| \sqrt{ 2\nedg} \left( \frac{ \sigma_0 \nver }{ \nedg } \right)^{1/(1-\sigma_0)} - \tau_0 \right| = O_p(  t^{-\eta} ).$$
Combining the above with  \cref{qMLE:equiv}, we obtain that
\begin{equation}\label{qMLE-well}
|\qMLEtau - \tau_0| = O_p(t^{-\eta}), \quad \left| \qMLEsize - t\frac{\sqrt{2\int_{\NNReals^2}\bar{W}_1(x,y)\intd x\intd y}
    \tau_0^{1-\sigma_0}}{2}\right| = \left| \qMLEsize - t\right| =O_p(t^{1-\eta}).
\end{equation}

We now refine these rates. The \cref{lem:bias:score} below controls the bias of
the score based on $\mathcal Q_t^{*} (\phi_{0,u})$ with $\phi_{0,u}= (\sigma_0, \tau_0, 1)$
and implies that
\begin{align}\label{mle:eq}
\nabla_{\phi_{u}} \mathcal Q^{*}_t(\hat \phi_{t,u} )= 0 = \nabla_{\phi_{u}} \mathcal Q^{*}_t(\phi_{0,u} )+\int_0^1 D_{\phi_{u}}^2\mathcal Q^{*}_t(v\hat \phi_{t,u} +(1-v) \phi_{0,u}) (\hat \phi_{t,u} -  \phi_{0,u})dv
\end{align}
and
$$-D_{\phi_{u}}^2\mathcal Q^{*}_t(v\hat \phi_{t,u} +(1-v) \phi_{0,u}) = \Sigma(v\hat \phi_{t,u} +(1-v) \phi_{0,u})^{-1} $$
which is definite positive from \cite[\cref{supp-sec:Sigmat}]{caron:naulet:rousseau:supplement}. 
We then write $\bar \Sigma^{-1} = \int_0^1 \Sigma(v\hat \phi_{t,u} +(1-v) \phi_{0,u})^{-1} dv$, so
\begin{align*}
 \hat \phi_{t,u} -  \phi_{0,u} = \bar \Sigma \nabla_{\phi_{u}} \mathcal Q^{*}_t(\phi_{0,u} )
\end{align*}
Since $|\hat \phi_{t,u} - \phi_{0,u}| = O(t^{-\delta})$, we have uniformly for $v$ in $(0,1)$
$$\Sigma(v\hat \phi_{t,u} +(1-v) \phi_{0,u})^{-1} = \Sigma(\phi_{0,u})^{-1} (1 +o_p(1))$$
and $\bar \Sigma = \Sigma(\phi_{0,u})(1 +o_p(1))$. Hence \cref{mle:eq} becomes
\begin{align*}
\Sigma_t(\phi_{0,u})^{-1/2} &(\hat \phi_{t,u} -  \phi_{0,u}) =(1 +o(1))\Sigma_t(\phi_{0,u})^{1/2} \nabla_{\phi_{u}} \mathcal Q^{*}_t(\phi_{0,u} )  \\
                    &=  (1 +o(1))\Sigma_t(\phi_{0,u})^{1/2}  \nabla_{\phi_{u}} L_t(\phi_{0,u} )\\
  &\quad+ (1 +o(1))\Sigma_t(\phi_{0,u})^{1/2} [\nabla_{\phi_{u}} \mathcal Q^{*}_t(\phi_{0,u} ) - \nabla_{\phi_{u}} L_t(\phi_{0,u} ) ]\\
& = (1 +o(1))\Sigma_t(\phi_{0,u})^{1/2}  \nabla_{\phi_{u}} L_t(\phi_{0,u} )  + o(1)
\end{align*}
where the term $o(1)$ comes from \cref{lem:bias:score}.
From the beginning of \cite[\cref{supp-sec:Sigmat}]{caron:naulet:rousseau:supplement},
$$ \Sigma_t(\phi_{0,u})^{-1} \geq
\left( \begin{array}{ccc}
c t^{1+ \sigma_0} + M_1(\params_0) & 0 & M_2(\params_0) \\
0 & c t & 0 \\
M_2(\params_0) & 0 &  \nver  \end{array}\right) ( 1+ o(1)) $$
so that
$$ \Sigma_t(\phi_{0,u}) \leq \frac{ c' }{ t }
\left( \begin{array}{ccccc}
 c_1t^{-\sigma_0}&  & 0 & & c_2t^{-\sigma_0}\log t \\
0 & & 1 & & 0 \\
t^{-\sigma_0}\log t &  &0 & & t^{-\sigma_0} \log^2 t   \end{array}\right) ( 1+ o(1)),  $$
for some $c', c_1, c_2>0$.
Moreover
\begin{align*}
  &\mathbb P( \|\Sigma_t(\phi_{0,u})^{1/2}  \nabla_{\phi_{u}} L_t(\phi_{0,u} )  \|> M )\leq
    \frac{ \Tr\left[\Sigma_t(\phi_{0,u}) \mathbb E[-D^2{\phi_{u}} L_t(\phi_{0,u} )]\right]}{ M }\\
  &\qquad\qquad\leq  \frac{ \Tr\left[\Sigma_t(\phi_{0,u}) \mathbb E[-D^2{\phi_{u}} \mathcal Q^{*}_t(\phi_{0,u} )]\right]}{ M }
    +\frac{ \Tr\left[\Sigma_t(\params_0) \mathbb E[-D^2\mathcal Q^{*}_t(\phi_{0,u} )] \right]}{ M } +o(1) \lesssim \frac{ 1}{M}
\end{align*}
using \cref{lem:bias:score}, so that
$
\Sigma_t(\phi_{0,u})^{-1/2} (\hat \phi_{t,u} -  \phi_{0,u}) = O_p(1).
$

 \begin{lemma}\label{lem:bias:score}
 There exists $\delta_0>0$ such that if $V_2 = \mathrm{diag}( t^{-(1+\spars_0)/2} , t^{-1/2},  t^{-(1+\spars_0)/2} \log t )$,
$$\mathbb P_0\left(\left\|V_2 \nabla_{\phi_u} (\mathcal Q^{*}_t-L_t) (\spars_0, \tau_0, 1) \right\| > t^{-\delta_0} \right)= o(1).$$
and
$$\Tr\left[ V_2D^2_{\phi_u} (\mathcal Q^{*}_t-L_t) (\spars_0, \tau_0, 1)\right] = O(1) .$$
\end{lemma}

The proof of  \cref{lem:bias:score} is given in \cite[\Cref{supp-sec:pr:lem:bias:score}]{caron:naulet:rousseau:supplement}.

\subsection{Proof of \texorpdfstring{\cref{lemma:Ntilde}}{Proposition \ref{lemma:Ntilde}}}
\label{sec:proof-crefl}

The family of random multigraphs $(\widetilde{\mathcal G}_t)$ is a multigraphex process with multigraphon function
\begin{align*}
\widetilde W_m(x,y,0)&=W_m(x,y,0)+W_m(x,y,1), \quad \widetilde W_m(x,y,1)=0\\
\widetilde W_m(x,y,k)&= W_m(x,y,k)\text{ for all }k\geq 2.
\end{align*}
The associated simple graph $\widetilde{\mathcal G}^{(s)}_t$ obtained by discarding multiple edges from $\widetilde{\mathcal G}_t$ is therefore a graphon process with graphon function
$\widetilde W(x,y)=W(x,y)-W_m(x,y,1).$
Denote
$\tmu(x)=\int_0^\infty \widetilde W(x,y)dy.$
If $\int\tmu(x)dx=0$, then $\widetilde N_t=0$ almost surely. We now assume $\int\tmu(x)dx>0$. Using~\cite{Veitch2015}, the expectation of the number of nodes $\widetilde N_t$ in the simple subgraph $\widetilde {\mathcal G}^{(s)}_t$ (which is the same as the number of nodes in $\widetilde {\mathcal G}_t$) is
\begin{align*}
\mathbb E[\widetilde N_t]&=t\int_0^\infty \left(1-(1-\tW(x,x))e^{-t\tmu(x)}\right)dx.
\end{align*}

Under \cref{ass:Wm}, $\widetilde W(x,y)\leq W(x,y)^2$ hence, under \cref{assumpt:2},
$
\widetilde \mu(x)\leq \ell_3(x)^2\mu(x)^2.
$
Under the \cref{assumpt:1} with $\alpha_0>0$, using Proposition 1.5.15 in \cite{Bingham1987}, we additionally have
\begin{align*}
\mu^{2}(x)\sim  x^{-2/\alpha_0} \widetilde \ell^*_1(x)\text{ as }x\to\infty .
\end{align*}
for some slowly varying function $\widetilde \ell^*_1$. We obtain, using Lemma S3.2 in~\cite{caron:rousseau:18} and assuming without loss of generality that $\ell_3$ is monotone decreasing and right-continuous
\begin{equation*}
\int_0^\infty(1-e^{-t\tmu(x)})dx\leq \int_0^\infty(1-e^{-t\ell_3(x)^2\mu(x)^2})dx= o(t^{\alpha_0/2+\delta})
\end{equation*}
for any $\delta>0$. By dominated convergence, as $\int \widetilde W(x,x)dx<\infty$, $\int \widetilde W(x,x) e^{-t\tmu(x)}dx=o(1)$. Hence, we obtain, for any $\delta>0$
$
\mathbb E[\widetilde N_t] =o( t^{1+\alpha_0/2+\delta}).
$
Using Markov inequality,
$$
\Pr\left (\frac{\widetilde N_t}{t^{1+\alpha_0/2+2\delta}}>\epsilon\right )=o(t^{-\delta}).
$$
Let $t_m=m^{2/\delta}$ for $m=1,2,\ldots$. By the Borel-Cantelli lemma, almost surely
$
\widetilde N_{t_m} =o(t_m^{1+\alpha_0/2+2\delta})
$. For any  $t_m<t \leq t_{m+1}$, noting that $\widetilde N_t$ is monotone increasing, we have
$$
\frac{\widetilde N_{t}}{t^{1+\alpha_0/2+2\delta}}\leq \frac{\widetilde N_{t_{m+1}}}{t_{m+1}^{1+\alpha_0/2+2\delta}}\left(\frac{t_{m+1}}{t_m}\right)^{1+\alpha_0/2+2\delta}=o(1)
$$
almost surely.

\subsection{Proof of \texorpdfstring{\cref{thm:2}}{Theorem \ref{thm:2}}}
\label{sec:proof-thm:2-1}

The fact that $\ndeg{1} \ne \nver$ is almost-surely true by
\cref{th:asympmultigraphex}. The same theorem implies that there is
$\tau_{*} > 0$ such that
$\lim_{t\to \infty} \sqrt{2\nedg}\big( \frac{\alpha_0\nver}{\nedg}
\big)^{1/(1-\alpha_0)} = \tau_{*}$ almost-surely. It also implies that
$\log \frac{\nedg}{\nver} \sim (1-\alpha_0)\log(t)$ almost-surely. We prove below that there exists $\eta>0$ such that
\begin{equation}
  \label{eq:16}
|\hat{\alpha}_t - \alpha_0| = O_p(t^{-\eta}),
\end{equation}
which then implies \cref{thm:2}.

where we call that $\hat{\alpha}_t$ is the (unique) solution of
$-\hat{\alpha}_t \sparsfunc'(\hat{\alpha}_t) = \nver$. We define the functions
$\hat{\Phi}(\alpha) \coloneqq \sparsfunc'(\alpha) + \alpha^{-1}\nver$, and
$\Phi(\alpha) \coloneqq \EE[\hat{\Phi}(\alpha)]$. Likewise, $\hat{\alpha}_t \in (0,1)$ is solution to
$\hat{\Phi}(\hat{\alpha}_t) = 0$. Similarly we let $\alpha_{*} \in (0,1)$ be solution to
$\Phi(\alpha_{*}) = 0$. We prove \cref{eq:16} by showing that $|\alpha_{*} - \alpha_0| \lesssim t^{-\eta}$
for some $\eta >0$ and $|\hat{\alpha}_t - \alpha_{*}| \lesssim t^{-\eta}$ with probability going to 1,
where $\eta > 0$ is defined in \cref{lem:2}.

\paragraph*{Bounding $\PP(|\hat{\alpha}_t - \alpha_{*}| > t^{-\eta/2})$}

Using a Taylor expansion of $\hat{\Phi}$  we find that there exists $\bar{\alpha}$
between $\hat{\alpha}_t$ and $\alpha_{*}$ such that
$\hat{\Phi}(\alpha_{*}) = \hat{\Phi}(\hat{\alpha}_t) +
\hat{\Phi}'(\bar{\alpha})(\alpha_{*} - \hat{\alpha}_t)$, whence
$\hat{\Phi}(\alpha_{*}) = \hat{\Phi}'(\bar{\alpha})(\alpha_{*} - \hat{\alpha}_t)$. But,
$\hat{\alpha}_t,\alpha_{*} \in [0,1]$, and so $\bar{\alpha} \in [0,1]$ as well. Since
$\hat{\Phi}'(\alpha) = - \sum_{j\geq 2}\ndeg{j}\sum_{k=1}^{j-1}\frac{1}{(k -
  \alpha)^2} - \alpha^{-2}\nver$, it follows that $|\Phi'(\bar{\alpha})| \geq \bar{\alpha}^{-2}\nver$, and hence
$\bar{\alpha}^{-2}\nver |\alpha_{*} - \hat{\alpha}_t| \leq
|\Phi'(\bar{\alpha})(\alpha_{*} - \hat{\alpha}_t)| = |\hat{\Phi}(\alpha_{*})|$. We deduce that
$$
|\hat{\alpha}_t - \alpha_{*}|
  \leq \frac{ \bar \alpha^2}{ \nver } \left( \left| \sparsfunc'(\alpha) - \EE(\sparsfunc'(\alpha) )\right| + \alpha_{*}^{-1} |\nver - \EE(\nver) | \right)
$$
so that for all $\delta>0$, when $t$ is large enough
\begin{align*}
\PP\left( |\hat{\alpha}_t - \alpha_{*}|> t^{-\eta + \delta} \nver \right)
 & \leq \PP\left( \nver < t^{1+\alpha_0-\delta/2} \right) +   \PP\left( |\hat{\alpha}_t - \alpha_{*}|>  t^{1+\alpha_0 -\eta+\delta/2 }  \right) \\
 & \leq \frac{ 4 \var (\nver ) }{ \EE(\nver)^2 } +t^{-\delta}  = O( t^{- \delta} )
\end{align*}
by choosing $\delta>0$ small enough and using Chebyshev inequality for the term $\nver$ and \cref{lem:2} for the second term.

\paragraph*{Bounding $|\alpha_{*} - \alpha_0|$}

Recall that $\Phi(\alpha) = \EE[\sparsfunc'(\alpha)] +
\alpha^{-1}\EE[\nver]$. Hence, by combining \cref{lem:2} with \cite[\cref{supp-lem:4}]{caron:naulet:rousseau:supplement}, we find that as $t \to \infty$, for some
$\eta > 0$,
\begin{equation}
  \label{eq:34}
  \Phi(\alpha)%
  = t \bar{F}(t^{-1})\Gamma(1-\alpha_0)\Big\{\Big(\frac{\alpha_0}{\alpha}-1
  \Big)\frac{\Gamma(1-\alpha)\Gamma(\alpha_0)}{\Gamma(1+\alpha_0 - \alpha)}  +
  O\big( t^{-\eta} \big)\Big\}.
\end{equation}
Remember that $\Phi(\alpha) = 0$ has a unique solution $\alpha_{*}$ (because
$\Phi$ is monotone), but from the previous, it is clear that asymptotically this
solution has to lie in an interval of the form
$[\alpha_0 - O(t^{-\eta}) ), \alpha_0 + O(t^{-\eta})]$. Indeed from
\cref{eq:34}, for $c > 0$ large enough, $\Phi(\alpha_0 + ct^{-\eta} ) < 0$,
while $\Phi(\alpha_0 - ct^{-\eta} ) > 0$. Since $\Phi$ is monotone, then
$\Phi(\alpha) < 0$ for all $\alpha \geq\alpha_0 + ct^{-\eta} $ and
$\Phi(\alpha) > 0$ for all $\alpha \leq \alpha_0 - ct^{-\eta} $. Consequently $|\alpha_{*} - \alpha_0| \leq c t^{-\eta}$.

\begin{lemma}
  \label{lem:2}
  Let \cref{assumpt:1,assumpt:2,ass:Wm,ass:tails} be satisfied, together with  $\int_0^\infty \mu(x)^2dx < \infty$.  Then, there exists $\eta > 0$, depending only on $\alpha_0$, $\gamma_0$ and
  $a$, such that as $t \to \infty$, for every fixed $\alpha \in (0,1)$,
  \begin{equation}
    \label{eq:170}
    \EE[-\sparsfunc'(\alpha)]%
    = t \mu^{-1}(t^{-1})\Gamma(1-\alpha_0)\Big(1 +
    O\big(t^{-\eta}\big) \Big)\Big\{ \frac{1}{\alpha} + \Big(1 -
    \frac{\alpha_0}{\alpha}\Big)\frac{\Gamma(1-\alpha)\Gamma(\alpha_0)}{\Gamma(1+\alpha_0
      - \alpha)}\Big\},
  \end{equation}
  and  for any $\delta >0$
  \begin{equation}
    \label{eq:232}
    \PP\left( |\sparsfunc'(\alpha) - \EE(\sparsfunc'(\alpha)) | > t^{ 1+\alpha_0 - \eta + \delta } \right)
       = O\Big(t^{-2\delta}\Big).
  \end{equation}
\end{lemma}

The proof of \cref{lem:2} can be found in \cite[\cref{supp-sec:proof-lem:2}]{caron:naulet:rousseau:supplement}.

%

\section*{Acknowledgements}
The project leading to this work has received funding from the European Research Council
(ERC) under the European Union’s Horizon 2020 research and innovation programme
(grant agreement No 834175).

  \putbib%
  \stopcontents[main]%
\end{bibunit}

\newpage

\setcounter{page}{1}
\setcounter{section}{0}
\setcounter{table}{0}
\setcounter{figure}{0}
\setcounter{theorem}{0}
\setcounter{proposition}{0}
\setcounter{lemma}{0}
\setcounter{equation}{0}

\renewcommand{\thepage}{S\arabic{page}}
\renewcommand{\thesection}{S\arabic{section}}
\renewcommand{\thetable}{S\arabic{table}}
\renewcommand{\thefigure}{S\arabic{figure}}
\renewcommand{\thetheorem}{S\arabic{section}.\arabic{theorem}}
\renewcommand{\theproposition}{S\arabic{section}.\arabic{proposition}}
\renewcommand{\thelemma}{S\arabic{section}.\arabic{lemma}}
\renewcommand{\theequation}{S\arabic{section}.\arabic{equation}}


\begin{bibunit}[imsart-number]
  \startcontents[supp]%

\begin{frontmatter}
\title{Asymptotic Analysis of Statistical Estimators related to MultiGraphex
  Processes under Misspecification : supplementary material}
\runtitle{Asymptotic Analysis of Statistical Estimators related to MultiGraphex Processes}

\begin{aug}
\author[A]{\fnms{Zacharie} \snm{Naulet}\ead[label=e1]{zacharie.naulet@universite-paris-saclay.fr}},
\author[B]{\fnms{Judith} \snm{Rousseau}\ead[label=e2,mark]{judith.rousseau@stats.ox.ac.uk}}
\and
\author[B]{\fnms{François} \snm{Caron}\ead[label=e3,mark]{caron@stats.ox.ac.uk}}
\address[A]{Université Paris-Saclay, Laboratoire de mathématiques d’Orsay, 91405, Orsay, France.
\printead{e1}}

\address[B]{University of Oxford, Department of Statistics, Oxford, UK.
\printead{e2,e3}}
\end{aug}

\begin{abstract}
   This document is supplementary material for the article. It contains the missing proofs. We refer to the main
document for all the definitions.
\end{abstract}

\begin{keyword}
\kwd{Bayesian Nonparametrics}
\kwd{Networks}
\kwd{Random Graphs}
\kwd{Sparsity}
\kwd{Caron and Fox model}
\kwd{Inference}
\kwd{Maximum Likelihood Estimation}
\kwd{Bayesian Estimation}
\kwd{Misspecification}
\end{keyword}

\end{frontmatter}

\section{Asymptotic expansion of the log-likelihood: auxiliary results}
\label{sec:proofs-relat-asympt}


\subsection{Proof of the expression of the likelihood in \texorpdfstring{\cite[\cref{main-eq:likelihood}]{caron:naulet:rousseau:main}}{equation \ref{main-eq:likelihood}}}
\label{sec:proof-crefl-like}

\cite[\Cref{main-eq:likelihood}]{caron:naulet:rousseau:main} is a consequence of \cite[Theorem~6]{Caron2017}. Using
exchangeability, we assume without loss of generality that the non isolated
vertices corresponds to the indices $1,\dots,\nver$. Then, from their theorem,
we find that
\begin{align*}
  e^{\loglik(\params)}%
  &\propto \size^{\nver} \int_{\Reals_+}\int_{\Reals_+^{\nver}}%
    \exp\Big\{- \Big(\sum_{i=1}^{\nver}w_i + w_{*}\Big)^2 \Big\}\Big\{
    \prod_{i=1}^{\nver}w_i^{\deg{i}}\rho(\intd w_i) \Big\}
    \tiltstab(\intd w_{*})\\
  &=\size^{\nver} \int_{\Reals_+}\int_{\Reals_+^{\nver}}%
    \exp\Big\{- \Big(\sum_{i=1}^{\nver}w_i + w_{*}\Big)^2 \Big\}\Big\{
    \prod_{i=1}^{\nver}\frac{w_i^{\deg{i}-1-\spars}e^{-\tau w_i} \intd
    w_i}{\Gamma(1-\spars)}\Big\} \tiltstab(\intd w_{*}).
\end{align*}
That is, introducing the PDF of the Gamma distribution,
\begin{multline*}
  e^{\loglik(\params)}%
  \propto%
  s^{\nver}\Big\{\prod_{i=1}^{\nver}\frac{\Gamma(\deg{i} -
    \spars)}{\Gamma(1-\spars)\tau^{\deg{i}-\spars}}\Big\}
    \int_{\Reals_+}\int_{\Reals_+^{\nver}}%
    \exp\Big\{- \Big(\sum_{i=1}^{\nver}w_i + w_{*}\Big)^2 \Big\}\\
    \times \Big\{%
    \prod_{i=1}^{\nver}\frac{\tau^{\deg{i} -
    \spars}w_i^{\deg{i}-1-\spars}e^{-\tau w_i} \intd
    w_i}{\Gamma(\deg{i}-\spars) }\Big\}\tiltstab(\intd w_{*}).
\end{multline*}
Remark that the inner integral is the expectation of
$\exp\{-(\sum_{i=1}^{\nver}W_i + w_{*})^2\}$ under
$(W_1,\dots,W_{\nver}) \sim \otimes_{i=1}^{\nver}\gammaDist(\deg{i} -
\spars, \tau)$. But under this distribution we have that
$\sum_{i=1}^{\nver}W_i$ has a
$\gammaDist(\sum_{i=1}^{\nver}\deg{i} - \spars \nver,\tau)$
distribution. Hence, the conclusion follows.

\subsection{Proof of \texorpdfstring{\cite[\cref{main-pro:4}]{caron:naulet:rousseau:main}}{Lemma \ref{main-pro:4}}}
\label{sec:proof:pro:4}

Using the definition of $I(\params)$, we introduce the PDF of the
$\gammaDist(\nedg - \spars \nver, \tau)$ distribution,
\begin{align*}
  I(\params)
  &=\frac{1}{\tau^{\nedg - \spars \nver}}
    \int_{\NNReals^2}e^{-(x +y)^2}%
    \frac%
    {\tau^{\nedg-\spars \nver} x^{\nedg - \spars \nver - 1}e^{-\tau x}}%
    {\Gamma(\nedg - \spars \nver)}%
    \tiltstab(\intd y) \intd x.
\end{align*}
We rewrite the previous integral by using that the Fourier transform of a
Gaussian is Gaussian. In particular,
\begin{equation*}
  \frac{1}{2\sqrt{\pi}}\int_{\Reals}\exp\Big\{ -\frac{\xi^2}{4} + i(x+y)\xi
  \Big\}\, \intd \xi%
  = \exp\{ -(x+y)^2 \}.
\end{equation*}
Inserting the last display in the previous expression for $I(\params)$, it
follows from Fubini's theorem and the expressions for the Fourier transforms
of $\tiltstab$ (see \cite[\cref{main-eq:4}]{caron:naulet:rousseau:main}) and the Gamma distribution,%
\begin{align*}
  I(\params)
  &= \frac{1}{\tau^{\nedg - \spars \nver}}\int_{\NNReals^2} \frac{1}{2
    \sqrt{\pi}}\int_{\Reals}e^{-\frac{\xi^2}{4} + i(x + y)\xi}\,\intd \xi%
    \frac%
    {\tau^{\nedg - \spars \nver} x^{\nedg - \spars \nver - 1}e^{-x}}%
    {\Gamma(\nedg - \spars \nver)}%
    \tiltstab(\intd y) \intd x\\
  &= \frac{1}{2 \sqrt{\pi}\tau^{\nedg - \spars\nver}}\int_{\Reals}
    e^{-\frac{\xi^2}{4}}%
    \Big\{%
    \int_{\NNReals}e^{ix\tau^{-1}\xi}\frac%
    {x^{\nedg - \spars \nver - 1}e^{-x}}%
    {\Gamma(\nedg - \spars \nver)}\,\intd x\Big\}%
    \Big\{ \int_{\NNReals}e^{i\xi y}\tiltstab(\intd y) \Big\} \,\intd \xi\\
  &=
    \frac{1}{2\sqrt{\pi}\tau^{\nedg - \spars\nver}}\int_{\Reals}\Big(1 -
    \frac{i \xi}{\tau}\Big)^{-(\nedg - \spars \nver)}  \exp\Big\{-
    \frac{\xi^2}{4} - \size \LTS(\spars,\tau;\tau - i \xi) \Big\}\,\intd \xi\\
  &= \frac{1}{2 \sqrt{\pi}}%
    \int_{\Reals}\big( \tau- i\xi)^{-(\nedg - \spars \nver)}%
    \exp\Big\{-
    \frac{\xi^2}{4} - \size \LTS(\spars,\tau;\tau  - i \xi ) \Big\}\,\intd
    \xi.
\end{align*}
The conclusion follows because for any $\xi \in \Reals$, from \cite[\cref{main-eq:2}]{caron:naulet:rousseau:main},
\begin{equation*}
  - \frac{\xi^2}{4}%
  =-\nedg \mathcal{A}(\params;\tau-i\xi)%
  + \size \LTS(\spars,\tau;\tau -i\xi)%
  + (\nedg + \spars\nver)\log(\tau-i\xi ).%
\end{equation*}
Now we establish \cite[\cref{main-eq:7}]{caron:naulet:rousseau:main}, the uniqueness of $\zeta(\params)$, and that
$\zeta(\params) > \tau$. Remark that for all $\spars$ (including
$\spars = 0$), we have from the definition of $\LTS$ that
$\partial_z\LTS(\spars,\tau;z) = 1/z^{1-\spars}$ for all
$z \in \Complex \backslash \Reals_-^{*}$. Then for all $\params$,
\begin{equation}
  \label{eq:82}
  \partial_z\mathcal{A}(\params;z)
  = - \frac{z}{2\nedg}%
  + \frac{\tau}{2\nedg}%
  + \Big(1 - \frac{\spars \nver}{\nedg}\Big) \frac{1}{z}%
  + \frac{s}{\nedg}z^{-1+\spars}.
\end{equation}
Therefore,
\begin{equation*}
  \partial_z\mathcal{A}(\params;z) = 0%
  \iff%
  z^2%
  = \tau z%
  + 2s z^{\spars} + 2\nedg \Big(1 - \frac{\spars\nver}{\nedg}\Big).
\end{equation*}
Since $\nedg \geq \nver$ and $\spars < 1$, the last display implies that any
solution of $\partial_z\mathcal{A}(\params;z) = 0$ must satisfy
$z^2 > \tau z$, \textit{i.e.} $z >\tau$ since we retained the non-negative
solution. By differentiating \cref{eq:82} another time with respect to $z$, we
see that $\partial_z^2 \mathcal{A}(\params;z) < 0$ for $z > 0$, hence
$\zeta(\params)$ must be unique. Further
$\lim_{z\to 0}\partial_z \mathcal{A}(\params;z) = +\infty$ and
$\lim_{z\to \infty}\partial_z \mathcal{A}(\params;z) = -\infty$ by
\cref{eq:82} again, so $\zeta(\params)$ must exists.

\subsection{Proof of \texorpdfstring{\cite[\cref{main-pro:5}]{caron:naulet:rousseau:main}}{Lemma \ref{main-pro:5}}}%
\label{sec:proof:pro:5}

Using the defining equation of $z\mapsto R_1(\params;z)$, it is clear that for any
$z \in \Reals$,
\begin{equation}
  \label{eq:3}
  \frac{1}{2}\partial_z^2\mathcal{A}(\params; \zeta(\params))z^2 - R_1(\params;-iz)%
  =- \{\mathcal{A}(\params;\zeta(\params) - iz) - \mathcal{A}(\params; \zeta(\params)) \}.%
\end{equation}
From the definition of $\mathcal{A}$, we have,
\begin{align*}
  \mathcal{A}(\params;\zeta(\params) -i z) - \mathcal{A}(\params; \zeta(\params))%
  &=- \frac{(\zeta(\params)-iz)^2 - \zeta(\params)^2}{4\nedg}%
    + \frac{\tau(\zeta(\params) - iz) - \tau \zeta(\params)}{2\nedg}\\%
  &\quad + \Big(1 - \frac{\spars \nver}{\nedg}\Big)%
    \Big( \log(\zeta(\params) - iz)- \log(\zeta(\params)) \Big)\\
  &\quad + \frac{\size}{\nedg}\Big( \LTS(\spars,\tau;\zeta(\params) -i z) -
    \LTS(\spars,\tau;\zeta(\params)) \Big).
\end{align*}
Hence for every $z \in \Reals$,
\begin{align*}
  \Re\{\mathcal{A}(\params;\zeta(\params) -i z) - \mathcal{A}(\params; \zeta(\params)) \}
  &= \frac{z^2}{4\nedg}%
    + \Big(1 - \frac{\spars \nver}{\nedg}\Big)%
    \Big( \log \sqrt{\zeta(\params)^2 + z^2} - \log \zeta(\params)\Big)\\
  &\quad + \frac{\size}{\nedg} \Re\Big(\LTS(\spars,\tau;\zeta(\params) -i z) -
    \LTS(\spars,\tau;\zeta(\params)) \Big).
\end{align*}
Therefore, for all $z\in \Reals$,
\begin{equation}
  \label{eq:17}
  \Re\{\mathcal{A}(\params;\zeta(\params) -i z) - \mathcal{A}(\params; \zeta(\params)) \}%
  \geq \frac{z^2}{4\nedg}%
  + \frac{\size}{\nedg} \Re\Big(\LTS(\spars,\tau;\zeta(\params) -i z) -
  \LTS(\spars,\tau;\zeta(\params)) \Big).
\end{equation}
Assume first that $\sigma = 0$. Then, whenever $z \in \Reals$,
\begin{align*}
  \Re\Big(\LTS(\spars,\tau;\zeta(\params) -i z) -
  \LTS(\spars,\tau;\zeta(\params)) \Big)%
  &= \Re\Big( \log(\zeta(\params) - iz) - \log(\zeta(\params))\Big)\\%
  &= \log \sqrt{\zeta(\params)^2 + z^2} - \log \zeta(\params)\\%
  &\geq 0.
\end{align*}
The previous display combined with \cref{eq:3,eq:17} gives the proof of the
proposition in the case $\spars = 0$.

\par When $\spars \ne 0$, because $\zeta(\params) > 0$ we have
$\arg( \zeta(\params) - iz) = -\arctan(z/\zeta(\params))$, and then,
\begin{align*}
  \Re\Big(\LTS(\spars,\tau;\zeta(\params) -i z) -
  \LTS(\spars,\tau;\zeta(\params)) \Big)%
  &= \Re\Big(
    \frac{(\zeta(\params) - iz)^{\spars} -
    \zeta(\params)^{\spars}}{\spars}%
    \Big)\\
  &= \frac{ (\zeta(\params)^2 + z^2)^{\spars/2}\cos(\spars
    \arctan(z/\zeta(\params)) - \zeta(\params)^{\spars}  }{\spars}\\
  &= \zeta(\params)^{\spars}f_{\spars}(z/\zeta(\params)),
\end{align*}
where for simplicity, we defined $f_{\spars} : \Reals \rightarrow \Reals$ such
that
\begin{equation*}
  f_{\spars}(y)%
  \coloneqq%
  \frac{ (1 + y^2)^{\spars/2}\cos(\spars
    \arctan(y)) - 1  }{\spars}.
\end{equation*}
To finish the proof of the proposition, it suffices to show that $f$ is
non-negative. When $\spars < 0$, we have $\cos(\spars \arctan(y)) \leq 1$
uniformly, and hence
\begin{equation*}
  f_{\spars}(y)%
  \geq \frac{(1 + y^2)^{\spars/2} - 1}{\spars}%
  \geq \frac{1}{-\spars} - \frac{(1+y^2)^{\spars/2}}{-\spars}%
  \geq 0.%
\end{equation*}
Now we consider the case $0 < \spars < 1$. Because $f$ is symmetric, it
suffices to do the analysis for $y \geq 0$. By differentiation, we get
\begin{equation*}
  f_{\spars}'(y)%
  = (1 + y^2)^{-1+\spars/2}\cos(\spars \arctan(y))\Big(%
  y - \tan(\spars \arctan(y)) \Big).
\end{equation*}
But, $y \geq 0 \Rightarrow \arctan(y) \in [0,\pi/2]$, and hence
$\cos(\spars \arctan(y)) \geq 0$ as $0 < \spars < 1$. Similarly,
$y > \tan(\spars \arctan(y))$ when $0 < \spars < 1$, and thus
$f_{\spars}'(y) \geq 0$ when $y \geq 0$ and $0 < \spars < 1$. This means that
$f$ is increasing on $\NNReals$, and thus
$f_{\spars}(y) \geq f_{\spars}(0) = 0$.

\subsection{Remaining part of the proof of \texorpdfstring{\cite[\cref{main-lem:3}]{caron:naulet:rousseau:main}}{Theorem \ref{main-lem:3}}}
\label{sec:proof-lem:6}

To finish the proof of \cite[\cref{main-lem:3}]{caron:naulet:rousseau:main}, it
remains to prove the \cref{main-item:lem:3:3}.
The proof is almost identical to the proof of the \cref{main-item:lem:3:2} established in
\cite[\cref{main-sec:proof-lem:3}]{caron:naulet:rousseau:main}, and consists mostly on obtaining a refinement in the
bound for $\Delta(\params)$. In particular, the starting point to the proof is
\cite[\cref{main-eq:6}]{caron:naulet:rousseau:main}. To ease the notations, we write
$a(\params) \coloneqq (-\nedg
\partial_z^2\mathcal{A}(\params;\zeta(\params)))^{1/2}$. Then, it follows from
\cite[\cref{main-eq:1,main-eq:8}]{caron:naulet:rousseau:main} and a suitable change of variables that for some $C > 0$ to be
chosen accordingly%
\begin{multline}
  \label{eq:18}
  1 + \Delta(\params)%
  = \int_{[-C a(\params) ,C a(\params)]}
  \frac{1}{\sqrt{2\pi}}e^{-u^2/2}e^{-\nedg R_1(\params, -i u /
    a(\params))}\,\intd u\\
  +\int_{[-C,C]^c} \frac{1}{\sqrt{2\pi}}e^{- \frac{a(\params)^2}{2}z^2 -
    \nedg R_1(\params;-iz) }\,\intd z.
\end{multline}
Now remark that because of \cite[\cref{main-pro:5}]{caron:naulet:rousseau:main} the second integral in the last display
is a $O(e^{-C^2/4})$ as $C\to \infty$, where we have used the well-known tail
bound for the Normal distribution. On the other hand, it is seen that
$\nedg \partial_z^2\mathcal{A}(\params,z) \leq -1/2$ for all admissible
$\params$ and for all $z$, since $\spars < 1$, $\size > 0$ and
$\nver \leq \nedg$. This entails that $a(\params) \geq 1/2$ for any $\params$,
and thus
$\int_{[-Ca(\params),Ca(\params)]^c}e^{-u^2/2}\,\intd u = O(e^{-C^2/8})$. By
\cref{eq:18}, it follows as $C\to \infty$
\begin{equation}
  \label{eq:19}
  \Delta(\params)%
  =\frac{1}{\sqrt{2\pi}}\int_{[-Ca(\params)],Ca(\params)]}e^{-u^2/2}\Big(%
  e^{-\nedg R_1(\params;\frac{-iu}{a(\params)})} - 1\Big)\, \intd u%
  + O(e^{-C^2/8}).
\end{equation}

In order to control $\Delta(\params)$, it thus remain to control
$x \mapsto \nedg R_1(\params;-ix)$ for $x \in [-C,C]$, in virtue of
\cref{eq:19}. We will proceed using Cauchy's integral formula. In particular,
from our choice for the determination of the complex logarithm, the function
$z \mapsto \mathcal{A}(\params;\zeta(\params) + z)$ is complex-analytic on
$\Complex \backslash \Set{z \in \Complex \given \Re(z) = 0,\ \Im(z) \leq -
  \zeta(\params)}$. Then, by Cauchy's integral formula, for any $z\in \Complex$
with $|z| < \zeta(\params)/2$, we have
\begin{equation*}
  R_1(\params;z)%
  = \frac{1}{2\pi i}\oint_{|\xi| = \zeta(\params)/2} \frac{\mathcal{A}(\params; \zeta(\params) +
    \xi)}{\xi - z}\frac{z^3}{\xi^3}\,\intd \xi.
\end{equation*}
Since $|\xi - z| \geq |\xi| - |z|$, the previous implies for every $z \in
\Complex$ with $|z| \leq \zeta(\params)/4$ that
\begin{align}
  \notag
  |R_1(\params;z)|
  &\leq \frac{1}{2\pi} \frac{4}{\zeta(\params)}\frac{2^3}{\zeta(\params)^3}|z|^3%
    \sup_{|\xi| = \zeta(\params)/2}|\mathcal{A}(\params;\zeta(\params)+\xi)| \times 2\pi
    \frac{\zeta(\params)}{2 }\\
  \notag
  &= \frac{16 |z|^3}{\zeta(\params)^3}\sup_{\varphi \in [-\pi,\pi]}\mathcal{A}(\params;
    \zeta(\params)(1 + e^{i\varphi}/2))|\\
  \label{eq:24}
  &\leq \frac{16|z|^3}{\zeta(\params)^3} \sup_{1/2\leq x \leq 2}\sup_{\varphi\in
    [-\pi,\pi]}|\mathcal{A}(\params; x\zeta(\params)e^{i\varphi})|.
\end{align}
But, assuming without loss of generality $\spars \ne 0$, for all $x > 0$ and all
$\varphi \in (-\pi,\pi)$,
\begin{multline*}
  \mathcal{A}(\params;x \zeta(\params) e^{i\varphi})%
  = - \frac{\tau^2}{4\nedg}%
  - \frac{x^2 \zeta(\params)^2e^{2i\varphi}}{4\nedg}%
  + \frac{\tau \zeta(\params) xe^{i\varphi}}{2\nedg}%
  + \Big(1 - \frac{\spars \nver}{\nedg}\Big)\log\big(x \zeta(\params)
  \big)\\
  + i\Big(1 - \frac{\spars \nver}{\nedg}\Big)\varphi%
  + \frac{\size}{\nedg} \frac{(x \zeta(\params))^{\spars}e^{i\spars \varphi} -
    \tau^{\spars}}{\spars}.
\end{multline*}
Remark that,
\begin{align*}
  \frac{(x \zeta(\params))^{\spars}e^{i\spars
  \varphi} - \tau^{\spars}}{\spars}%
  &= \frac{(x \zeta(\params))^{\spars} - \tau^{\spars}}{\spars}e^{i \spars
    \varphi}%
    + \frac{\tau^{\spars}}{\spars}(e^{i \spars \varphi} - 1)\\
  &= \LTS(\spars,\tau;x \zeta(\params)) e^{i\spars \varphi}%
    + \frac{\tau^{\spars}}{\spars}(e^{i\spars \varphi/2} - e^{-i\spars
    \varphi/2})e^{i\spars \varphi/2}\\
  &=\LTS(\spars,\tau;x \zeta(\params)) e^{i\spars \varphi}%
    + \varphi\tau^{\spars} \frac{\sin(\spars \varphi/2)}{\spars \varphi/2}
    e^{i\spars \varphi/2}.
\end{align*}
Hence,%
\begin{multline*}
  |\mathcal{A}(\params;x \zeta(\params)e^{i\varphi})|%
  \leq \frac{\tau^2 + x^2 \zeta(\params)^2 + 2\tau \zeta(\params)}{4\nedg}%
  + 2\Big(1 - \frac{\spars \nver}{\nedg}\Big)\Big(\pi + \big|\log\big(x
  \zeta(\params) \big) \big|\Big)\\
  + \frac{s}{\nedg}\Big| \LTS(\spars,\tau;x \zeta(\params)) + \pi \tau^{\spars}
  \Big|.
\end{multline*}
By \cite[\cref{main-eq:7}]{caron:naulet:rousseau:main}, we find that for every $\params$ we always have
$\zeta(\params) \geq \tau$, and $\zeta(\params)^2 \geq 2\nedg(1 - \spars \nver/\nedg)$.
Thus, from the previous display we obtain for all $x\in [1/2,2]$ and all
$\varphi \in [-\pi,\pi]$
\begin{equation}
  \label{eq:25}
  |\mathcal{A}(\params;x \zeta(\params)e^{i\varphi})|%
  \lesssim \frac{\zeta(\params)^2|\log(\zeta (\params))|}{\nedg}%
  + \frac{\size}{\nedg}\Big| \LTS(\spars,\tau;x \zeta(\params)) +
  \pi\tau^{\spars}  \Big|.
\end{equation}

Now, consider the case $\spars \geq 0$. Then assuming without loss of generality
$\spars \ne 0$ (it suffices to extend by continuity), we have for any $x \in
[1/2,2]$ and any $\tau > 0$,
\begin{align}
  \notag
  |\size \LTS(\spars,\tau; x\zeta(\params))|
  &\leq \size \Big|\frac{x^{\spars} - 1}{\spars}\zeta(\params)^{\spars}  \Big|%
    + \size \Big|\frac{\zeta(\params)^{\spars} - \tau^{\spars}}{\spars}  \Big|\\
  \notag
  &\leq (1\vee x^{\spars}) \size
    \zeta(\params)^{\spars}\sup_{u\in[1/2,1]}\Big|\frac{1 -
    u^{\spars}}{\spars}\Big|%
    + \size \zeta(\params)^{\spars} \frac{1 - (\tau/\zeta(\params))^{\spars}}{\spars}\\
  \label{eq:29}
  &\leq \Big(2\log(2) + \log \frac{\zeta(\params)}{\tau}\Big)%
    \size \zeta(\params)^{\spars},
\end{align}
where for the last line we have used that
$(1 - z^{\spars})/\spars \leq \log(1/z)$ for all $z\in (0,1)$ and all
$\spars \in (0,1)$ (this can be seen by differentiating with respect to
$\spars$). We also have used that $\zeta(\params) \geq \tau$, which can be
deduced from \cite[\cref{main-eq:7}]{caron:naulet:rousseau:main}. From \cite[\cref{main-eq:7}]{caron:naulet:rousseau:main} again, we also see that
$\zeta(\params)^2 \geq \size \zeta(\params)^{\spars}$. Also, because
$\params \in \mathcal{S}_K$ we have
$\size \tau^{\spars} = O( \sqrt{2\nedg} ) = O(\zeta(\params))$ and
$\tau \geq K^{-1}$, and thus we obtain from \cref{eq:25} that
$|\mathcal{A}(\spars; x\zeta(\params)e^{i\varphi})| = O(\zeta(\params)^2
\log(\zeta(\params))/\nedg)$, at least for $\spars \geq 0$. When $\spars < 0$,
it is easily seen with a similar reasoning that
$|\size \LTS(\spars,\tau;x \zeta(\params))| \lesssim \zeta(\params)^2 +
\size\tau^{\spars} \log(\zeta(\params))$, and hence, we obtain from \cref{eq:25}
that for any $x\in [1/2,2]$, any $\varphi\in [-\pi,\pi]$ and any
$\params \in \mathcal{S}_K$,
\begin{equation*}
  |\mathcal{A}(\params;x \zeta(\params)e^{i\varphi})|%
  = o( \zeta(\params)^3 / \nedg).
\end{equation*}
It then follows by combining \cref{eq:19,eq:24,eq:25} that as $\nedg \to \infty$
\begin{align}
  \label{eq:27}
  \Delta(\params)
  &= C^3a(\params)^3 \times o(1) + O\big(e^{-C^2/8} \big),
\end{align}

We already have shown that $a(\params) \geq 1/2$, we now need a more precise
estimate. Recalling that
$a(\params)^2 = -\nedg \partial_z^2\mathcal{A}(\params;\zeta(\params))$, we have
under the assumptions of the Lemma,
\begin{align*}
  a(\params)^2
  &= \frac{1}{2} + \frac{1}{2}\Big(1 - \frac{\spars
    \nver}{\nedg}\Big)\frac{2\nedg}{\zeta(\params)^2}%
    + (1-\spars) \frac{\size \zeta(\params)^{\spars}}{\zeta(\params)^2}\\
  &= \frac{1}{2} + \frac{1}{2}\frac{1}{\zeta(\params)^2}\Big(%
    2\nedg\Big(1 - \frac{\spars \nver}{\nedg}\Big)+ (1-\spars)\size
    \zeta(\params)^{\spars} \Big)\\
  &=\frac{1}{2} + \frac{1}{2} \frac{1}{\zeta(\params)^2} \Big(\zeta(\params)^2 - \tau
    \zeta(\params) - \spars \size \zeta(\params)^{\spars}\Big),
\end{align*}
where the last line follows from the definition of $\zeta(\params)$, in
\cite[\cref{main-eq:7}]{caron:naulet:rousseau:main}. Thus, under the assumptions of the Lemma we have
$a(\params)^2 = 1 + o(1)$. Since \cref{eq:27} is true for arbitrary choice of
$C > 0$, this indeed shows that $\Delta(\params) = o(1)$ if the conditions of
the Lemma are met. Then by \cite[\cref{main-eq:6}]{caron:naulet:rousseau:main}, we have
\begin{align*}
  \log I(\params)%
  &= -\nedg \mathcal{A}(\params;\zeta(\params)) - \log \sqrt{2a(\params)} + \log(1 +
    \Delta(\params))\\
  &= -\nedg \mathcal{A}(\params;\zeta(\params)) - \frac{\log(2)}{2} + o(1).
\end{align*}

\section{Existence and uniqueness of the MLE, concentration of the likelihood:
  auxiliary results}
\label{sec:exist-uniq-mle}

\subsection{Proof of \texorpdfstring{\cite[\cref{main-pro:concent:4}]{caron:naulet:rousseau:main}}{Lemma \ref{main-pro:concent:4}}}
\label{sec:concent:case-where-spars}


We will use multiple times that under \cite[\cref{main-ass:concent:degdist}]{caron:naulet:rousseau:main} we have for any
$x \in (0,1)$ that $\sparsfunc'(x) = O(\nver)$, otherwise it cannot be the
case that $\hat{\alpha}_t$ converges in $(0,1)$.

We first study $\spars \leq -C$, for $C>0$. Observe that for any $c_0 \in (0,1)$, we
have $\sparsfunc(c_0) = \sparsfunc(0) + O(\nver)$. By definition $\Psi(\spars)$ and
by \cref{pro:concent:1},
\begin{align*}
  \Psi(\spars) - \Psi(c_0)%
  &\leq \nver \log(-\spars)%
    + \{\sparsfunc(\spars) - \sparsfunc(0)\}%
    - \nedg \log \beta_{\spars}%
    - \spars \nver \log \frac{\beta_{\spars}-1}{\beta_{\spars}}%
    + O(\nver)\\
  &\eqqcolon F(\spars) + O(\nver).
\end{align*}
We consider two situations here, according to whether
$\spars \leq - b\nedg /\nver$ or $\spars > - b\nedg/\nver$ for some constant
$b> 0$. For any $b>0$, if $\spars \leq -b \nedg/\nver$,  \cref{pro:concent:concent:1} implies that
\begin{align*}
  \Psi(\spars) - \Psi(c_0)%
  \leq -\frac{1}{2} \frac{\nedg(1+o(1))}{1 + \frac{1}{2b}}%
  + O(\nver).
\end{align*}
If $\spars \in (-b\nedg/\nver,-C)$,  then using
$\beta_{\spars} = 1 - \spars \nver/\nedg$
\begin{align*}
  F'(\spars)%
  &= - \frac{\nver}{-\spars}%
    +\partial_\spars \zeta(\params)%
    +\nver \frac{1}{1 - \spars\nver/\nedg}%
    +\frac{\nver^2}{\nedg}\frac{-\spars}{1 - \spars\nver/\nedg}%
    + \nver \log\Big(1 - \frac{\spars\nver}{\nedg}\Big)\\
  &\quad%
    - \nver - \nver \log\Big(\frac{-\spars \nver}{\nedg}\Big)\\
  &= - \frac{\nver}{-\spars}%
    +\partial_\spars \zeta(\params)%
    + \nver \log \Big(1 - \frac{\spars \nver}{\nedg}\Big)%
    - \nver \log\Big(\frac{-\spars\nver}{\nedg}\Big)\\
  &=- \frac{\nver}{-\spars}%
    +\partial_\spars \zeta(\params)%
    + \nver \log \Big(1 + \frac{\nedg}{-\spars \nver} \Big).
\end{align*}
Since $\partial_\spars^2 \zeta(\params)< 0$ on $(-\infty, 1)$,
$\sup_{\spars \leq -C}\{-\sparsfunc'(\spars)\} \leq - \sparsfunc'(\hat\alpha_t) =
O(\nver)$ under \cite[\cref{main-ass:concent:degdist}]{caron:naulet:rousseau:main}. On the other hand,
$\nver \log(1 + \frac{\nedg}{-\spars \nver}) \geq \nver \log(1 +
\frac{1}{b})$, which can be made arbitrary larger than any multiple constant
of $\nver$ by choosing $b$ small enough. Hence
$F'(\spars) > 0$ on $(-b\nedg/\nver,-C)$ and for all
$\spars \in (-b\nedg/\nver,-C)$,
\begin{align*}
  \Psi(\spars) - \Psi(c_0)
  &\leq%
    \nver \log(C)%
    + \{\sparsfunc(-C) - \sparsfunc(0)\}%
    - C\nver \log\Big\{1 + \frac{\nedg}{C\nver}\Big\}%
    - C \nver + O\Big(\frac{\nver^2}{\nedg}\Big).
\end{align*}
With similar arguments $\sparsfunc(-C) - \sparsfunc(0) = O(\nver)$ under
\cite[\cref{main-ass:concent:degdist}]{caron:naulet:rousseau:main}, and \cite[\cref{main-eq:concent:63}]{caron:naulet:rousseau:main} holds.

We now study $\spars >c_2$.   Since $u > 0$ and $g(\spars,\varepsilon) > 0$, we have
\begin{align*}
  H_{\spars}(\varepsilon,u)%
  \leq \nver \log u - \frac{\nedg}{2}\log(1+u)%
  + \frac{\nedg}{2}\log(1-\varepsilon) + \frac{\nedg}{2}(1-\varepsilon)\beta_{\spars}.
\end{align*}
the right hand side is maximized in $u = 2\nver/\nedg(1+ O(\nver/\nedg))$ and
$\varepsilon=0$ which leads to
\begin{align*}
  H(\spars,\varepsilon,u)%
  \leq \nver \log \frac{\nver}{\nedg} - \nver + \nver \log(2)%
  + \frac{\nedg}{2}\beta_{\spars}%
  +O\Big( \frac{\nver^2}{\nedg}\Big), \quad \forall u,\varepsilon, \spars.
\end{align*}
Moreover  let $c_0 > 0$ then
$\Psi(c_0) = K(c_0) + \sup_{\varepsilon,u}H(c_0,\varepsilon,u) = \sparsfunc(c_0) +
\sup_{\varepsilon,u}H(0,\varepsilon,u) + O(\nver)$.  Choosing $\varepsilon_{*} = \frac{c_0\nver}{\nedg }=o(1)$, we have at $\spars = c_0$,
$\frac{1}{2} + \beta_{\spars}g(\spars,\varepsilon) =
(1+o(1))c_0^{-1}$
which combined with \cref{pro:concent:3}, leads to
$$
\sup_{\varepsilon,u}H(c_0,\varepsilon,u)%
\geq \sup_uH(c_0,\varepsilon_{*},u)  = -\nver \log \frac{\nver}{\nedg}%
+ \frac{\nedg}{2}%
+ O(\nver).
$$
Hence, as soon as $\nver = o(\nedg)$,
\begin{equation}
  \label{pro:concent:7}
  \Psi(c_0)%
  \geq \sparsfunc(c_0) + \nver \log \frac{\nver}{\nedg}%
  + \frac{\nedg}{2}%
  + O(\nver).
\end{equation}
Using \cref{pro:concent:7}, we then obtain that for all $\sigma > c_2$
\begin{align*}
  \Psi(\spars) - \Psi(c_0)%
  &\leq \{\sparsfunc(\spars) - \sparsfunc(c_0)\} + O(\nver)\\
  &= \sum_{j\geq 2}\ndeg{j}\sum_{k=1}^{j-1}\log \frac{k -
    \spars}{k-c_0} + O(\nver)\\
  &\leq \sum_{j\geq 2}\ndeg{j}\sum_{k=1}^{j-1}\log \frac{k -
    c_2}{k-c_0} + O(\nver)\\
  &\leq \sum_{j\geq 2}\ndeg{j}\log \frac{1 - c_2}{1-c_0} + O(\nver)\\
  &\leq -K(\nver - \ndeg{1})
\end{align*}
for any $K>0$, by choosing $c_2$ sufficiently close to 1. Furthermore,
\cite[\cref{main-ass:concent:degdist}]{caron:naulet:rousseau:main} implies that $\nver - \ndeg{1}
\asymp \nver$, otherwise $\hat\alpha_t$ would  converge to $1>\alpha_0$. Hence \cite[\cref{main-pro:concent:9}]{caron:naulet:rousseau:main} is proved.

\begin{lemma}
  \label{pro:concent:1}
  Let $\spars \leq -C$ for some $C > 0$. Then,
  \begin{multline*}
    \sup_{\varepsilon,u}H(\spars,\varepsilon,u)%
    =%
    \nver \log \frac{\nver}{\nedg}%
    +\nver \log(-\spars)%
    - \nver \log \beta_{\spars}%
    - \spars \nver \log \frac{\beta_{\spars}-1}{\beta_{\spars}}\\%
    + \frac{\nedg}{2}%
    - \frac{\nedg}{2}\log \beta_{\spars} + O(\nver).
  \end{multline*}
\end{lemma}
\begin{proof}

  The starting point is \cref{pro:concent:3}, and in particular the
  \cref{eq:concent:boundH}. The term
  $-\log( \frac{1}{2} + \beta_{\spars}g(\spars,\varepsilon))$ is not trivial to
  apprehend. We split the analysis into two scenarios, according to whether
  $1-\varepsilon \leq \frac{a}{\beta_{\spars}}$ or not, for some $a \in (0,1)$ to
  be determined. Under the scenario
  $1 - \varepsilon \leq \frac{a}{\beta_{\spars}}$, we note that
  $\beta_{\spars}g(\spars,\varepsilon) \geq 0$, from \cref{eq:concent:boundH},
  \begin{align}
    \label{eq:concent:66}
    H(\spars,\varepsilon,u)%
    &\leq
      - \frac{\nedg}{2}\Big\{\log \frac{1}{a} - a\Big\}%
      + \nver \log \frac{\nver}{\nedg}%
      + O(\nver).
  \end{align}
  We claim that the previous implies that the supremum of
  $(\varepsilon,u) \mapsto H(\spars,\varepsilon,u)$ has to be achieved for
  $1 - \varepsilon > \frac{a}{\beta_{\spars}}$. We keep that claim in mind, and we
  now analyse $H(\spars,\varepsilon,u)$ for
  $1 - \varepsilon > \frac{a}{\beta_{\spars}}$. Then, we can simplify things a bit.
  Indeed, by a Taylor expansion of $f(\spars,\cdot)$ near
  $\varepsilon \approx 1$, we can obtain that
  $f(\spars,\varepsilon) \geq 1 - \varepsilon$ for all $\varepsilon \in
  (0,1)$. That means that
  $g(\spars,\varepsilon) \geq \frac{1}{2}f(\spars,\varepsilon)$. Thus,
  $\nver \log\big( \frac{1}{2} + \beta g(\spars,\varepsilon)\big) \geq \nver
  \log\big( \frac{1}{2} + \frac{\beta_{\spars}}{2}f(\spars,\varepsilon) \big)
  \geq \nver \log \beta_{\spars} + \nver \log f(\spars,\varepsilon) + O(\nver)$.
  But,
  $-\nver \log f(\spars,\varepsilon)= \nver\log(-\spars) + \nver (-\spars) \log
  \varepsilon - \nver \log(1 - \varepsilon^{-\spars})$, Since we assume that
  $\varepsilon < 1 - \frac{a}{\beta_{\spars}}$, and $-\spars > 0$, then
  $\varepsilon^{-\spars} \leq \exp\{-(-\spars)(-\log(1 -
  \frac{a}{1-\spars\nver/\nedg}) \} \leq \exp\{- \frac{(-\spars)a}{1 -
    \spars\nver/\nedg} \}$. Because $-\spars \geq C$, this means that
  $\varepsilon^{-\spars} \leq \exp\{- \frac{C a}{1 + C\nver/\nedg}\}$ uniformly,
  which is always bounded away from $1$. Consequently,
  $-\nver \log(1 - \varepsilon^{-\spars}) = O(\nver)$, uniformly, and in the
  second scenario \cref{eq:concent:boundH} becomes
  \begin{multline*}
    \sup_{u>0}H(\spars,\varepsilon,u)%
    = \nver \log \frac{\nver}{\nedg}%
    +\nver \log(-\spars)%
    - \nver \log\beta_{\spars}\\%
    - \spars \nver \log \varepsilon%
    + \frac{\nedg}{2}(1-\varepsilon)\beta_{\spars}%
    + \frac{\nedg}{2}\log(1-\varepsilon) + O(\nver).
  \end{multline*}
  That is, it is enough to maximize
  $G_{\spars}(\varepsilon) \coloneqq -\spars \nver\log \varepsilon +
  \frac{\nedg}{2}\log(1-\varepsilon) +
  \frac{\nedg}{2}(1-\varepsilon)\beta_{\spars}$. We note that
  $G_{\spars}'(\varepsilon) = \frac{-\spars \nver}{\varepsilon} -
  \frac{\nedg}{2}\frac{1}{1-\varepsilon} - \frac{\nedg}{2}\beta_{\spars}$, and
  clearly $G_{\spars}''(\varepsilon) < 0$ for all $\varepsilon > 0$, whence
  $G_{\spars}$ admits a unique maximizer solution to
  \begin{align*}
    \frac{-\spars \nver}{\varepsilon}%
    = \frac{\nedg}{2}\frac{1}{1-\varepsilon} + \frac{\nedg}{2}\Big(1 -
    \frac{\spars \nver}{\nedg}\Big)%
    &\iff%
      \frac{-\spars\nver}{\varepsilon}%
      = \frac{\nedg}{2} \frac{1}{1 - \varepsilon} + \frac{\nedg}{2}  +
      \frac{-\spars \nver}{2}\\
    &\iff%
      -\spars \nver(1 - \varepsilon)%
      = \frac{\nedg}{2} \varepsilon + \frac{\nedg}{2}(1-\varepsilon)\varepsilon +
      \frac{-\spars\nver}{2}(1-\varepsilon) \varepsilon\\
    &\iff%
      - \frac{1}{2}(\nedg  - \spars \nver)\varepsilon^2%
      + \Big( \nedg  - \frac{3\spars\nver}{2} \Big)\varepsilon%
      + \spars\nver = 0.
  \end{align*}
  The previous has two easy solutions, one is seen to be $\varepsilon = 2$ so it
  is outside the domain of $G$, and the other one, of interest, is
  \begin{equation*}
    \varepsilon_{*}%
    \coloneqq%
    \frac{-\spars \nver}{\nedg - \spars\nver}%
    = \frac{-\spars\nver}{\nedg}\cdot \frac{1}{1 - \spars \nver/\nedg }%
    = \frac{\beta_{\spars} - 1}{\beta_{\spars}}.
  \end{equation*}
  It follows, for all $\varepsilon$ such that
  $1-\varepsilon > \frac{a}{\beta_{\spars}}$,
  \begin{multline*}
    \sup_{u > 0}H(\spars,\varepsilon,u)%
    \leq%
      \nver \log \frac{\nver}{\nedg}%
      +\nver \log(-\spars)%
      - \nver \log \beta_{\spars}\\%
      - \spars \nver \log \frac{\beta_{\spars}-1}{\beta_{\spars}}%
      + \frac{\nedg}{2}%
      - \frac{\nedg}{2}\log \beta_{\spars} + O(\nver).
  \end{multline*}
  It is clear that the previous is indeed achieved at $\varepsilon_{*}$, and for
  $a > 0$ small enough (but constant), it is also bigger than the bound in
  \cref{eq:concent:66} when $\spars \leq -C \leq 0$. Hence the conclusion.
\end{proof}

\begin{lemma}
  \label{pro:concent:3}
  If $\nver = o(\nedg)$, for any fixed value of $(\spars,\varepsilon)$, the function
  $u \mapsto H(\spars,\varepsilon,u)$ admits a unique maximizer
  $\bar{u}(\spars,\varepsilon) > 0$. Furthermore, this maximizer has the asymptotic
  expansion
  \begin{align*}
    \bar{u}(\spars,\varepsilon)
    &= \frac{\nver}{\nedg}\cdot \frac{1 + O(\nver/\nedg)}{\frac{1}{2} + \beta_{\spars} g(\spars,\varepsilon)},
  \end{align*}
  and,
  \begin{multline}
    \label{eq:concent:boundH}
    H(\spars,\varepsilon,\bar{u}(\spars,\varepsilon))%
    =%
      -\nver + \nver \log \frac{\nver}{\nedg}%
      - \nver \log \Big(\frac{1}{2} + \beta_{\spars}
      g(\spars,\varepsilon)\Big)\\%
      + \frac{\nedg}{2}\log(1-\varepsilon)%
      + \frac{\nedg}{2}(1-\varepsilon)\beta_{\spars}%
      + O\Big( \frac{\nver^2}{\nedg}\Big).
  \end{multline}
\end{lemma}
\begin{proof}
  We first establish that for fixed values of $(\spars,\varepsilon)$ the
  function $u \mapsto H(\spars,\varepsilon,u)$ has a maximum. Indeed, any
  extremum of $H(\spars,\varepsilon,\cdot)$ must be solution to
  \begin{align}
    \label{eq:concent:70}
    \frac{\nver}{u}%
    - \frac{\nedg}{2}\frac{1}{1+u}%
    - \nedg \beta_{\spars}g(\spars,\varepsilon) = 0.%
  \end{align}
  The limit as $u \to 0$ of the lhs of \cref{eq:concent:70} is $+\infty$, and the limit
  as $u \to \infty$ is $-\nedg \beta_{\spars}g(\spars,\varepsilon) < 0$, and it is
  a continuous function of $u$. Hence, it is the case that
  $\partial_uH(\spars,\varepsilon,u) = 0$ has solutions. Furthermore, it is clear
  that any solution also satisfies
  $\frac{\nver}{u}\geq \frac{\nedg}{2}\frac{1}{1+u}$, \textit{i.e}
  $\frac{u}{1+u}\leq \frac{2\nver}{\nedg}$, whence
  $u \leq \frac{2\nver}{\nedg}(1+o(1))$. Hence, it is enough to look for
  solutions in $(0,\frac{3\nver}{\nedg})$. On that interval,
  $\partial_u^2 H(\spars,\varepsilon,u) = -\frac{\nver}{u^2} + \frac{\nedg}{2}
  \frac{1}{(1+u)^2} < 0$. Then, \cref{eq:concent:70} has a unique solution, and it is a
  maximum of $H(\spars,\varepsilon,\cdot)$. Regarding the asymptotic form of
  $\bar{u}(\spars,\varepsilon)$, let
  $u = \frac{\nver}{\nedg}\frac{1}{\frac{1}{2} +
    \beta_{\spars}g(\spars,\varepsilon)}$. Then, $u$ is the solution to
  $\frac{\nver}{u} = \frac{\nedg}{2} + \nedg \beta_{\spars}
  g(\spars,\varepsilon)$. It then follows from \cref{eq:concent:70} that
  $\frac{\nver}{u} = \frac{\nver}{\bar{u}(\spars,\varepsilon)} + \frac{\nedg}{2}
  \frac{\bar{u}(\spars,\varepsilon)}{1 + \bar{u}(\spars,\varepsilon)}$. The first
  claim follows. The second claim follows because $\log(1+x) = x + O(x^2)$ for
  all $x > -1$.
\end{proof}

\begin{lemma}
  \label{pro:concent:concent:1}
  Under \cite[\cref{main-ass:concent:degdist}]{caron:naulet:rousseau:main}, for all $\spars < 0$,
  \begin{multline*}
    \sparsfunc(\spars) - \sparsfunc(0)%
    \leq%
      - \nver \log(1 - \spars)%
      + \nedg \log\Big(1 - \frac{\spars\nver}{\nedg}\Big)\\%
      - \spars \nver \log\Big(1 - \frac{\nedg}{\spars \nver}\Big)%
      - \frac{1}{2}\frac{\nedg(1+o(1))}{1 + \frac{1}{2} \frac{\nedg}{-\spars\nver}}%
      + O(\nver).
  \end{multline*}
\end{lemma}
\begin{proof}
  For $-\spars \geq 0$ the function $k \mapsto \log(k-\spars)$ is non-negative
  and monotone increasing on $(1,\infty)$, and hence we can bound
  \begin{align*}
    \sparsfunc(\spars)%
    &= \sum_{j\geq 2}\ndeg{j}\sum_{k=1}^{j-1}\log(k - \spars)\\
    &\leq%
      \sum_{j\geq 2}\ndeg{j}\int_1^j\log(x - \spars)\,\intd x\\
    &=\sum_{j\geq 2} \ndeg{j}\Big\{1 - j -(1-\spars)\log(1-\spars) +
      (j-\spars)\log(j-\spars) \Big\}\\
    &= - \nedg%
      - \nver (1-\spars)\log(1-\spars)%
      + \sum_{j\geq 1}\ndeg{j}(j-\spars)\log(j-\spars).
  \end{align*}
  That is,
  \begin{align*}
    \sparsfunc(\spars)%
    &\leq%
      - \nedg - \nver \log(1-\spars)%
      - \spars \sum_{j\geq 1}\ndeg{j}\log \Big\{1 + \frac{j-1}{1-\spars}\Big\}%
      + \sum_{j\geq 1}\ndeg{j}\cdot j \log(j - \spars).%
  \end{align*}
  By assumption, $-\spars > 0$, and then we remark that
  $x \mapsto \log(1 + \frac{x}{-\spars})$ is concave on $(1,\infty)$, hence by
  Jensen's inequality,
  \begin{align*}
    \sum_{j\geq 1}\ndeg{j}\log \Big\{1 + \frac{j-1}{1-\spars}\Big\}%
    &\leq \nver \sum_{j\geq 1}\frac{\ndeg{j}}{\nver}\log \Big\{1 +
      \frac{j}{-\spars}\Big\}\\%
    &\leq \nver \log\Big\{1 + \frac{\sum_{j\geq 1}\ndeg{j}\cdot j}{-\spars}
      \Big\}\\
    &= \nver \log\Big\{1 + \frac{\nedg}{-\spars \nver} \Big\}.
  \end{align*}
  So,
  \begin{align}
    \label{eq:concent:4}
    \sparsfunc(\spars)%
    &\leq%
      -\nedg - \nver\log(1-\spars)%
      - \spars \nver \log\Big(1 + \frac{\nedg}{-\spars\nver} \Big)%
      + \sum_{j\geq 1} \ndeg{j}\cdot j \log(j-\spars).
  \end{align}
  On the other hand, the function $x \mapsto \log(x)$ is non-negative and
  monotone increasing on $(1,\infty)$. Hence we can bound,
  \begin{align*}
    \sparsfunc(0)%
    &= \sum_{j\geq 2}\ndeg{j}\sum_{k=1}^{j-1}\log(k)
    = \sum_{j\geq 3}\ndeg{j}\sum_{k=2}^{j-1}\log(k)
    \geq
      \sum_{j\geq 3}\ndeg{j} \int_{1}^{j-1}\log(x)\,\intd x.
  \end{align*}
  Thus,
  \begin{align*}
    \sparsfunc(0)%
    &\geq
      \sum_{j\geq 3}\ndeg{j}%
      \Big\{2 - j + (j-1)\log(j-1) \Big\}
  \end{align*}
  That is,
  \begin{align}
    \notag
    \sparsfunc(0)%
    &\geq%
      -\nedg%
      + \sum_{j\geq 1}\ndeg{j}\cdot j\log(j)%
      + O(1)\cdot \sum_{j\geq 2}\ndeg{j}\log(j)\\
    \label{eq:concent:5}
    &=%
      -\nedg%
      + \sum_{j\geq 1}\ndeg{j}\cdot j\log(j)%
      + O(\nver),
  \end{align}
  where the second line is true under \cite[\cref{main-ass:concent:degdist}]{caron:naulet:rousseau:main}, because
  $-\sparsfunc'(\alpha_t) \asymp \sum_{j\geq 2} \ndeg{j} \log(j)$, and
  $-\sparsfunc'(\alpha_t) = \frac{\nver}{\alpha_t} = O(\nver)$. Hence, to finish
  the proof it is enough to understand,
  \begin{align}
    \sum_{j\geq 1}\ndeg{j}\cdot j\Big\{ \log(j-\spars) - \log(j)
    \Big\}%
    \label{eq:concent:8b}
    &= \nver \sum_{j\geq 1}\frac{\ndeg{j}}{\nver}\cdot j \log\Big\{1 +  \frac{- \spars}{j}\Big\}.
  \end{align}
  Let $p_{\TrueSize,j} \coloneqq \ndeg{j}/\nver$,
  $\bm{p}_{\TrueSize} = (p_{\TrueSize,1},p_{\TrueSize,2},\dots)$, and
  $\Phi(x) \coloneqq x \log(1 + \frac{-\spars}{x})$. Then we may see the rhs of
  the last display as $\nver\cdot\EE_{\bm{p}_{\TrueSize}}[\Phi(J)]$. We note
  that $\Phi$ is concave on $(1,\infty)$, so we can use Jensen's inequality to
  obtain a bound, but we actually need a finer estimate. To get the next order
  term, we remark that
  $\Phi'(x) = - \log(\frac{x}{-\spars}) - 1 + \frac{x}{1 + \frac{x}{-\spars}} +
  \log(1 + \frac{x}{-\spars}) = -\log( \frac{x}{-\spars}) + \log(1 +
  \frac{x}{-\spars}) - \frac{1}{1 + \frac{x}{-\spars}}$, and then
  $\Phi''(x) = - \frac{1}{x} + \frac{1}{-\spars} \frac{1}{1+\frac{x}{-\spars}} +
  \frac{1}{-\spars} \frac{1}{(1 + \frac{x}{-\spars})^2} =
  \frac{1}{-\spars}\frac{1}{x(1 + \frac{x}{-\spars})^2}\{x + x(1 +
  \frac{x}{-\spars}) - (-\spars)(1 + \frac{x}{-\spars})^2 \} = -\frac{1}{x(1 +
    \frac{x}{-\spars})^2}$, and thus $\Phi$ is concave as claimed. Furthermore,
  $\EE_{\bm{p}_{\TrueSize}}[J] = \frac{\nedg}{\nver}$, and thus by a Taylor
  expansion of $\Phi$ near $\EE_{\bm{p}_{\TrueSize}}[J]$, we find that there is
  some $J_{*}$ in the line segment between $J$ and $J_{*}$ such that,
  \begin{align*}
     \EE_{\bm{p}_{\TrueSize}}[\Phi(J)]%
    &=%
      \Phi(\EE_{\bm{p}_{\TrueSize}}[J])%
      + \EE[\Phi''(\EE_{\bm{p}_{\TrueSize}}[J])(J-\EE_{\bm{p}_{\TrueSize}}[J])]%
      + \frac{1}{2}\EE[\Phi''(J_{*})(J-\EE_{\bm{p}_{\TrueSize}}[J])^2]\\
    &= %
      \Phi(\EE_{\bm{p}_{\TrueSize}}[J])%
      + \frac{1}{2}\EE[\Phi''(J_{*})(J-\EE_{\bm{p}_{\TrueSize}}[J])^2].%
  \end{align*}
  Now we just need a tight enough upper bound on the second term of the last
  display. Since $\Phi' < 0$, we can upper-bound as follows
  \begin{align*}
    \EE_{\bm{p}_{\TrueSize}}[\Phi''(J_{*})(J-\EE_{\bm{p}_{\TrueSize}}[J])^2]%
    &\leq \EE_{\bm{p}_{\TrueSize}}[\Phi''(J_{*})(J-\EE_{\bm{p}_{\TrueSize}}[J])^2 \Ind_{\Set{J \leq
      \frac{1}{2}\EE_{\bm{p}_{\TrueSize}}[J] }}]\\
    &\leq - \frac{\EE_{\bm{p}_{\TrueSize}}[J]^2}{4}%
      \cdot%
      \min_{J \leq\frac{1}{2}\EE_{\bm{p}_{\TrueSize}}[J]}\{-\Phi''(J) \}%
      \cdot%
      \EE_{\bm{p}_{\TrueSize}}[\Ind_{\Set{J\leq\frac{1}{2}\EE_{\bm{p}_{\TrueSize}}[J]}}]\\%
    &\leq%
      - \frac{\EE_{\bm{p}_{\TrueSize}}[J]^2}{4}%
      \cdot%
      \frac{1}{\frac{1}{2}\EE_{\bm{p}_{\TrueSize}}[J]\big(1 + \frac{\frac{1}{2}\EE_{\bm{p}_{\TrueSize}}[J]}{-\spars} \big)}
      \cdot%
      \EE_{\bm{p}_{\TrueSize}}[\Ind_{\Set{J\leq\frac{1}{2}\EE_{\bm{p}_{\TrueSize}}[J]}}]\\
    &=- \frac{1}{2} \frac{\nedg}{\nver} \frac{1}{1 +
      \frac{1}{2}\frac{\nedg}{-\spars\nver} }%
      \sum_{j\geq 1} \frac{\ndeg{j}}{\nver}\Ind_{\Set{j\leq \frac{1}{2} \frac{\nedg}{\nver} } }.
  \end{align*}
  So the rhs of \cref{eq:concent:8b} is no more than,
  \begin{align}
    \label{eq:concent:1b}
    \nver%
    \EE_{\bm{p}_{\TrueSize}}[\Phi(J)]%
    &\leq  \nedg \log\Big(1 + \frac{-\spars \nver}{\nedg}\Big)%
      - \frac{1}{2}%
      \frac{\nedg}{1+\frac{1}{2}\frac{\nedg}{-\spars\nver}}%
      \sum_{j\geq 1} \frac{\ndeg{j}}{\nver}\Ind_{\Set{j\leq \frac{1}{2} \frac{\nedg}{\nver}}}.%
  \end{align}
  We finish the proof by noting that
  $j > \frac{\nedg}{2\nver} \Leftrightarrow \log(j) > \log(\frac{\nedg}{2\nver})$, and thus by
  Markov's inequality, we have
  $\sum_{j \geq 1}\ndeg{j}\Ind\Set{j > \frac{\nedg}{2\nver}} \leq \frac{1}{\log(\nedg/(2\nver))} \sum_{j\geq 1}\ndeg{j} \log(j) = o(\nver)$
  by \cite[\cref{main-ass:concent:degdist}]{caron:naulet:rousseau:main}. Consequently,
  $\sum_{j \geq 1}\ndeg{j}\Ind\Set{j \leq \frac{\nedg}{2\nver}} = \nver(1 + o(1))$. Then,
  the conclusion follows by combining
  \cref{eq:concent:4,eq:concent:5,eq:concent:8b,eq:concent:1b}.
\end{proof}


\subsection{Proof of \texorpdfstring{\cite[\cref{main-pro:concent:5}]{caron:naulet:rousseau:main}}{Lemma \ref{main-pro:concent:5}}}
\label{sec:concent:case-where-}

We first establish the existence and uniqueness
of the maximizer of $(\varepsilon,u) \mapsto H(\spars,\varepsilon,u)$ when $\spars \in (-C,c_2)$.

As already established in \cref{pro:concent:3}, the equation has a unique solution $\bar{u}(\spars,\varepsilon) \in (0,3\nver/\nedg)$, and
$\bar{u}(\spars,\varepsilon) = \frac{\nver}{\nedg}\frac{1 +
  O(\nver/\nedg)}{\frac{1}{2} + \beta_{\spars}g(\spars,\varepsilon)}$. Similarly, if there is a solution
$\tilde{\varepsilon}$ to $\partial_{\varepsilon}H(\spars,\varepsilon,u) = 0$, it must
be the case that
\begin{align*}
  \nedg \beta_{\spars} u \tilde{\varepsilon}^{-1+\spars}%
  &\geq \nedg \beta_{\spars} u\Big( \tilde{\varepsilon}^{-1+\spars} -
    \frac{1}{2}\Big)= \frac{\nedg}{2}\frac{1}{1-\tilde{\varepsilon}} +
    \frac{\nedg}{2}\beta_{\spars} \geq \frac{\nedg}{2}(1 + \beta_{\spars}).
\end{align*}
Since $\spars \geq -C$, we have $\beta_{\spars} = 1 + O(\nver/\nedg)$, and
thus any solution $\tilde{\varepsilon}$ to
$\partial_{\varepsilon}H_{\spars}(\varepsilon,u) = 0$ must satisfy
$\tilde{\varepsilon}^{-1+\spars} \geq u^{-1}(1 + O(\nver/\nedg))$. In particular,
if $(\tilde{\varepsilon},\tilde{u})$ is solution to
$\partial_{\varepsilon}H(\spars,\varepsilon,u) = 0$ and
$\partial_uH(\spars,\varepsilon,u) = 0$, then it has to be the case that
\begin{equation*}
  \tilde{\varepsilon}^{-1+\spars} \geq \frac{1 + O\big(\frac{\nver}{\nedg}\big)}{\tilde{u}}%
  = \frac{\nedg}{\nver}\Big( \frac{1}{2} +
  \beta_{\spars}g(\spars,\tilde{\varepsilon})\Big)\Big(1+O\Big(\frac{\nver}{\nedg}\Big)\Big)%
  \geq \frac{\nedg}{\nver}g(\spars,\tilde{\varepsilon})\Big(1+O\Big(\frac{\nver}{\nedg}\Big)\Big).
\end{equation*}
Now we remark that $g(\spars,\varepsilon) = f(\spars,\varepsilon) -
\frac{1-\varepsilon}{2}\geq \frac{1}{2}f(\spars,\varepsilon) = \frac{1}{2}\frac{1 -
  \varepsilon^{\spars}}{\spars}$, where the last equality is assuming without
loss of generality that $\spars \ne 0$. It follows, whenever $\spars \geq -C$,
because the mapping $\spars \mapsto \frac{\tilde{\varepsilon}^{-\spars} -
  1}{\spars}$ is monotone increasing (recall $\tilde{\varepsilon} \in (0,1)$),
\begin{align*}
  \tilde{\varepsilon}^{-1}%
  \geq
  \frac{1}{2}\frac{\nedg}{\nver}\frac{\tilde{\varepsilon}^{-\spars}-1}{\spars}\Big(1+O\Big(\frac{\nver}{\nedg}\Big)\Big)%
  \geq \frac{1}{2}\frac{\nedg}{\nver}\frac{1 - \tilde{\varepsilon}^C}{C}\Big(1+O\Big(\frac{\nver}{\nedg}\Big)\Big).
\end{align*}
We have proved that $\tilde{\varepsilon}(\spars)$, if it exists, must satisfies
$\tilde{\varepsilon}(\spars) \leq 3C\nver/\nedg$ for all $\spars \geq -C$, at least
when $t$ is large enough.

To prove that $(\tilde{\varepsilon}(\spars),\tilde{u}(\spars))$ exists uniquely, it is
enough to establish that $\varepsilon \mapsto H(\spars,\varepsilon,\bar{u}(\spars,\varepsilon))$ has a unique
maximum at $\tilde{\varepsilon}(\spars)$, and then
$\tilde{u}(\spars) = \bar{u}(\spars,\tilde{\varepsilon}(\spars))$. Let
$\tilde{H}_{\spars}(\varepsilon) \coloneqq H(\spars,\varepsilon,\bar{u}(\spars,\varepsilon))$. Clearly,
$\tilde{H}_{\spars}'(\varepsilon) = \partial_{\varepsilon}H_{\spars}(\varepsilon, \bar{u}(\spars,\varepsilon))$ since
$\partial_uH_{\spars}(\varepsilon,\bar{u}(\spars,\varepsilon)) = 0$. Then, for any $\varepsilon \leq 3C\nver/\nedg$,
recalling that $\bar{u}(\spars,\varepsilon) \leq 3\nver/\nedg$,
\begin{align*}
  \tilde{H}_{\spars}'(\varepsilon)%
  &= \nedg \beta_{\spars} \bar{u}(\spars,\varepsilon)\Big(\varepsilon^{-1+\spars} -
    \frac{1}{2}\Big)%
    - \frac{\nedg}{2}\frac{1}{1-\varepsilon} - \frac{\nedg}{2}\beta_{\spars}\\
  &= \nedg \bar{u}(\spars,\varepsilon)(1+o(1))\varepsilon^{-1+\spars} - \nedg(1+o(1)).
\end{align*}
By \cref{pro:concent:3}, we obtain easily that when $\varepsilon\to 0$ we have
$\bar{u}(\spars,\varepsilon) \asymp \spars \nver/\nedg$ if $\spars > 0$,
$\bar{u}(\spars,\varepsilon) \asymp -\spars\varepsilon^{-\spars}\nver/\nedg$ if
$\spars < 0$, and
$\bar{u}(\spars,\varepsilon) \asymp
\frac{\nver}{\nedg}\frac{1}{\log(1/\varepsilon)}$ if $\spars = 0$. So  $\lim_{\varepsilon \to 0}\tilde{H}_{\spars}'(\varepsilon) = +\infty$. We
already know from the above that $\tilde{H}_{\spars}'(\varepsilon) < 0$ when
$\varepsilon > 3C\nver/\nedg$. Since $\tilde{H}_{\spars}'$ is a continuous
function of $\varepsilon$, there are solutions to
$\tilde{H}_{\spars}'(\varepsilon) = 0$ in $(0,3C\nver/\nedg)$, and only in this
interval. To prove the uniqueness, it is enough to show that
$\tilde{H}_{\spars}''(\varepsilon) < 0$ for all $\varepsilon \in
(0,3\nver/\nedg)$. We have,
\begin{align*}
  \tilde{H}_{\spars}''(\varepsilon)%
  &= \partial_{\varepsilon}^2H(\spars,\varepsilon,\bar{u}(\spars,\varepsilon))%
    +
    \partial_u\partial_{\varepsilon}H(\spars,\varepsilon,\bar{u}(\spars,\varepsilon))%
    \partial_{\varepsilon}\bar{u}(\spars,\varepsilon).%
\end{align*}
Using $\partial_uH(\spars,\varepsilon,\bar{u}(\spars,\varepsilon))
= 0$, we find that
$\partial_{\varepsilon}\partial_uH(\spars,\varepsilon,\bar{u}(\spars,\varepsilon)) +
\partial_u^2H(\spars,\varepsilon,\bar{u}(\spars,\varepsilon))
\partial_{\varepsilon}\bar{u}(\spars,\varepsilon) = 0$, \textit{i.e.}
\begin{align*}
  \partial_{\varepsilon}\bar{u}(\spars,\varepsilon)%
  &=
    \frac{\partial_{\varepsilon}\partial_uH(\spars,\varepsilon,\bar{u}(\spars,\varepsilon))}{-\partial_u^2
    H(\spars,\varepsilon,\bar{u}(\spars,\varepsilon)) }.
\end{align*}
Then,
\begin{align*}
  \tilde{H}_{\spars}''(\varepsilon)%
  &= \partial_{\varepsilon}^2H(\spars,\varepsilon,\bar{u}(\spars,\varepsilon))%
    +
    \frac{\{\partial_{\varepsilon}\partial_uH(\spars,\varepsilon,\bar{u}(\spars,\varepsilon))
    \}^2 }{-\partial_u^2H(\spars,\varepsilon,\bar{u}(\spars,\varepsilon)) }\\
  &=  - \frac{\nedg}{2}\frac{1}{(1-\varepsilon)^2}%
    - (1-\spars)\nedg \beta_{\spars}\bar{u}(\spars,\varepsilon)\varepsilon^{-2+\spars}%
    + \frac{\bar{u}(\spars,\varepsilon)^2 \nedgsq}{\nver}\Big(\varepsilon^{-1+\spars}-
    \frac{1}{2}\Big)^2(1+o(1))\\
  &\leq - \frac{\nedg}{2} - \nedg \bar{u}(\spars,\varepsilon) \varepsilon^{-2+\spars}\Big((1 -
    \spars)(1+o(1))%
    - \frac{\bar{u}(\spars,\varepsilon) \nedg \varepsilon^{\spars}}{\nver}\Big(1 -
    \frac{\varepsilon^{1-\spars}}{2}\Big)^2(1+o(1)) \Big)\\
  &\leq - \frac{\nedg}{2} - \nedg \bar{u}(\spars,\varepsilon) \varepsilon^{-2+\spars}\Big((1 -
    \spars)(1+o(1))%
    - \frac{\bar{u}(\spars,\varepsilon) \nedg \varepsilon^{\spars}}{4\nver}(1+o(1)) \Big).
\end{align*}
Recall that
$$\bar{u}(\spars,\varepsilon) =
\frac{\nver}{\nedg}\frac{1+O(\nver/\nedg)}{\frac{1}{2} +
  \beta_{\spars}g(\spars,\varepsilon)}=
\frac{\nver}{\nedg}\frac{1+O(\nver/\nedg)}{f(\spars,\varepsilon)}, \quad \text{if} \quad \varepsilon \leq 3C\nver/\nedg.$$
Assuming without loss of generality that
$\spars \ne 0$, we have
$\frac{\bar{u}(\spars,\varepsilon) \nedg \varepsilon^{\spars}}{4\nver} = \frac{1 +
  o(1)}{4} \frac{\spars\varepsilon^{\spars}}{1 - \varepsilon^{\spars}}$ and
remark that
\begin{align*}
  1 - \spars
  - \frac{\bar{u}(\spars,\varepsilon) \nedg \varepsilon^{\spars}}{4\nver}
  &= 1-\spars -  \frac{\spars\varepsilon^{\spars}(1+o(1))}{4(1 - \varepsilon^{\spars})}= 1-\spars + o(1) \quad \text{if } \spars >0  \\
  & =      1 - \frac{3\spars }{4} + o(1)  \quad \text{if } \quad \spars<0,
\end{align*}
and in all cases it is large than $1-c_2$. This result extends to $\spars = 0$ by continuity.
Therefore,
$\tilde{H}_{\spars}''(\varepsilon) \leq -\nedg /2< 0$ for
all $\varepsilon \in (0,3C\nver/\nedg)$ and $(\tilde u(\spars), \tilde \varepsilon(\spars))$ exists and is unique.

We now prove \cite[\cref{main-eq:concent:103}]{caron:naulet:rousseau:main}.
We have
$\Psi(\spars) = K(\spars) +
H(\spars,\tilde{\varepsilon}(\spars),\tilde{u}(\spars))$, and,
$\partial_{\varepsilon}H(\spars,\tilde{\varepsilon}(\spars),\tilde{u}(\spars)) = 0$
and $\partial_{u}H(\spars,\tilde{\varepsilon}(\spars),\tilde{u}(\spars)) =
0$. Then,
$\Psi'(\spars) = K'(\spars) +
\partial_{\spars}H(\spars,\tilde{\varepsilon}(\spars),\tilde{u}(\spars))$. Recall
that $\beta_{\spars} = 1 - \spars\nver/\nedg$, and
$\tilde{\varepsilon}(\spars) \leq 3C \nver/\nedg$. then
\begin{align*}
  \Psi'(\spars)%
  &=- \frac{\nver^2}{\nedg} \frac{1}{\beta_{\spars}}%
    +\sparsfunc'(\spars) + \frac{\nver}{2}\frac{1}{\beta_{\spars}}%
    - \nedg \beta_{\spars} \tilde{u}(\spars)
    \partial_{\spars}g(\spars,\tilde{\varepsilon}(\spars))%
    + \nver \tilde{u}(\spars) g(\spars,\tilde{\varepsilon}(\spars))%
    - \frac{\nver}{2}(1-\tilde{\varepsilon}(\spars))\\
  &=\partial_\spars \zeta(\params)%
    - \nedg \beta_{\spars} \tilde{u}(\spars)
    \partial_{\spars}g(\spars,\tilde{\varepsilon}(\spars))%
    + \nver \tilde{u}(\spars) g(\spars,\tilde{\varepsilon}(\spars))%
    + O\Big(\frac{\nver^2}{\nedg}\Big).
\end{align*}
Since
$\tilde{u}(\spars) = \frac{\nver}{\nedg} \frac{1 +
  O(\nver/\nedg)}{f(\spars,\tilde{\varepsilon}(\spars)}$, so that
$\nver \tilde{u}(\spars) g(\spars,\tilde{\varepsilon}(\spars)) =
O(\nver^2/\nedg)$. Hence, \cite[\cref{main-eq:concent:103}]{caron:naulet:rousseau:main} follows because
$\partial_{\spars}g(\spars,\tilde{\varepsilon}(\spars)) =
\partial_{\spars}f(\spars,\tilde{\varepsilon}(\spars))$.%

Finally,  assuming $\spars \ne 0$ we have,
\begin{align*}
  \frac{\partial_{\spars}f(\spars,\varepsilon)}{f(\spars,\varepsilon)}
  &= \frac{\spars}{1 - \varepsilon^{\spars}}%
    \Big\{ - \frac{1 - \varepsilon^{\spars}}{\spars^2} -
    \frac{\varepsilon^{\spars}\log(\varepsilon)}{\spars} \Big\}\\
  &= -\frac{1}{\spars} \frac{\varepsilon^{-\spars} - 1 - \spars
    \log(1/\varepsilon)}{\varepsilon^{-\spars} - 1},
\end{align*}
which can be extended by continuity at $\spars = 0$. Since
$\tilde{\varepsilon}(\spars) = o(1)$, this establishes that for any $K > 0$ we
can choose $c_1> 0$ such that
$-\nver
\frac{\partial_{\spars}f(\spars,\tilde{\varepsilon}(\spars))}{f(\spars,\tilde{\varepsilon}
  (\spars))} \geq K \nver$ for all $\spars \geq c_1$. But on the other hand,
$\partial_\spars^2 \zeta(\params)< 0$, meaning that
$\sparsfunc'(\spars) \geq \sparsfunc'(c_1)$ for all $\spars \in (-C,c_1)$. But
for $c_1 \in (0,1)$,
$\sparsfunc'(c_1) = -\sum_{j\geq 2}\ndeg{j} \sum_{k=1}^{j-1}\frac{1}{k-c_1}
\asymp \sum_{j\geq 1}\ndeg{j}\cdot \log(j) = O(\nver)$, by
\cite[\cref{main-ass:concent:degdist}]{caron:naulet:rousseau:main}. It follows that $\Psi'(\spars) > 0$ for all
$\spars < c_1$ if $c_1 > 0$ is small enough. With similar arguments, if $c_2$
is close enough to one, it must be the case that $\Psi'(c_2) < 0$, by
\cite[\cref{main-eq:concent:103}]{caron:naulet:rousseau:main}.

\subsection{Proof of \texorpdfstring{\cite[\cref{main-pro:concent:11}]{caron:naulet:rousseau:main}}{Lemma \ref{main-pro:concent:11}}}%
\label{sec:concent:case-where-spars-2}

By \cite[\cref{main-pro:concent:5}]{caron:naulet:rousseau:main},
$\Psi(\spars) = K(\spars) + H(\spars,\tilde{\varepsilon}(\spars),\tilde{u}(\spars))$,
$\partial_{\varepsilon}H(\spars,\tilde{\varepsilon}(\spars),\tilde{u}(\spars)) = 0$ and
$\partial_{u}H(\spars,\tilde{\varepsilon}(\spars),\tilde{u}(\spars)) = 0$. It follows that
$\Psi'(\spars) = K'(\spars) + \partial_{\spars}H(\spars,\tilde{\varepsilon}(\spars),\tilde{u}(\spars))$.
We now write $\tilde{\varepsilon} \equiv \tilde{\varepsilon}(\spars)$ and $\tilde{u} \equiv \tilde{u}(\spars)$
to ease the notations. Then,
\begin{align*}
  \Psi''(\spars)%
  &= K''(\spars)%
    + \partial_{\spars}^2H(\spars,\tilde{\varepsilon},\tilde{u})%
    +
    \partial_{\varepsilon}\partial_{\spars}H(\spars,\tilde{\varepsilon},\tilde{u})
    \tilde{\varepsilon}'%
    +
    \partial_u\partial_{\spars}H(\spars,\tilde{\varepsilon},\tilde{u}) \tilde{u}'.
\end{align*}
By definition, $\tilde{u} = \bar{u}(\spars,\tilde{\varepsilon})$, so that
$\tilde{u}' = \partial_{\spars}\bar{u}(\spars,\tilde{\varepsilon}) +
\partial_{\varepsilon}\bar{u}(\spars,\tilde{\varepsilon}) \tilde{\varepsilon}'$. From
the fact that $\partial_uH(\spars,\varepsilon,\bar{u}(\spars,\varepsilon)) = 0$, we
deduce that
$\partial_{\spars}\partial_uH(\spars,\varepsilon,\bar{u}(\spars,\varepsilon)) +
\partial_u^2H(\spars,\varepsilon,\bar{u}(\spars,\varepsilon))
\partial_{\spars}\bar{u}(\spars,\varepsilon) = 0$, and
$\partial_{\varepsilon}\partial_uH(\spars,\varepsilon,\bar{u}(\spars,\varepsilon)) +
\partial_u^2H(\spars,\varepsilon,\bar{u}(\spars,\varepsilon))
\partial_{\varepsilon}\bar{u}(\spars,\varepsilon) = 0$; that is,
\begin{equation}
  \label{eq:concent:107}
  \tilde{u}'%
  = \frac{\partial_{\spars}\partial_uH(\spars,\tilde{\varepsilon},\tilde{u})
  }{-\partial_u^2H(\spars,\tilde{\varepsilon},\tilde{u})}%
  +
  \frac{\partial_{\varepsilon}\partial_uH(\spars,\tilde{\varepsilon},\tilde{u})}{-\partial_u^2
    H(\spars,\tilde{\varepsilon},\tilde{u})} \tilde{\varepsilon}'.
\end{equation}
It follows,
\begin{multline*}
  \Psi''(\spars)%
  = K''(\spars)%
  + \partial_{\spars}^2H(\spars,\tilde{\varepsilon},\tilde{u})%
  + \frac{\{\partial_{\spars}\partial_uH(\spars,\tilde{\varepsilon},\tilde{u})
    \}^2 }{-\partial_u^2H(\spars,\tilde{\varepsilon},\tilde{u}) }\\%
  +\Big\{%
  \partial_{\varepsilon}\partial_{\spars}H(\spars,\tilde{\varepsilon},\tilde{u})+
  \frac{\partial_u\partial_{\spars}H(\spars,\tilde{\varepsilon},\tilde{u})\cdot
    \partial_{\varepsilon}\partial_uH(\spars,\tilde{\varepsilon},\tilde{u})}{-\partial_u^2
    H(\spars,\tilde{\varepsilon},\tilde{u})} \Big\}\tilde{\varepsilon}'.
\end{multline*}
By definition
$K'(\spars) = -\frac{\nver^2}{\nedg}\frac{1}{\beta_{\spars}} +
\sparsfunc'(\spars)%
+ \frac{\nver}{2} \frac{1}{\beta_{\spars}}$, and hence
$K''(\spars) = \partial_\spars^2 \zeta(\params)+ o(\nver)$. It follows using the
estimates established in \cref{pro:concent:10},
\begin{align*}
  \Psi''(\spars)%
  &= \sparsfunc''(\spars)%
    - \frac{\nver(1+o(1))}{\spars^2}%
    - \frac{\nedg(1+o(1))\log \frac{\nedg}{\nver}}{1-\spars} \tilde{\varepsilon}'
    + o(\nver).
\end{align*}
We obtain an estimate on $\tilde{\varepsilon}'$ by using that
$\partial_{\varepsilon}H(\spars,\tilde{\varepsilon},\tilde{u}) = 0$. Differentiating
both sides of the equation with respect to $\spars$ gives
$\partial_{\spars}\partial_{\varepsilon}H(\spars,\tilde{\varepsilon},\tilde{u}) +
\partial_{\varepsilon}^2H(\spars,\tilde{\varepsilon},\tilde{u})\tilde{\varepsilon}' +
\partial_u\partial_{\varepsilon}H(\spars,\tilde{\varepsilon},\tilde{u})\tilde{u}' =
0$. Hence, by \cref{eq:concent:107}, and then by \cref{pro:concent:10}, since
$\tilde{\varepsilon} \lesssim [\nver/\nedg]^{1/(1-\spars)}$ and $0<\spars<1$,
\begin{align}
  \label{eq:concent:73}
  \tilde{\varepsilon}'%
  &= - \frac{\partial_{\spars}\partial_{\varepsilon}H(\spars,\tilde{\varepsilon},\tilde{u})
    +
    \frac{\partial_u\partial_{\varepsilon}H(\spars,\tilde{\varepsilon},\tilde{u})\cdot
    \partial_{\spars}\partial_uH(\spars,\tilde{\varepsilon},\tilde{u})}{-\partial_u^2
    H(\spars,\tilde{\varepsilon},\tilde{u})}}{\partial_{\varepsilon}^2H(\spars,\tilde{\varepsilon},\tilde{u})%
    +
    \frac{\{\partial_u\partial_{\varepsilon}H(\spars,\tilde{\varepsilon},\tilde{u})\}^2}{
    - \partial_u^2H(\spars,\tilde{\varepsilon},\tilde{u}) }}
  = - \frac{\tilde{\varepsilon} (1+o(1)) \log \frac{\nedg}{\nver}}{(1-\spars)^2}.
\end{align}
It follows,
\begin{align*}
  \Psi''(\spars)%
  &=\partial_\spars^2 \zeta(\params)- \frac{\nver(1+o(1))}{\spars^2}%
    +  \nver\frac{\tilde{\varepsilon} \frac{\nedg}{\nver}\log^2
    \frac{\nedg}{\nver} }{(1-\spars)^3} + o(\nver)=\partial_\spars^2 \zeta(\params)- \frac{\nver(1+o(1))}{\spars^2} <0
\end{align*}
since $\sparsfunc''(\spars)<0$.

\begin{lemma}
  \label{pro:concent:10}
  Let $\spars \in [c_1,c_2]$, and let
  $\tilde{\varepsilon}\equiv \tilde{\varepsilon}(\spars)$ and
  $\tilde{u} \equiv \tilde{u}(\spars)$ as given in \cite[\cref{main-pro:concent:5}]{caron:naulet:rousseau:main}. Then, the
  following estimates are true.
  \begin{enumerate}
    \item\label{item:concent:1}
    $\partial_{\spars}\log f(\spars,\tilde{\varepsilon}) = -\frac{1+o(1)}{\spars}$; and
    $\partial_{\spars}^2 \log f(\spars,\tilde{\varepsilon}) = \frac{1+o(1)}{\spars^2}$.

    \item\label{item:concent:8}
    $\tilde{u} = \frac{\spars \nver}{\nedg}(1+o(1))$; and
    $\tilde{\varepsilon}^{1-\spars} = \frac{\spars \nver}{\nedg}(1+o(1))$.

    \item\label{item:concent:2}
     $$\partial_{\spars}^2H(\spars,\tilde{\varepsilon},\tilde{u})
    \sim - \frac{2\nver}{\spars^2}, \quad \partial_u^2 H(\spars,\tilde{\varepsilon},\tilde{u}) \sim - \frac{\nedgsq}{\spars^2\nver}, \quad \partial_{\varepsilon}^2H(\spars,\tilde{\varepsilon},\tilde{u}) \sim - (1-\spars)\nedg
    \tilde{\varepsilon}^{-1}.$$%

    \item\label{item:concent:4}
    $\partial_{\spars}\partial_u H(\spars,\tilde{\varepsilon},\tilde{u}) = \frac{\nedg(1+o(1))}{\spars^2}$.

    \item\label{item:concent:5}
    $\partial_{\varepsilon}\partial_{\spars}H(\spars,\tilde{\varepsilon},\tilde{u}) = - \frac{\nedg(1+o(1))\log \frac{\nedg}{\nver}}{1-\spars}$.

    \item\label{item:concent:6}
    $\partial_{\varepsilon}\partial_{u}H(\spars,\tilde{\varepsilon},\tilde{u}) = \frac{\nedgsq(1+o(1))}{\spars \nver}$.
  \end{enumerate}
\end{lemma}
\begin{proof}
  \Cref{item:concent:1}. It follows by definition of $f$, whenever $\varepsilon = o(1)$,
  and because $\spars \in (c_2,c_1)$,
  \begin{align*}
    \partial_{\spars}\log f(\spars,\varepsilon)%
    &=
      \frac{\partial_{\spars}f(\spars,\tilde{\varepsilon})}{\partial_{\spars}f(\spars,\tilde{\varepsilon})}%
      = -\frac{1}{\spars} - \frac{\tilde{\varepsilon}^{\spars}\log(\tilde{\varepsilon})}{1 -
      \varepsilon^{\spars}}%
      = - \frac{1 + o(1)}{\spars}.
  \end{align*}
  From the previous computation,
  \begin{align*}
    \partial_{\spars}^2 \log f(\spars,\varepsilon)%
    &= \frac{1}{\spars^2}%
      - \frac{\varepsilon^{\spars}\log^2(\varepsilon)}{(1 - \varepsilon^{\spars})^2}%
      = \frac{1 + o(1)}{\spars^2}.
  \end{align*}
  \Cref{item:concent:8}. We already know that
  $\tilde{u} = \frac{\nver}{\nedg} \frac{1+o(1)}{f(\spars,\tilde{\varepsilon})}$,
  which follows for instance from \cref{pro:concent:3}, as $\tilde{\varepsilon}=o(1)$ and
  $\frac{1}{2} + \beta_{\spars}g(\spars,\tilde{\varepsilon}) =
x1  f(\spars,\tilde{\varepsilon})(1+o(1))$. Then it is obvious that
  $f(\spars,\tilde{\varepsilon}) = (1/\spars)(1+o(1))$ as $\spars \in
  (c_2,c_1)$. The second claim follows from
  $\partial_{\varepsilon}H(\spars,\tilde{\varepsilon},\tilde{u}) = 0$.
  \Cref{item:concent:2}. By definition of $H$,
  $\partial_{\spars}^2H(\spars,\varepsilon,u) = - \nedg \beta_{\spars} u
  \partial_{\spars}^2g(\spars,\varepsilon)%
  + 2\nver u \partial_{\spars}g(\spars,\varepsilon)$. Clearly
  $\partial_{\spars}g(\spars,\varepsilon) = \partial_{\spars}f(\spars,\varepsilon)$,
  and
  $\partial_{\spars}^2\log f(\spars,\varepsilon) =
  \frac{\partial_{\spars}^2f(\spars,\varepsilon)}{f(\spars,\varepsilon)} -
  \frac{\{\partial_{\spars}f(\spars,\varepsilon)\}^2}{f(\spars,\varepsilon)^2}$. It
  follows,
  \begin{align*}
    \partial_{\spars}^2H(\spars,\tilde{\varepsilon},\tilde{u})%
    &= -\nedg \beta_{\spars}\tilde{u} f(\spars,\tilde{\varepsilon})\partial_{\spars}^2\log
      f(\spars,\tilde{\varepsilon})%
      - \nedg \beta_{\spars}\tilde{u}
      \frac{\{\partial_{\spars}f(\spars,\tilde{\varepsilon})\}^2}{f(\spars,\tilde{\varepsilon})}%
      + o(\nver)\\
    &= -\nedg \beta_{\spars}u f(\spars,\varepsilon)\partial_{\spars}^2\log
      f(\spars,\varepsilon)%
      - u f(\spars,\varepsilon)\nedg\beta_{\spars}\{\partial_{\spars}\log
      f(\spars,\varepsilon)\}^2 + o(\nver).
  \end{align*}
  Now, we have that
  $\tilde{u}f(\spars,\tilde{\varepsilon}) = \frac{\nver}{\nedg}(1+o(1))$, and then
  $\partial_{\spars}^2 H(\spars,\tilde{\varepsilon},\tilde{u}) =
  -\frac{2\nver(1+o(1))}{\spars^2}$ by \cref{item:concent:1}.  Since,
  $\tilde{u} = \frac{\nver}{\nedg}\frac{1+o(1)}{f(\spars,\tilde{\varepsilon})} =
  \frac{\spars \nver}{\nedg}(1+o(1))$, and
  $\partial_u^2H(\spars,\tilde{\varepsilon},\tilde{u}) = -
  \frac{\nver}{\tilde{u}^2} + \frac{\nedg}{2}\frac{1}{(1+\tilde{u})^2} =
  -\frac{\nver}{\tilde{u}^2}(1+o(1))$ we obtain the second term of  \cref{item:concent:2}.
Moreover $\partial_{\varepsilon}^2H(\spars,\varepsilon,u) = -(1-\spars)\nedg
  \beta_{\spars} u \varepsilon^{-2+\spars} -
  \frac{\nedg}{2}\frac{1}{(1-\varepsilon)^2}$. But we know that $\tilde{\varepsilon}=
  o(1)$ and that $\tilde{\varepsilon}^{-1+\spars} =
  \frac{1}{\tilde{u}}(1+o(1))$. Then,
  \begin{align*}
    \partial_{\varepsilon}^2H(\spars,\tilde{\varepsilon},\tilde{u})%
    &= - (1-\spars)\nedg \tilde{\varepsilon}^{-1}(1+o(1))%
      - \frac{\nedg(1+o(1))}{2}.
  \end{align*}
  Since $\tilde{\varepsilon} = o(1)$ and $\spars \leq c_1 < 1$, the result follows.

  \Cref{item:concent:4}, we have that
  $\partial_u\partial_{\spars}H(\spars,\varepsilon,u) = -\nedg \beta_{\spars} \partial_{\spars}g(\spars,\varepsilon) + \nver g(\spars,\varepsilon)$,
  and since $\partial_{\spars}g(\spars,\varepsilon) = \partial_{\spars}f(\spars,\varepsilon)$,
  $f(\spars,\tilde{\varepsilon}) = (1+o(1))/\spars$ and
  $g(\spars,\tilde{\varepsilon}) = (\frac{1}{\spars} - \frac{1}{2})(1+o(1))$, we deduce
  that
  \begin{align*}
    \partial_u\partial_{\spars}H(\spars,\tilde{\varepsilon},\tilde{u})%
    &= -\nedg(1+o(1)) f(\spars,\tilde{\varepsilon}) \partial_{\spars}\log
      f(\spars,\tilde{\varepsilon})%
      + \frac{\nver}{\spars}(1+o(1)) - \frac{\nver}{2}(1+o(1))\\
    &= \frac{\nedg(1+o(1))}{\spars^2}.
  \end{align*}
  \cref{item:concent:5}. By definition of $H$,
  $\partial_{\varepsilon}\partial_{\spars}H(\spars,\varepsilon,u) = \nedg
  \beta_{\spars}u \varepsilon^{-1+\spars}\log(\varepsilon) - \nver u
  \varepsilon^{-1+\spars} + \frac{\nver}{2}$. But
  $\tilde{\varepsilon}^{-1+\spars} = \tilde{u}^{-1}(1+o(1))$ by \cref{item:concent:8}, so
  that,
  \begin{align*}
    \partial_{\varepsilon}\partial_{\spars}H(\spars,\tilde{\varepsilon},\tilde{u})%
    &= \nedg(1+o(1))\log(\tilde{\varepsilon}) - \frac{\nver(1+o(1))}{2}\\
    &= \frac{\nedg(1+o(1))}{1 - \spars}\log \tilde{u}%
      - \frac{\nver(1+o(1))}{2}\\
    &= - \frac{\nedg(1+o(1))\log \frac{\nedg}{\nver} }{1-\spars}.
  \end{align*}
  \Cref{item:concent:6}. We have $\partial_{\varepsilon}\partial_uH(\spars,\varepsilon,u) =
  \nedg \beta_{\spars}\big(\varepsilon^{-1+\spars} - \frac{1}{2} \big)$. We already
  know that $\tilde{\varepsilon}^{-1+\spars} = \frac{1}{\tilde{u}}(1+o(1)) =
  \frac{\nedg}{\spars \nver}(1+o(1))$. Hence the result.
\end{proof}

\subsection{Proof of \texorpdfstring{\cite[\cref{main-pro:concent:16}]{caron:naulet:rousseau:main}}{Lemma \ref{main-pro:concent:16}}}
\label{sec:concent:asympt-expans-mle}

We first prove \cref{main-item:pro:concent:16:1}. From
\cite[\cref{main-pro:concent:4}]{caron:naulet:rousseau:main}, we need only prove
$$\Psi(\spars)  -\sup  \Psi \leq -K \log\nedg, \quad \text{for } \quad \spars \in [-C,c_2]^c$$
and from \cite[\cref{main-pro:concent:11}]{caron:naulet:rousseau:main}, the function $\Psi$ is concave on $[c_1,c_2]$, $\Psi(\spars) \leq \Psi(c_1)$ if $\spars \leq c_1$ and for any $\eta = o(1)$, and any $\spars \in [c_1,c_2]$ such that
$|\spars - \qMLEspars| > \eta$,
\begin{align*}
  \Psi(\spars) - \sup \Psi
  &\leq \max\{ \Psi(\qMLEspars + \eta) - \Psi(\qMLEspars),\, \Psi(\qMLEspars -
    \eta) - \Psi(\qMLEspars) \}\\
  &= \frac{1}{2}\Psi''(\qMLEspars)(1+o(1))\eta^2\\
  &= -\frac{1 + o(1)}{2}\Big(- \sparsfunc''(\qMLEspars) + \frac{\nver}{\qMLEspars^2}\Big)\eta^2\\
  & \leq  -K\log(\nedg),
\end{align*}
choosing $\eta^2 = 4K \alpha_0\log(\nedg)/\nver$  since $-\sparsfunc''(\qMLEspars) > 0$,
when $\TrueSize$ is large enough.

Consider now $| \spars - \qMLEspars| \leq C \sqrt{\log(\nedg)/\nver}$ and $|\varepsilon - \qMLEeps| \geq C \sqrt{\log (\nedg)}/(\nedg)^{3/4}$.
We know from \cref{pro:concent:3} that
$\tilde{H}_{\spars}(\varepsilon) \coloneqq \sup_uH(\spars,u,\varepsilon) =
H(\spars,\varepsilon,\bar{u}(\spars,\varepsilon))$. Furthermore, we proved in
\cite[\cref{main-pro:concent:5}]{caron:naulet:rousseau:main} that $\tilde{H}_{\spars}$ attains its unique maximum at
$\tilde{\varepsilon}(\spars)$, and that $\tilde{H}_{\spars}''(\varepsilon) < 0$ for
all $\varepsilon \in (0,1)$ whenever $\spars \approx \qMLEspars \in
[c_1,c_2]$. Concretely, $\sup_uH(\spars,\varepsilon,u)$ is a concave function
attaining its maximum at $\tilde{\varepsilon}(\spars)$, and thus for all
$\varepsilon \in (0,1)$ such that
$|\varepsilon - \tilde{\varepsilon}(\spars)| > (C/2)\sqrt{\log \nedg}/ (\nedg)^{3/4} \eqqcolon \eta_t$
\begin{align*}
  \tilde{H}_{\spars}(\varepsilon) - \tilde{H}_{\spars}(\tilde{\varepsilon}(\spars))
  &\leq \frac{1}{2}\sup_{|x| \leq \eta_t}\tilde{H}_{\spars}''(\tilde{\varepsilon}(\spars) +
    x)\frac{ 4C^2 \log \nedg }{ (\nedg)^{3/2} },
\end{align*}
and,
\begin{align*}
  \sup_{|x|\leq \eta_t}\tilde{H}_{\spars}''(\tilde{\varepsilon}(\spars) + x)%
  &\lesssim - \nedg \cdot
    \frac{\nver}{\nedg}\tilde{\varepsilon}(\spars)^{-2+\spars}
  \asymp - \nedg \cdot  \frac{1}{\tilde{\varepsilon}(\spars)}. 
\end{align*}
Since $|\spars- \qMLEspars| \leq C\sqrt{\log (\nedg)/\nver}$, the
\cref{pro:concent:10} together with \cite[\cref{main-qMLeps}]{caron:naulet:rousseau:main} imply
$$\tilde{\varepsilon}(\spars) \asymp [\nver/\nedg]^{1/(1-\qMLEspars)} \asymp [\nver/\nedg]^{1/(1-\alpha_0)} \asymp 1/\sqrt{\nedg}$$
so that
$
\sup_{|x|\leq \eta_t}\tilde{H}_{\spars}''(\tilde{\varepsilon}(\spars) + x)%
\lesssim -\sqrt{\nedg},
$
which implies that
\begin{align*}
  \tilde{H}_{\spars}(\varepsilon) - \tilde{H}_{\spars}(\tilde{\varepsilon}(\spars))
  &\lesssim - C \log \nedg.
\end{align*}
Moreover using a Taylor expansion of $\tilde{\varepsilon}(\spars)$, there exists
$\bar{\spars} \in (\spars,\qMLEspars)$ such that
$\tilde{\varepsilon}(\spars) - \tilde{\varepsilon}(\qMLEspars) =
\tilde{\varepsilon}'(\bar{\spars})(\spars - \qMLEspars)$.
Using the approximation of $  \tilde{\varepsilon}'(\spars)$ in \cref{eq:concent:73}, we obtain
\begin{align*}
  |\tilde{\varepsilon}(\spars) - \tilde{\varepsilon}(\qMLEspars)|%
  &= \frac{\tilde{\varepsilon}(\bar{\spars})(1+o(1)) \log
    \frac{\nedg}{\nver}}{(1-\bar{\spars})^2}|\spars - \qMLEspars|
    \lesssim \frac{1}{\sqrt{\nedg}}\cdot  \log
    \frac{\nedg}{\nver} \cdot  \sqrt{\frac{\log(\nedg)}{\nver}}=o((\nedg)^{-3/2}).
\end{align*}
Hence, if $|\varepsilon - \tilde{\varepsilon}(\qMLEspars) |\geq  C\sqrt{\log \nedg}/ (\nedg)^{3/4} $ then $|\varepsilon - \tilde{\varepsilon}(\qMLEspars) |\geq C\sqrt{\log \nedg}/ (\nedg)^{3/4} $ and for all $|\spars - \qMLEspars| \leq C \sqrt{\log (\nedg)/\nver})$ and all $u>0$,
\begin{equation*}
  \sup_{|\spars - \qMLEspars| \leq C \sqrt{\log (\nedg)/\nver}, u>0} \qloglik^{*}( \sigma, \varepsilon,u) - \sup \Psi \leq -K \log \nedg.
\end{equation*}

Finally we consider $|\spars - \qMLEspars| \leq C \sqrt{\log (\nedg)/\nver})$,
$|\varepsilon - \tilde{\varepsilon}(\qMLEspars) |\leq C\sqrt{\log \nedg}/ (\nedg)^{3/4} $ and
$| u - \qMLEu| > C\sqrt{\nver \log \nedg}/\nedg$. Recall that
$\sup_uH(\spars,u,\varepsilon) = H(\spars,\varepsilon,\bar{u}(\spars,\varepsilon))$. Furthermore, from the
proof of \cref{pro:concent:3} we know that $\partial_uH(\spars,\varepsilon,u) < 0$ if
$u > 3\nver/\nedg$, and that $\partial_u^2H(\spars,\varepsilon,u) = -\frac{\nver}{u^2}(1+o(1))$
for all $0 < u \leq 3\nver/\nedg$. Therefore, when $u > 3\nver/\nedg$,
\begin{align*}
  H(\spars,\varepsilon,u) - \sup_u H(\spars,\varepsilon,u)%
  &\leq H(\spars,\varepsilon,3\nver/\nedg) - \sup_u H(\spars,\varepsilon,u) \\
  &= \frac{ \partial_u^2H(\spars,\varepsilon,\bar{u}(\spars,\varepsilon))\Big(
    \frac{3\nver}{\nedg} - \bar{u}(\spars,\varepsilon) \Big)^2}{2}\\
  &\lesssim - \frac{\nver}{(\nver/\nedg)^2} \frac{\nver^2}{\nedgsq}\lesssim  - \nver,
\end{align*}
and whenever $0 < u \leq 3\nver/\nedg$,
\begin{align*}
  H(\spars,\varepsilon,u) - \sup_u H(\spars,\varepsilon,u)%
  &\leq - \frac{1}{2}\frac{\nver}{(3\nver/\nedg)^2}(u -
    \bar{u}(\spars,\varepsilon))^2\\
  &\leq -\frac{1}{18}\frac{\nedgsq}{\nver}(u - \bar{u}(\spars,\varepsilon))^2.
\end{align*}
To conclude we thus only need to prove that
$| u - \qMLEu| > C\sqrt{\nver \log \nedg}/\nedg \implies | u - \bar{u}(\spars,\varepsilon)| >(C/2)\sqrt{\nver \log \nedg}/\nedg$.
It is enough to show that
$|\qMLEu - \bar{u}(\sigma,\varepsilon)| \leq (C/2)\sqrt{\nver \log \nedg}/\nedg$. But
$\qMLEu = \bar{u}(\qMLEspars,\qMLEeps)$, and thus by a Taylor expansion, we can
find $\bar{\spars} \in (\spars,\qMLEspars)$ and $\bar{\varepsilon} \in (\varepsilon,\qMLEeps)$ such that
$\bar{u}(\spars,\varepsilon) - \bar{u}(\qMLEspars,\qMLEeps) = \partial_{\spars}\bar{u}(\bar{\spars},\bar{\varepsilon})(\spars - \qMLEspars) + \partial_{\varepsilon}\bar{u}(\bar{\spars},\bar{\varepsilon})(\varepsilon - \qMLEeps)$.
That is (see for instance the proof of \cite[\cref{main-pro:concent:11}]{caron:naulet:rousseau:main} for details,
and also \cref{pro:concent:10}),
\begin{align*}
  \bar{u}(\spars,\varepsilon) - \bar{u}(\qMLEspars,\qMLEeps)%
  &=
    \frac{\partial_{\spars}\partial_uH(\bar{\spars},\bar{\varepsilon},\bar{u}(\bar{\spars},\bar{\varepsilon}))}{
    -\partial_u^2H(\bar{\spars},\bar{\varepsilon},\bar{u}(\bar{\spars},\bar{\varepsilon}))
    } (\spars - \qMLEspars)%
    + \frac{\partial_{\varepsilon}\partial_uH(\bar{\spars},\bar{\varepsilon},\bar{u}(\bar{\spars},\bar{\varepsilon}))}{
    -\partial_u^2H(\bar{\spars},\bar{\varepsilon},\bar{u}(\bar{\spars},\bar{\varepsilon}))
    } (\varepsilon - \qMLEeps)\\
  &= O\Big( \frac{\nver}{\nedg} \sqrt{\frac{\log(\nedg)}{\nver}}\Big)%
    + O\Big( \sqrt{\frac{\log(\nedg)}{(\nedg)^{3/2}}} \Big)\\
  &= O\Big( \frac{\sqrt{\nver \log (\nedg)}}{\nedg}\Big),%
\end{align*}
which concludes the proof of \cref{main-item:pro:concent:16:1} by choosing $C$ large enough.

We now prove \cref{main-item:pro:concent:16:2}. Recall that
$\varphi(\spars,\tau,\size) = (\spars,
\varepsilon(\spars,\tau,\size),u(\spars,\tau,\size))$ and
$$ U_t(C) = \Set*{\params \given%
  \textstyle%
  |\spars - \qMLEspars|^2 \leq \frac{C\log(\nedg)}{\nver},\ %
  |\varepsilon(\params) - \qMLEeps|^2 \leq
  \frac{C\log(\nedg)}{(\nedg)^{3/2}},\ %
  |u(\params) - \qMLEu|^2 \leq \frac{C \nver \log(\nedg)}{\nedgsq}%
}
$$
The maximum belongs to $U_t(C)$ from \cref{main-item:pro:concent:16:1}.  We now prove that $\varphi$ is
invertible on $U_t$. To do so we compute the Jacobian of the transformation. We first compute $   \nabla \zeta(\params)$. Using \cite[\cref{main-eq:7}]{caron:naulet:rousseau:main} which defined $\zeta$,
\begin{align}
  \notag
  \nabla \zeta(\params)%
  &= \frac{2\zeta(\params)}{2\zeta(\params)^2 - \tau \zeta(\params) - 2\spars
    \size \zeta(\params)^{\spars}}%
    \begin{pmatrix}
      \size \zeta(\params)^{\spars} \log \zeta(\params) - \nver\\
      \zeta(\params)/2\\
      \zeta(\params)^{\spars}
    \end{pmatrix}\\%
  \label{eq:concent:17}
  &=
    \frac{1}{\zeta(\params)}
    \begin{pmatrix}
      \size \zeta(\params)^{\spars} \log \zeta(\params) - \nver\\
      \zeta(\params)/2\\
      \zeta(\params)^{\spars}
    \end{pmatrix}(1+o(1)), \quad \forall \params \in U_t(C).
\end{align}
Next, since $\tau = \zeta(\params) \varepsilon(\params)$,
\begin{align}
  \label{eq:concent:16}
  \nabla \varepsilon(\params)%
  &= \frac{1}{\zeta(\params)}%
    \begin{pmatrix}
      0\\
      1\\
      0
    \end{pmatrix}%
  - \frac{\varepsilon(\params)}{\zeta(\params)}\nabla \zeta(\params),
\end{align}
and, from the fact that $\nedg(1 - \spars\nver/\nedg)u(\params) = \size
\zeta(\params)^{\spars}$, when $\phi \in U_t$,
\begin{align}
  \label{eq:concent:25}
  \nabla u(\params)%
  &= \frac{1}{\nedg(1-\spars\nver/\nedg)}\Bigg\{%
    \begin{pmatrix}
      \nver u(\params) + \size \zeta(\params)^{\spars} \log\zeta(\params)\\
      0\\
      \zeta(\params)^{\spars}
    \end{pmatrix}%
  +%
  \spars\size \zeta(\params)^{-1+\spars}\nabla \zeta(\params)
  \Bigg\}.
\end{align}
Now we define,
\begin{equation}
  \label{eq:concent:30}
  J_{*}(\params)%
  \coloneqq%
  \begin{pmatrix}
    1 & 0 & 0\\
    0 & \frac{1}{\zeta(\params)} & 0\\
    \frac{\nver u(\params) + \size \zeta(\params)^{\spars}
      \log\zeta(\params)}{\nedg \beta_{\spars}}%
    & 0 & \frac{\zeta(\params)^{\spars}}{\nedg \beta_{\spars}}%
  \end{pmatrix}
  ,
\end{equation}
as well as,
\begin{equation}
  \label{eq:concent:31}
  E(\params)%
  \coloneqq%
  \begin{pmatrix}
    0 & 0 & 0\\
    - \frac{\varepsilon(\params)}{\zeta(\params)^2}(\size
    \zeta(\params)^{\spars}\log\zeta(\params) - \nver)%
    & - \frac{\varepsilon(\params)}{\zeta(\params)^2}
    \frac{\zeta(\params)}{2}%
    & - \frac{\varepsilon(\params)}{\zeta(\params)^2}
    \zeta(\params)^{\spars}\\
    \frac{\spars \size \zeta(\params)^{-1+\spars}}{\nedg\beta_{\spars}\zeta(\params)} (\size
    \zeta(\params)^{\spars}\log\zeta(\params) - \nver)%
    & \frac{\spars \size \zeta(\params)^{-1+\spars}}{\nedg\beta_{\spars}} \frac{1}{2}
    &\frac{\spars \size \zeta(\params)^{-1+\spars}}{\nedg\beta_{\spars}\zeta(\params)} \zeta(\params)^{\spars}%
  \end{pmatrix}%
  .
\end{equation}
Likewise, by \cref{eq:concent:17,eq:concent:16,eq:concent:25,eq:concent:30,eq:concent:31}, we have that the Jacobian
matrix of $\varphi$ is given by $J(\params) \sim J_{*}(\params) + E(\params)$
as $t \to \infty$, at least when $\params \in U_t$. We further remark that
\begin{equation*}
  J_{*}(\params)^{-1}
  =%
  \begin{pmatrix}
    1 & 0 & 0\\
    0 & \zeta(\params) & 0\\
    - \frac{\nver u(\params)}{\zeta(\params)^{\spars}}%
    + \size \log \zeta(\params)%
    & 0 & \frac{\nedg\beta_{\spars}}{\zeta(\params)^{\spars}},
  \end{pmatrix}
\end{equation*}
and, When $\params \in U_t$, we can find constants
$c_1,\dots,c_6 \in \Reals$, such that asymptotically as $t \to \infty$,
\begin{equation*}
  E(\params)%
  \sim%
  \begin{pmatrix}
    0 & 0 & 0\\
    c_1\frac{\nver \log(\nedg)}{(2\nedg)^{3/2}}%
    & c_2\frac{1}{2\nedg}%
    & c_3\frac{\nver}{4\nedgsq}\\
    c_4\frac{\nver^2\log(\nedg)}{4\nedgsq}%
    & c_5\frac{\nver}{(2\nedg)^{3/2}}%
    & c_6\frac{\nver^2}{(2\nedg)^{5/2}}%
  \end{pmatrix}
  .
\end{equation*}
The constants $c_1,\dots,c_6$ depend uniquely on
$\alpha_0$ and $\tau_{*}$ and can be made explicit. We choose to not do it
since this is not needed for our purpose. Similarly, we have $\zeta(\params)
\sim \sqrt{2\nedg}$, and for some constant $c_7 \in \Reals$, we also have
that $\frac{\nedg \beta_{\spars}}{\zeta(\params)^{\spars}}\sim c_7
\frac{(2\nedg)^{3/2}}{\nver}$. Thus,
\begin{equation*}
  J_{*}(\params)^{-1}E(\params)%
  \sim%
  \begin{pmatrix}
    0 & 0 & 0\\
    c_1\frac{\nver \log(\nedg)}{2\nedg}%
    & \frac{c_2}{\sqrt{2\nedg}}%
    & c_3\frac{\nver}{(2\nedg)^{3/2}}\\
    c_4c_7 \frac{\nver \log(\nedg)}{\sqrt{2\nedg}}%
    & c_5 c_7%
    & c_6c_7 \frac{\nver}{2\nedg}%
  \end{pmatrix}
  .
\end{equation*}
The eigenvalues  $(\lambda_1,\lambda_2,\lambda_3)$ of
$J_{*}(\params)^{-1}E(\params)$ go to zero with $\TrueSize$  under \cite[\cref{main-ass:concent:degdist}]{caron:naulet:rousseau:main} and $\lambda_1 = 0$. Then, by a
Neumann series expansion of the inverse of $J(\params) = J_{*}(\params)(I +
J_{*}(\params)^{-1}E(\params))$, we obtain that whenever $\params \in U_t$,
\begin{equation*}
  J(\params)^{-1}
  \sim
  \begin{pmatrix}
    1 & 0 & 0\\
    0  &\sqrt{2\nedg}        &  0 \\
    c_8 \sqrt{2\nedg}  \log( 2 \nedg)
    & 0
    & c_7 \frac{(2\nedg)^{3/2}}{ \nver}
  \end{pmatrix}.
\end{equation*}

\section{Local analysis of the likelihood}
\label{sec:local-analys-likel}

\subsection{Computation of $\hat \Sigma_t $ }
\label{sec:Sigmat}

In the whole \cref{sec:local-analys-likel}, we define for convenience
$\qloglik^*(\spars, \tau, u ) = \qloglik(\spars, \tau,\KLsize u)$. With this definition in
mind, we recall that by definition
$\Sigma_t(\spars, \tau, u)^{-1} = - D^2\qloglik^*(\spars, \tau, u )$. From \cite[\cref{main-qMLspars,main-ass:concent:degdist}]{caron:naulet:rousseau:main},
$| \qMLEspars - \alpha_0| = o(1/\log \nver) $, $| \qMLEtau - \KLtau |= o(1)  $ and
$|\qMLEsize -\KLsize| =o(\KLsize)$. Now let $ |\params_u - \hat{\params_u}| \leq \nver^{-\delta}$ with $\delta>0$, \cref{lem:Sigmat} implies that 
 $\Sigma_t (\spars, \tau, \size/\KLsize)^{-1} $ can be written as 
 $$\left(- \sparsfunc''(\spars) +   \frac{ \KLsize  (2\nedg)^{\spars/2} }{ \alpha_0^3} (1 +o(1) )\right) \text{diag}(1,0,0) +
 \left( \begin{array}{ccc}
 M_1(\params) & O(\KLsize) & M_2(\params)\\
O(\KLsize) & \size (1 -\sigma)\tau^{\sigma-2} + O(1)  & O(\KLsize) \\
M_2(\params) &  O(\KLsize) &  \frac{ \nver }{ u^2 } +O(\KLsize^{2\alpha_0+2\epsilon} ) \end{array}\right)
 $$
where 
$$ M_1(\params) \coloneqq  u \KLsize \frac{ \zeta(\params)^{\spars} \log^2 ( \zeta(\params) ) }{2\spars } \left( 1 - \frac{ 1  }{ \spars\log (\zeta(\params)) } \right)^2 ,$$
and
$$ M_2(\params) \coloneqq \KLsize \frac{ \zeta(\params)^{\spars} \log ( \zeta(\params) ) }{ \spars } \left( 1 - \frac{ 1  }{ \spars\log (\zeta(\params)) } \right)(1 + O( \KLsize^{-\delta_1} ).$$
Let $X = (x_1, x_2, x_3)^t $. We bound
$X^t \Sigma_t (\spars, \tau, \size/\KLsize)^{-1} X $ using that for $i\neq j $ it holds
$2|a_{ij} x_ix_j| = \inf_{b > 0}\{ a_{ij}^2 x_i^2 b + x_j^2/b\}$. In particular we apply the above  $i,j = 1,2$ with $a_{12} = O(\KLsize) $ and $b = \KLsize^{a-1} $ for $a < \alpha_0$ so that
$$2|\KLsize x_1x_2| \leq x_1^2 \KLsize^{1 + a} + x_2^2\KLsize^{1-a}$$
and similarly for $i,j=2,3$
$$2|\KLsize x_2x_3| \leq x_3^2 \KLsize^{1 + a} + x_2^2\KLsize^{1-a}.$$
which implies that 
\begin{align*}
  X^t \Sigma_t (\spars, \tau, \size/\KLsize)^{-1} X%
  &\leq
    x_1^2 \left[- \sparsfunc''(\spars) +   \frac{ \KLsize  (2\nedg)^{\spars/2} }{ \alpha_0^3} (1 +o(1) )\right]\\
  &\quad+ x_2^2 \KLsize (1 -\alpha_0)\KLtau^{\alpha_0-2} (1 +o(1)) + 2x_3^2 \KLsize^{1 + a}  \\
 &\quad +(x_1, x_3) \left( \begin{array}{cc}
 M_1(\qMLEparams)  & M_2(\qMLEparams) \\
M_2(\params) &  \frac{ \nver }{ \hat{u}^2 } \end{array}\right)(x_1, x_3)^t 
\end{align*}
and similarly 
\begin{align*}
  X^t \Sigma_t (\spars, \tau, \size/\KLsize)^{-1} X
  &\geq
    x_1^2 \left[- \sparsfunc''(\spars) +   \frac{ \KLsize  (2\nedg)^{\spars/2} }{ \alpha_0^3} (1 +o(1) )\right]\\
  &\quad+ x_2^2 \KLsize (1 -\alpha_0)\KLtau^{\alpha_0-2} (1 +o(1)) - 2x_3^2 \KLsize^{1 + a}  \\
 &\quad + (x_1, x_3) \left( \begin{array}{cc}
  M_1(\qMLEparams)  & M_2(\qMLEparams) \\
M_2(\qMLEparams) &  \frac{ \nver }{ \hat{u}^2 } \end{array}\right)(x_1, x_3)^t 
\end{align*}
From \cite[\cref{main-hatu}]{caron:naulet:rousseau:main}, we have that $\qMLEsize\zeta(\params)^{\qMLEspars}  = \qMLEspars \nver (1 + O(\KLsize^{-\delta_1})$ for some $\delta_1 >0$ which in turns implies that
 $ M_1(\qMLEparams)\frac{ \nver }{ \hat{u}^2 } - M_2(\qMLEparams)^2 = O( \KLsize^{-\delta_1} M_1(\qMLEparams)\nver  )$. 
It thus implies that if $|\params - \hat{\params}_{t,u}| \leq \KLsize^{-\delta}$ for some $\delta>0$ then
\begin{align*}
X^t \Sigma_t (\spars, \tau, \size/\KLsize)^{-1} X & \gtrsim
 x_1^2  \KLsize^{1 + \alpha_0}+ x_2^2 \KLsize  +  x_3^2\frac{ \nver }{ \log^2 \nedg } 
 \end{align*} 
and 
\begin{align*}
X^t \Sigma_t (\spars, \tau, \size/\KLsize)^{-1} X & \lesssim
 x_1^2  \KLsize^{1 + \alpha_0}\log^2 (\nedg)+ x_2^2 \KLsize  +  x_3^2\nver.
 \end{align*}

\begin{lemma}
  \label{lem:Sigmat}
  Let
  $\epsilon >0$, and $$ B_\epsilon = \Set{  \params \given |\spars - \KLspars| \leq \epsilon,\,  |\tau -\KLtau|\leq \epsilon, \, |\size  - \KLsize|\leq \epsilon\KLsize}.$$
Then, there exists $\delta_1>0$ such that  for all $\phi \in B_{\epsilon}$ and $\epsilon$ small enough
  \begin{gather*}
    \partial_{\spars}^2 \qloglik^{*}(\spars,\tau,u)%
    = \sparsfunc''(\spars) - \size \frac{ \zeta(\params)^\spars \log^2\zeta(\params)  }{ \spars } \left[ 1 - \frac{1    }{ \spars \log \zeta(\params)  } \right]^2( 1 + O( \KLsize^{-\delta_1})), \\
    \partial_{\tau}^2 \qloglik^{*}(\spars,\tau,u)%
    = 
    -\size (1-\spars)\KLtau^{\spars-2} + O(1) ,\quad 
    \partial_u^2 \qloglik^{*} (\spars,\tau,u)    = - \frac{ \nver }{ u^2 } ( 1 + \KLsize^{-\delta_1}),
  \end{gather*}
 and  
  \begin{gather*}
    \partial_{\spars}\partial_{u} \qloglik^{*}(\spars,\tau,u)%
    = -  \KLsize \frac{ \zeta(\params)^\spars \log^2\zeta(\params)  }{ \spars } \left[ 1 - \frac{1    }{ \spars \log \zeta(\params)  } \right]( 1 + O( \KLsize^{-\delta_1})),\\
    \partial_{\tau}\partial_{\spars} \qloglik^{*}(\spars,\tau,u)%
    = -  \size \KLtau^{\spars-1}\log (\tau) ( 1 + O( \KLsize^{-\delta_1})),\\
    \partial_{u} \partial_{\tau} \qloglik^{*}(\spars,\tau,u)%
    = \KLsize \tau^{\spars -1} ( 1 + O( \KLsize^{-\delta_1})).
  \end{gather*}
\end{lemma}

\begin{proof}
  We recall the convention in the main paper that $\phi = (\sigma,\tau,\size)$ and
  $\phi_u = (\sigma,\tau,u)$ with $u = s/\KLsize$.  Recall that 
  $$ \qloglik^{*}(\params_u) =   \sparsfunc(\sigma) + \nver \log \size - \nedg \mathcal A( \params, \zeta(\params)) - \log 2/2 .$$
  By direct computations,
  \begin{align*}
    \partial_{\spars}\qloglik^{*}(\params_u)
    &= 
    \sparsfunc'(\sigma) - \nedg \partial_\spars \mathcal A( \params, \zeta(\params) )-  \nedg \partial_\spars \zeta(\params)\partial_z\mathcal A( \params, \zeta(\params))\\
    & =  \sparsfunc'(\sigma) - \nedg \partial_\spars \mathcal A( \params, \zeta(\params)
  \end{align*}
  since $\partial_z\mathcal A( \params, \zeta(\params))=0$ for all $\params$. 
  Similarly 
  \begin{align*}
    \partial_{\tau}\qloglik^{*}(\params_u)=
 - \nedg \partial_\tau \mathcal A( \params, \zeta(\params), \quad \partial_{u}\qloglik^{*}(\params_u)%
    =      \frac{ \nver }{ u} - \KLsize \nedg \partial_\size  \mathcal A( \params, \zeta(\params)),
  \end{align*}
so that 
 \begin{align*}
    \partial_{\spars}^2\qloglik^{*}(\params_u)
    &= 
    \sparsfunc''(\sigma) - \nedg \partial_\spars^2 \mathcal A( \params, \zeta(\params) ) - \nedg \partial_\spars \zeta(\params)\partial_z\partial_\spars \mathcal A( \params, \zeta(\params) ), \\
\partial_{\spars,\tau}^2\qloglik^{*}(\params_u)   &= 
- \nedg \partial_{\spars,\tau}^2 \mathcal A( \params, \zeta(\params) ) - \nedg \partial_\tau \zeta(\params)\partial_z\partial_\spars \mathcal A( \params, \zeta(\params) ),  \\
      \partial_{\tau}^2\qloglik^{*}(\params_u) &=
 - \nedg \partial_\tau^2 \mathcal A( \params, \zeta(\params) -\nedg \partial_\tau \zeta(\params)\partial_z\partial_\tau \mathcal A( \params, \zeta(\params) ), \\ \partial_{u}^2\qloglik^{*}(\params_u)%
    &=      -\frac{ \nver }{ u^2} - \KLsize^2 \nedg \partial_\size^2  \mathcal A( \params, \zeta(\params))-\KLsize^2\nedg \partial_\size \zeta(\params)\partial_z\partial_\size \mathcal A( \params, \zeta(\params) ),\\
  \partial_{\spars,u}^2\qloglik^{*}(\params_u)   &=  - \nedg \KLsize \partial_{\spars,\size}^2 \mathcal A( \params, \zeta(\params) ) -\KLsize \nedg \partial_\spars \zeta(\params)\partial_z\partial_\size \mathcal A( \params, \zeta(\params) ), \\
   \partial_{\tau,u}^2\qloglik^{*}(\params_u)   &=  - \nedg \KLsize \partial_{\tau,\size}^2 \mathcal A( \params, \zeta(\params) )-\KLsize \nedg \partial_\size \zeta(\params)\partial_z\partial_\tau \mathcal A( \params, \zeta(\params) ),
  \end{align*}
Moreover  
  \begin{align*}
    \nedg \partial_\spars \mathcal A( \params, \zeta(\params) ) &= - \nver \log \zeta(\params) + \size \left[ \frac{ \zeta^\spars \log \zeta - \tau^\spars \log \tau }{ \spars } - \frac{ \zeta^\spars  - \tau^\spars}{ \spars^2 } \right] ,\\
    \nedg \partial_\tau \mathcal A( \params, \zeta(\params) ) &= \frac{ \zeta(\params) - \tau }{ 2 } - \size  \tau^{\spars-1},\\
 \nedg \partial_\size \mathcal A( \params, \zeta(\params) ) &= \frac{ \zeta^\spars - \tau^\spars  }{ \spars },
  \end{align*}
  which implies that 
 \begin{align*}
   \nedg \partial_\spars^2 \mathcal A( \params, \zeta(\params) )
   &=  \size \left[ \frac{ \zeta^\spars \log^2\zeta - \tau^\spars \log^2 \tau }{ \spars } - 2\frac{ \zeta^\spars\log \zeta  - \tau^\spars\log \tau }{ \spars^2 } +2 \frac{ \zeta^\spars - \tau^\spars }{ \spars^3 } \right] \\
   \nedg \partial_\tau^2 \mathcal A( \params, \zeta(\params) )
   &= -\frac{  1 }{ 2 } - \size (\spars-1) \tau^{\spars-2},\\
   \nedg \partial_{\size}^2 \mathcal A( \params, \zeta(\params) )
   &= 0,\\
   \nedg \partial_{\spars,\tau}^2 \mathcal A( \params, \zeta(\params) ) &=-\size    \tau^{\spars-1} \log \tau \\
   \nedg \partial_{\spars\size }^2 \mathcal A( \params, \zeta(\params) ) &=   \frac{ \zeta^\spars \log \zeta - \tau^\spars \log \tau }{ \spars } - \frac{ \zeta^\spars  - \tau^\spars}{ \spars^2 },\\
   \nedg \partial_{\tau, \size}^2 \mathcal A( \params, \zeta(\params) ) &= -   \tau^{\spars-1},
  \end{align*}
  and that 
  \begin{align*}
    \nedg\partial_z \partial_\spars \mathcal A( \params, \zeta(\params) ) &= -\frac{ \nver }{ \zeta(\params) } + \size\zeta^{\spars-1}  \log \zeta ,\\
    \nedg \partial_z \partial_\tau \mathcal A( \params, \zeta(\params) ) &= \frac{ 1 }{ 2 }, \\
 \nedg \partial_z\partial_\size \mathcal A( \params, \zeta(\params) ) &=  \zeta^{\spars-1}.
  \end{align*} 
   Note also that: 
  \begin{align}
  \label{zetaprime1}
  \partial_{\size} \zeta(\params) = \spars  \zeta(\params)^{\spars-1}\left(1 - \frac{ \tau }{ 2\zeta(\params)}- \size\spars  \zeta(\params)^{\spars-3}\right)^{-1}=\alpha_0 (2\nedg)^{(\spars-1)/2}(1 + o(1)),
\end{align}
and similarly
\begin{align*}
\partial_\tau \zeta(\params) = \frac{ 1 }{2} \left(1 - \frac{ \tau }{ 2\zeta(\params)}- \size\spars  \zeta(\params)^{\spars-3}\right)^{-1}= \frac{ 1 }{2} (1+o(1)),
\end{align*}
and,
\begin{align}
  \notag
  \partial_\spars \zeta(\params) &= \left( \size  \zeta(\params)^{\spars-1}\log \zeta(\params) - \frac{\nver  }{\zeta(\params)}\right) \left(1 - \frac{ \tau }{ 2\zeta(\params)}- \size\spars  \zeta(\params)^{\spars-3}\right)^{-1}\\
  \label{zetaprime3}
&=\left( \frac{\KLsize  (2\nedg)^{(\spars-1)/2}\log (2\nedg)}{ 2} (1+o(1))  - \frac{\nver  }{\sqrt{2\nedg}}+O(1)\right)(1+o(1)).
\end{align}

  Finally we obtain 
  \begin{align*}
    \partial_\spars^2\qloglik^{*}(\params_u)
   & =  \sparsfunc''(\sigma) - \size \left[ \frac{ \zeta^\spars \log^2\zeta  }{ \spars } - 2\frac{ \zeta^\spars\log \zeta   }{ \spars^2 } + 2\frac{ \zeta^\spars  }{ \spars^3 } +O(1)\right] +O(\KLsize^{2\alpha_0+2\epsilon}) \\
     \partial_{u}^2\qloglik^{*}(\params_u) &=-\frac{ \nver }{ u^2} + O( \KLsize^{2\alpha_0+2\epsilon} ),\\
     \partial_{\tau}^2\qloglik^{*}(\params_u) &=  \frac{  1 }{ 2 } - \size (1-\spars) \tau^{\spars-2}  + O(1)= - \size (1-\spars) \tau^{\spars-2} + O(1),\\
        \partial_{\spars,u}^2\qloglik^{*}(\params_u)   &=  -  \KLsize     \frac{ \zeta^\spars \log \zeta - \tau^\spars \log \tau }{ \spars } - \KLsize\frac{ \zeta^\spars  - \tau^\spars}{ \spars^2 } + O(\KLsize^{2\alpha_0+2\epsilon}), \\
        \partial_{\tau,u}^2\qloglik^{*}(\params_u)   &=   \KLsize \tau^{\spars-1} + O(\KLsize^{\alpha_0+\epsilon}),\\         \partial_{\tau,\spars}^2\qloglik^{*}(\params_u) &=\size    \tau^{\spars-1} \log \tau +O(\KLsize^{\alpha_0+\epsilon}).
  \end{align*}
\end{proof}

\subsection{Proof of \texorpdfstring{\cite[\cref{main-lem:bias:score}]{caron:naulet:rousseau:main}}{Lemma \ref{main-lem:bias:score}}}
\label{sec:pr:lem:bias:score}

We write  $\tilde I_0(\params) =  e^{ - \nedg \mathcal{A}(\params;\zeta(\params))}/(2\sqrt{\pi}) $, from \cite[\cref{main-eq:80}]{caron:naulet:rousseau:main},
\begin{align} 
I(\params) &= \tilde I_0(\params) \int_{\Reals}\exp\{-\nedg [\mathcal{A}(\params;
    \zeta(\params) - iy) - \mathcal{A}(\params;
    \zeta(\params))] \}du ;= \tilde I_0(\params)  \int_{\Reals} H_\params(y)dy \nonumber\\
I_0(\params) &= \tilde I_0(\params) \int_{\Reals}\exp\{- \frac{ y^2\nedg [-\partial_z^2\mathcal{A}(\params; \zeta(\params))] }{2} \} dy ;=\tilde I_0(\params)\int_{\Reals} H_{0,\params}(y)dy
\end{align} 
and
$$\nabla [\qloglik - \loglik]  = - \nabla \log I(\params) + \nabla \log I_0(\params) + \frac{1}{2}\nabla \log \left( \nedg \partial_z^2\mathcal{A}(\params;\zeta(\params))\right)
$$
  Then,
 \begin{align*}
   \nabla I(\params)  & = \tilde I_0(\params) \int_{\Reals}\nabla H_\params(y) dy + \frac{  \nabla \tilde I_0(\params) }{ \tilde I_0(\params) }  I(\params)\\
     \nabla I_0(\params) &= - \frac{I_0(\params) }{2}\nabla \log \left( \nedg \partial_z^2\mathcal{A}(\params;\zeta(\params))\right)
             +\frac{  \nabla \tilde I_0(\params) }{ \tilde I_0(\params) }  I_0(\params)
  \end{align*}
  and 
    \begin{align} \label{Nabla_qloglik}
    \nabla [\qloglik - \loglik] 
        & =\frac{  -\tilde I_0(\params)  \int_{\Reals}\nabla H_\params(y) dy}{ I(\params) } 
            -\frac{  \nabla \tilde I_0(\params) }{ \tilde I_0(\params) } \frac{  I(\params) - I_0(\params)}{ I_0(\params)}  \nonumber \\
            &=  \frac{  -  \int_{\Reals}\nabla H_\params(y) dy}{ \int_{\Reals} H_{0,\params}(y)dy } 
     + \frac{  I(\params) - I_0(\params)}{ I_0(\params)} \left[ \frac{  \tilde I_0(\params)  \int_{\Reals}\nabla H_\params(y) dy}{ I(\params) } 
 -    \frac{  \nabla \tilde I_0(\params) }{ \tilde I_0(\params) } \right] 
 \end{align}
We have
\begin{align*}  
\int_{\Reals}\nabla H_\params(y) dy &=     -\nedg
\int_{\Reals} e^{-\nedg [\mathcal{A}(\params;
    \zeta(\params) - iy) - \mathcal{A}(\params;
    \zeta(\params))} [ \nabla_\params \mathcal{A}(\params;
    \zeta(\params) - iy) - \nabla_\params \mathcal{A}(\params; \zeta(\params))  ] dy \\
     &    -\nedg \nabla \zeta(\params) \int_{\Reals} e^{-\nedg [\mathcal{A}(\params;
    \zeta(\params) - iy) - \mathcal{A}(\params;
    \zeta(\params))} [ \nabla_z \mathcal{A}(\params;
       \zeta(\params) - iy) - \nabla_z \mathcal{A}(\params; \zeta(\params))  ] dy\\
  &\eqqcolon \Delta_1 + \Delta_2
 \end{align*}
Recall that 
$\sqrt{-2\nedg \partial_z^2\mathcal{A}(\params_0;\zeta(\params_0))} = 2^{-1/2}(1+o(1))$ so that 
$$  \int_{\Reals} H_{0,\params}(y)dy  =  2\sqrt{\pi}(1 +o(1)) .$$
Also 
\begin{align}
  \label{nablaphiA}
\nedg \partial_\spars\mathcal{A}(\params; z) &= - \nver \log z + \frac{ \size  }{2 } \left[ \frac{ \log (z) z^\spars - \log (\tau) \tau^\spars }{ \spars } - \frac{ z^\spars -  \tau^\spars}{\spars^2 } \right]\\
\nedg \partial_{\tau}\mathcal{A}(\params; z) &= \frac{ z-\tau }{2} -  \size \tau^{\spars-1}
, \quad \nedg \partial_{\size}\mathcal{A}(\params; z) =  \frac{ z^\spars -  \tau^\spars}{\spars }\nonumber
\end{align}
  In particular  when $|y| \leq \sqrt{ \zeta(\params_0)}$,  i.e. $|y| \leq c_0\sqrt{ t} $  for some $c_0>0$, then
\begin{align*}
 \log ( \zeta(\params_0) -iy )- \log ( \zeta(\params_0) )  &= \frac{ \log ( \zeta(\params_0)^2 +u^2 )-\log ( \zeta(\params_0)^2  ) }{ 2}  - \frac{ i u }{ \zeta(\params_0)  } + O( |y|^2 / t^2 ),\\  ( \zeta(\params_0) -iy )^{\spars_0}- ( \zeta(\params) )^{\spars_0} &=    -iy\spars_0 \zeta(\params_0)^{\spars_0-1} + O(|y|^2t^{\spars_0-2} )
 \end{align*}
and also, from  \cref{eq:24,eq:25,eq:29}, writing $A_0 = 2[-\nedg \partial^2_z\mathcal A(\params;\zeta(\params))]$
\begin{align*}
  &e^{-\nedg [\mathcal{A}(\params; \zeta(\params) - iy) - \mathcal{A}(\params;
    \zeta(\params)) ]} = e^{-y^2 A_0/4}\times \\
     & \quad \left(  1 - \frac{ iy^3\nedg  }{ 2 \pi } \left[ 
\int_{|\xi| = \zeta(\params_0)/2} \frac{ \mathcal A( \params_0; \zeta(\params_0) + \xi) }{ \xi^4 }  + O\left( \frac{  \sup_{1/2\leq x\leq 2} \sup_{\varphi \in [-\pi ,\pi]}| \mathcal A( \params_0; x \zeta(\params_0)e^{i\varphi})| }{ \zeta(\params_0)^{4} } \right) \right]\right)\\
& = e^{-y^2 A_0/4}\left(  1 - \frac{ iy^3\nedg  }{ 2 \pi } 
\int_{|\xi| = \zeta(\params_0)/2} \frac{ \mathcal A( \params_0; \zeta(\params_0) + \xi) }{ \xi^4 } d\xi + O\left( \frac{ |y|^4 \log t }{ t^2} \right) \right).
\end{align*}
We combine this with the arguments as in the proof of \cite[\cref{main-lem:3}]{caron:naulet:rousseau:main} in \cref{sec:proof-lem:6}, but we consider a decomposition  into $|y|\leq \zeta(\params)^{1/2}$ and $|y|>\zeta(\params)^{1/2}$,  instead of $|y| \leq C a(\spars)$ and its complement we obtain  that
  \begin{align*}
& \left| \nedg \int_{|y|\leq \sqrt{ \zeta(\params_0) }} e^{-\nedg [\mathcal{A}(\params;
    \zeta(\params) - iy) - \mathcal{A}(\params;
    \zeta(\params)) ]} [ \nabla_\spars \mathcal{A}(\params;
    \zeta(\params) - iy) - \nabla_\spars \mathcal{A}(\params; \zeta(\params))  ] dy \right| \\
     &\quad \lesssim  (\int e^{-y^2/4} |y|^6 dy) t^{\spars_0-1}  \log t=  O(t^{\spars_0-1}  \log t) \\
 & \left| \nedg \int_{|y|\leq \sqrt{ \zeta(\params_0) }} e^{-\nedg [\mathcal{A}(\params;
    \zeta(\params) - iy) - \mathcal{A}(\params;
    \zeta(\params)) ]} [ \nabla_\tau \mathcal{A}(\params;
    \zeta(\params) - iy) - \nabla_\tau \mathcal{A}(\params; \zeta(\params))  ] dy \right|       \\
 &\quad \lesssim    t^{-1} \\
 &\left| \nedg \int_{|y|\leq \sqrt{ \zeta(\params_0) }} e^{-\nedg [\mathcal{A}(\params;
    \zeta(\params) - iy) - \mathcal{A}(\params;
    \zeta(\params)) ]} [ \nabla_{\size} \mathcal{A}(\params;
    \zeta(\params) - iy) - \nabla_{\size} \mathcal{A}(\params; \zeta(\params))  ] dy \right|    \\
     &\quad \lesssim   t^{\spars_0-2} 
  \end{align*}

  Similarly, using \cref{eq:82}, if $|y| \leq c_0\delta t$
  \begin{multline*}
    \nedg[ \partial_z\mathcal{A}(\params_0; \zeta(\params_0) - iy) - \partial_z\mathcal{A}(\params_0;    \zeta(\params_0) )] 
    \\=  \frac{ iy }{2} \left(  1 + \frac{ 2(\nedg - \spars_0\nver) }{ \zeta(\params_0)^2 } +2(1-\spars_0) t\zeta(\params_0)^{\spars_0-2} \right)+ O(|y|^2t^{-1} )
  \end{multline*}
so that combining with $\nabla \zeta(\params_0)$ computed in the previous Section, 
\begin{align*}
|\Delta_2(1) | & \lesssim  \frac{ \log t }{ t^{1-\spars_0}} , \quad 
|\Delta_2(2) |  \lesssim  \frac{ \log t }{ t } , \quad 
|\Delta_2(3) |  \lesssim  \frac{ 1 }{ t^{1-\spars_0}}.
\end{align*}
Finally, since the integrals over $|y| > \sqrt{\zeta(\params_0)}$ are exponentially small. 
 we obtain that 
\begin{align}  \label{nablaH}
\int_{\Reals}\nabla_\spars H_\params(y) dy &= O(\frac{ \log t }{ t^{1-\spars_0}}), \quad 
\int_{\Reals}\nabla_\tau H_\params(y) dy = O(\frac{ 1 }{ t}), \quad 
\int_{\Reals}\nabla_u H_\params(y) dy  = O(\frac{ 1 }{ t^{1-\spars_0}}).
\end{align}
Combining the above upper bound, with  $|I(\params_0)-I_0(\params_0)|/I_0(\params_0) = O(\log t /t^2)$ and 
 \cref{nablaphiA} to bound $\nabla I_0(\params_0)/I_0(\params_0)$ we obtain 
 \begin{align*}
    \nabla_\spars [\qloglik - \loglik] \lesssim \frac{ (\log t)^2 }{ t^{1-\spars_0}} , \quad     \nabla_\tau [\qloglik - \loglik] \lesssim \frac{ \log t }{ t} , \quad
       t \nabla_{\size} [\qloglik - \loglik] \lesssim \frac{ \log t }{ t^{1-\spars_0}}.
 \end{align*}
 Note that the above control on  $\nabla \qloglik - \nabla L_t$ is valid in expectation 
since  $$\zeta(\params)^2 \geq 2 \nedg - 2 \spars_0 \nver \geq 2 \nedg(1-\spars_0) ,\quad  \zeta(\params) \geq (2t)^{1/(2-\spars_0)}. $$

We now prove that 
$$\Delta = \left|\Tr\left(V_0[D_{\params_u}^2\loglik - D_{\params_u}^2\qloglik^*] (\params_{0,u}) \right)\right|  =o(1)$$
with
$$V_0 = \frac{ 1 }{ t } 
\left( \begin{array}{ccccc} 
 c_1t^{-\sigma_0}&  & 0 & & c_2t^{-\sigma_0}\log t \\
0 & & 1 & & 0 \\
t^{-\sigma_0}\log t &  &0 & & t^{-\sigma_0} \log^2 t   \end{array}\right).$$
This boils down to proving that
\begin{align*}
 t^{-1-\spars_0}| D^2_{\spars,\spars}[\loglik -\qloglik](\params_0) | & = o(1)\\
 t^{-1-\spars_0} \log t t | D^2_{\spars,\size}[\loglik -\qloglik](\params_0) | & = o(1)\\
  t^{-1} | D^2_{\tau,\tau}[\loglik -\qloglik](\params_0) | & = o(1)\\
   t^{-1-\spars_0} (\log t)^2 t^2 | D^2_{\size,\size}[\loglik -\qloglik](\spars_0) | & = o(1)
\end{align*}
 We use \cref{Nabla_qloglik} so that 
  \begin{align*} 
    D^2_{\params} [\qloglik - \loglik] 
                   &=  \frac{  -  \int_{\Reals}D^2 H_\params(u) du}{ \int_{\Reals} H_{0,\params}(u)du }  +  \frac{   \int_{\Reals}\nabla H_\params(u) du \int_{\Reals} \nabla H_{0,\params}(u)du}{ [\int_{\Reals} H_{0,\params}(u)du]^2 } \\
                   & \quad
     + \frac{  I(\params) - I_0(\params)}{ I_0(\params)} \left[ \frac{  \tilde I_0(\params) \int_{\Reals}\nabla H_\params(u) du}{ I(\params) } 
 -    \frac{  \nabla \tilde I_0(\params) }{ \tilde I_0(\params) } \right] .
 \end{align*}
First since 
$$-\nedg  \partial^2_z \mathcal A(\params,z)  = \frac{ \nedg - \spars \nver }{ z^2 } + \size (1-\spars ) z^{-2+\spars} $$ 
simple computations imply that 
\begin{align*} 
\nabla_\spars [-\nedg  \partial^2_z \mathcal A(\params_0,\zeta(\params_0))] & = O(t^{-1+\spars_0}\log t) , \quad \nabla_\tau [-\nedg  \partial^2_z \mathcal A(\params_0,\zeta(\params_0))]  = O(t^{-1}) \\
 t^2\nabla_{\size} [-\nedg  \partial^2_z \mathcal A(\params_0,\zeta(\params_0))]  &= O(t^{\spars_0}) 
  \end{align*}
 which implies the same orders for  $\int_{\Reals} \nabla H_{0,\params}(u)du$.
 Moreover \cref{nablaphiA}, together with the relation $ \spars_0 N = t \zeta(\params_0)^{\spars_0}( 1 + O(t^{-\delta}) )$ for some $\delta>0$,  imply that 
 \begin{align*}
| \nabla_{\spars} \log \tilde I_0(\params_0) | & \lesssim t^{1 +\spars_0} , \quad |\nabla_{\tau} \log \tilde I_0(\params_0) |  \lesssim t, \quad t| \nabla_{\size} \log \tilde I_0(\params_0) |  \lesssim t^{1 +\spars_0} .
 \end{align*}
   Combining this with 
 $$ |I(\params_0) -I_0(\params_0)|/I_0(\params_0) =O(\log t /t^2) $$
 implies that term by term with $D^2_{\params_u} $ representing either $D^2_{\spars, \spars}$, $D^2_{\tau, \tau}$ or $t^2 D^2_{\size, \size}$ and $\nabla_{\params_u}$ representing either $\nabla_{\spars}$, $\nabla_{\tau}$ or $t \nabla_{\size}$
\begin{align*} 
  |  D^2_{\params_u} [\qloglik - \loglik] |
                   &\lesssim   | \int_{\Reals}D^2_{\params_u} H_{\params_u}(x) dx|   +   |\int_{\Reals}\nabla_{\params_u} H_{\params_u}(x) dx| |\int_{\Reals} \nabla_{\params_u} H_{0,\params_u}(x)dx|  \\
 &\quad      + \frac{ \log t }{t^2 }   |\int_{\Reals}\nabla_{\params_u} H_{\params_u}(x) dx | +o(1)
 \end{align*} 
Using \cref{nablaH} we then obtain that
 \begin{align*} 
  |  D^2_{\params_u} [\qloglik - \loglik] |&\lesssim   | \int_{\Reals}D^2_{\params_u} H_{\params_u}(x) dx|   +o(1)
 \end{align*} 
 We have 
 We now study the second derivatives .
\begin{align*}
\int_{\Reals} & D^2_{\params_u} H_{\params_u}(x) dx   = -\nedg \int_{\Reals} e^{-\nedg[ \mathcal{A}(\params; \zeta(\params) - iu) - \mathcal{A}(\params; \zeta(\params))] } \left(    D^2_{\params_u}\mathcal{A}(\params;
    \zeta(\params) - iu)-   D^2_{\params_u}\mathcal{A}(\params;
    \zeta(\params) )\right)\intd u \\
    & -2\nedg \nabla_{\params_u}\zeta(\params_u) \int_{\Reals}e^{-\nedg[ \mathcal{A}(\params; \zeta(\params) - iu) - \mathcal{A}(\params; \zeta(\params))] }  \left( \partial_{\params_u}\partial_z\mathcal{A}(\params;
    \zeta(\params) - iu) - \partial_{\params_u}\partial_z\mathcal{A}(\params;
    \zeta(\params) ) \right)\intd u \\
 &    - \nedg D^2_{\params_u} \zeta(\params) \int_{\Reals}e^{-\nedg[ \mathcal{A}(\params; \zeta(\params) - iu) - \mathcal{A}(\params; \zeta(\params))] } \left( \partial_z^2\mathcal{A}(\params;
    \zeta(\params) - iu) - \partial_z^2\mathcal{A}(\params;
    \zeta(\params) ) \right)\intd u \\
    & := \Delta_1 + \Delta_2 + \Delta_3
 \end{align*}
Using \cref{zetaprime1,zetaprime3}, we can bound 
$$ 
\partial^2_{\spars,\spars} \zeta(\params) = O(t^{\spars_0} \log^2 t), \quad \partial^2_{\tau,\tau} \zeta(\params) = O(1/t), \quad  t^2 \partial^2_{\size,\size} \zeta(\params) = O(t^{2\spars_0-1}), \quad t \partial^2_{\spars,\size} \zeta(\params) = O(t^{\spars_0} \log t).$$
Since $\nedg\partial_z^2\mathcal{A}(\params;
    \zeta(\params) ) = O(1)$ , 
  $$ \Tr[ V_0 \Delta_3 ]  = o(1). $$
Also using 
$$  \nedg  \partial_z\mathcal{A}(\params;z)  = -z/2 + \tau/2 + (\nedg- \spars \nver)/z + \size z^{-1+ \spars}$$
we have that 
\begin{align*}
|\nedg & \partial_\spars[ \partial_z\mathcal{A}(\params; \zeta(\params)- iu) - \partial_z \mathcal{A}(\params;\zeta(\params) ) ] | \lesssim   |u| t^{\spars_0-1} \log t,  \\
|\nedg & \partial_\tau[ \partial_z\mathcal{A}(\params; \zeta(\params)- iu) - \partial_z \mathcal{A}(\params;\zeta(\params) ) ] | =0, \\
t |\nedg & \partial_\spars[ \partial_z\mathcal{A}(\params; \zeta(\params)- iu) - \partial_z \mathcal{A}(\params;\zeta(\params) ) ] | \lesssim   |u| t^{\spars_0-1} ,  \\
\end{align*}
which together with 
$$ \partial_\spars \zeta(\params_0) = O ( t^{\spars_0}\log t), \quad  \partial_\tau  \zeta(\params_0) = O ( 1) , \quad  t \partial_{\size} \zeta(\params_0) = O ( t^{\spars_0}) $$
  lead to 
 $$\Tr [V_0 \Delta_2] = o(1).$$
Finally using \cref{nablaphiA}
\begin{align*}
\nedg \partial_{\spars,\spars}^2\mathcal{A}(\params; z) &=  \frac{ \size  }{2 } \left[ \frac{ \log (z)^2 z^\spars - \log (\tau)^2 \tau^\spars }{ \spars } -2 \frac{ z^\spars\log z - \log \tau \tau^\spars}{\spars^2 } +2 \frac{ z^\spars - \tau^\spars}{\spars^3 }  \right]  \\
  \nedg \partial_{\tau,\tau}^2\mathcal{A}(\params; z) &= \frac{ 1 }{2} -  \size (\spars -1)\tau^{\spars-1},\\
  \nedg \partial_{\size,\size}^2\mathcal{A}(\params; z) &=  0,\\
 \nedg \partial_{\size,\spars}^2\mathcal{A}(\params; z) &=  \frac{ z^\spars \log z-  \tau^\spars\log \tau }{\spars } -  \frac{ z^\spars  -  \tau^\spars }{\spars^2 },
\end{align*}
which implies that 
\begin{align*}
\nedg |\partial_{\spars,\spars}^2[\mathcal{A}(\params; z-iu)-\mathcal{A}(\params; z)]| &\lesssim   t^{\spars_0} |u| \log^2 t \\
|\nedg \partial_{\tau,\tau}^2[\mathcal{A}(\params; z-iu) -\mathcal{A}(\params; z)]| &= 0 \\
t| \nedg \partial_{\size,\spars}^2[\mathcal{A}(\params; z-iu)-\mathcal{A}(\params; z)]| &\lesssim    t^{\spars_0} \log t |u|.
\end{align*}
Hence
 $$\Tr [V_0 \Delta_1] = o(1)$$
 and 
 $$\left|\Tr\left(V_0 [D_{\params_u}^2\loglik - D_{\params_u}^2\qloglik^*] (\params_{0,u}) \right)\right|  =o(1).$$


\section{Asymptotic properties of multigraphex processes}
\label{sec:asympt-prop-mult}

\subsection{Proof of \texorpdfstring{\cite[\cref{main-th:asympmultigraphex}]{caron:naulet:rousseau:main}}{Theorem \ref{main-th:asympmultigraphex}}}
\label{sec:proof-theorem-asympmultigraphex}

\subsubsection{Number of nodes and number of multiedges}

The number of nodes $N_t$ is the same in the multigraph and simple graph. The asymptotic results for the number of nodes $N_t$ then follow directly from Theorems 3 and 4 in \cite{caron:rousseau:18}.

Recall that $\nedg$ is twice the number of multiedges. We have, using the Slivnyak-Mecke formula,
\begin{align*}
\mathbb E[\nedg]&=\mathbb E\left [\sum_{i\geq 1} D_{t,i}\1{\theta_i\leq t}\right ]\\
&=t\int_0^\infty \oW_1(x,x)dx +t^2 \int_{\mathbb R_+^2} \oW_1(x,y)dxdy.
\end{align*}


Similarly,
\begin{align*}
&\mathbb E[(\nedg)^2]\\&=\mathbb E\left [\sum_{i,j,k,\ell} \widetilde n_{ik}\widetilde n_{j\ell}\1{\theta_i\leq t}\1{\theta_j\leq t}\1{\theta_k\leq t}\1{\theta_\ell\leq t}\right ]\\
&=t^4 \int_{\mathbb R_+^4} \oW_1(x,u)\oW_1(y,v)dxdydudv +4t^3 \int_{\mathbb R_+^3} \oW_1(x,u)\oW_1(y,u)dxdydu \\
 & + 2t^3 \int_{\mathbb R_+^3} \oW_1(x,u)\oW_1(y,y)dxdydu + t^2\int_{\mathbb R_+^2} \oW_1(x,x)\oW_1(y,y)dxdy+2t^2\int_{\mathbb R_+^2} \oW_2(x,y)dxdy\\
 &+4t^2\int_{\mathbb R_+^2} \oW_1(x,x)\oW_1(x,y)dxdy+t\int_0^\infty \oW_2(x,x)dx\\
&=\mathbb E[\nedg]^2+4t^3 \int_{\mathbb R_+^3} \oW_1(x,u)\oW_1(y,u)dxdydu\\
&+2t^2\int_{\mathbb R_+^2} \oW_2(x,y)dxdy+4t^2\int_{\mathbb R_+^2} \oW_1(x,x)\oW_1(x,y)dxdy+t\int_0^\infty \oW_2(x,x)dx
\end{align*}
hence
\begin{align*}
\var(\nedg)&=4t^3 \int_{\mathbb R_+^3} \oW_1(x,u)\oW_1(y,u)dxdydu\\
&+2t^2\int_{\mathbb R_+^2} \oW_2(x,y)dxdy+4t^2\int_{\mathbb R_+^2} \oW_1(x,x)\oW_1(x,y)dxdy+t\int_0^\infty \oW_2(x,x)dx.
\end{align*}
Note that, for all $x,y>0$,
$$
W_1(x,y)^2\leq W_2(x,y).
$$
Hence, by Cauchy-Schwarz
\begin{align*}
\int_{\mathbb R_+^2} \oW_1(x,x)\oW_1(x,y)dxdy&\leq \left (\int \oW_1(x,x)^2dx \int \left (\int \oW_1(x,y)dy\right )^2 dx   \right )^{1/2}\\
&\leq \left (\int \oW_2(x,x)dx \int \left (\int \oW_1(x,y)dy\right )^2 dx   \right )^{1/2}<\infty.
\end{align*}
Therefore,
\begin{align*}
\var(\nedg)=O(t^{-1}\mathbb E[\nedg]^2)
\end{align*}
and it follows from Lemma S3.1 in \cite{caron:rousseau:18} that
$$
\nedg\sim \mathbb E[\nedg]\sim t^2 \int_{\mathbb R_+^2} \oW_1(x,y)dxdy
$$
almost surely as $t$ tends to infinity.


\subsubsection{Number of nodes of degree $j$.}

\paragraph*{ Case $\alpha_0=0$.}

Let $N^{(s)}_{tj}$ denote the number of nodes of degree $j$ in the simple graph. A node with degree $j$ in the multigraph $\mathcal G_t$ has a degree smaller or equal to $j$ in the corresponding simple graph $\mathcal G^{(s)}_t$, and we therefore have, for all $j\geq 1$.
\begin{align*}
N_{tj}\leq \sum_{k=1}^j N^{(s)}_{tk}
\end{align*}
If \cite[\cref{main-assumpt:1}]{caron:naulet:rousseau:main} holds for $\alpha_0=0$ and some slowly varying function $\ell_1$, Theorems 3 and 4 in \cite{caron:rousseau:18} imply that $N^{(s)}_{tk}=o(t\ell_1(t))$ in mean, and almost surely if \cite[\cref{main-assumpt:2}]{caron:naulet:rousseau:main} also holds. The result for $N_{tj}$ then follows by comparison.

\paragraph*{Case $\alpha_0\in(0,1]$.}

We have
$$
N_{tj} = N_{j}^{(s)}  + (N_{tj}-N_{tj}^{(s)})
$$
where, using \cite[\cref{main-lemma:Ntilde}]{caron:naulet:rousseau:main},
$$
|N_{tj}-N_{tj}^{(s)}|\leq \widetilde N_t=o(N_t).
$$
and therefore, using the results for the simple graph~\cite{caron:rousseau:18}, for $\alpha_0>0$, for all $j\geq 1$
$$
N_{tj}\sim N_{tj}^{(s)}
$$
almost surely, and for $\alpha_0=1$
$N_{t1}\sim N_{t1}^{(s)} $ and, for $j\geq 2$, $N_{tj}=o(N_t).$

 %
%
%

\subsection{Second order asymptotics of number of vertices}
\label{sec:second-order-asympt}

\cite[\cref{main-th:asympmultigraphex}]{caron:naulet:rousseau:main} gives the first order asymptotic of the number
of vertices as $t \to \infty$. We need, however, in the proof of \cite[\cref{main-thm:2}]{caron:naulet:rousseau:main}
the second order asymptotics under the additional \cite[\cref{main-ass:tails}]{caron:naulet:rousseau:main}. This is
given in the next Lemma.

\begin{lemma}
  \label{lem:4}
  Let \cite[\cref{main-assumpt:1,main-ass:tails}]{caron:naulet:rousseau:main} be satisfied for some $\alpha_0 \in (0,1)$, and
  let $\int_{\NNReals}W(x,x)\,\intd x < \infty$. Then, there exists $\eta > 0$
  such that as $t\to \infty$
  \begin{equation}
    \label{eq:33}
    \EE[\nver]%
    = t \mu^{-1}(t^{-1}) \Gamma(1 - \alpha_0)(1 + O(t^{-\eta}).
  \end{equation}
\end{lemma}
\begin{proof}
  As for the proof of \cite[\cref{main-th:asympmultigraphex}]{caron:naulet:rousseau:main},
\begin{align}
  \label{eq:234}
  \EE[\nver]%
  &= t\int_{\NNReals}\Big(1 - e^{-t\mu(x)}\Big)\,\intd x%
    +t \int_{\NNReals}W(x,x)e^{-t\mu(x)}\,\intd x.
\end{align}
The second term in the previous display is $O(t)$ as $x \mapsto W(x,x)$ is
integrable. Define $\bar \mu$ as
\begin{equation}\label{barmu}
\bar \mu(x) = c_0^{1/\alpha_0} x^{-1/\alpha_0}, \quad \text{so that } \,  t\int_{\NNReals}\Big(1 - e^{-t\bar \mu(x)}\Big)dx = t^{1+\alpha_0} c_0\Gamma(1 - \alpha_0).
\end{equation}
Then since $|\mu^{-1}(y)  - \bar\mu^{-1}(y) | \leq C y^{\beta}$ when $y$ is close to 0, then there exists $x_0$ such that for all $x \geq x_0$,
\begin{equation}\label{Deltamu}
\left|  x  -c_0 \mu(x)^{-\alpha_0} \right| \leq C \mu(x)^{ \beta }  \quad |\mu(x) - \bar \mu(x) | \leq C' \bar \mu(x)^{ \beta + \alpha_0 +1}
\end{equation}
for some $C'>0$. We split the first  integral of \cref{eq:234} between $x\leq x_0$ and $x>x_0$.
We have
\begin{align*}
|t\int_{0}^{x_0}\Big( e^{-t\bar \mu(x)} - e^{-t\mu(x)}\Big)\,\intd x | \leq t [e^{-t\mu(x_0}+e^{-t\mu(x_0)}] x_0 = o(t)
\end{align*}
and
\begin{align*}
\left|t\int_{x_0}^\infty\Big( e^{-t\bar \mu(x)} - e^{-t\mu(x)}\Big)dx \right| & \leq t^2 \int_{x_0}^\infty |\mu(x)-\bar \mu(x) |e^{-t\bar \mu(x)} dx  \\
& \lesssim t^{ 1 - \beta} = o(t)
\end{align*}
so that \cref{eq:33} is proved with $\eta = \alpha_0$.
\end{proof}

\subsection{Proof of \texorpdfstring{\cite[\cref{main-lem:2}]{caron:naulet:rousseau:main}}{Lemma \ref{main-lem:2}}}
\label{sec:proof-lem:2}


The proof first approximates  $-\sparsfunc'(\alpha)$ using the subgraph $\mathcal G_t^{(s)}$ obtained from
$\mathcal{G}_t$ by removing the multiedges and selfedges and \cite[\cref{main-lemma:Ntilde}]{caron:naulet:rousseau:main}.  More precisely define
$Z_{i,j} \coloneqq \Ind_{\tilde{n}_{i,j}\geq 1}$ if $i\ne j$, and $Z_{i,i} =
0$. The variables $Z_{i,j}$ with $i\neq j $ have a Bernoulli distribution with expectation
$W(x,y)$ where $W(x,y) = 1 - W_m(x,y,0)$ and
let $D_{t,i}^{(s)} \coloneqq \sum_{j\neq i} Z_{i,j}\Ind_{\theta_j \leq t}$ denote the
degrees of node $i$ in the simple graph $\mathcal{G}_t^{(s)}$.   We define
$X_t(\alpha) \coloneqq \sum_i\Ind_{\theta_i\leq t}\sum_{k\geq
  1}\frac{1}{k-\alpha}\Ind_{D_{t,i}^{(s)}>k}$ which  is the
simple graph analogous of $-\sparsfunc'(\alpha)$.

In a first step, we reduce to the simple graph model by showing that under
\cite[\cref{main-ass:Wm}]{caron:naulet:rousseau:main}, we have
$-\sparsfunc'(\alpha) - X_t(\alpha) = O(( \widetilde{N}_t + N_t^{\mathrm{se}}) \log(\nedg)) = o(t^{1+\alpha_0/2+\delta})$
for any $\delta>0$, almost-surely, where $\widetilde{N}_t$ is as in the
\cite[\cref{main-lemma:Ntilde}]{caron:naulet:rousseau:main} and $N_t^{\mathrm{se}}$ is the number of vertices with at
least one selfedge in $\mathcal{G}_t$. This implies that is enough to understand
$X_t(\alpha)$. We will proceed by showing in a second time that
\begin{equation*}
  X_t(\alpha) = U_t(\alpha) + V_t(\alpha),
\end{equation*}
where
\begin{equation*}
  U_t(\alpha)%
  \coloneqq%
  \sum_i\Ind_{\theta_i \leq t} \int_0^1 \frac{1 -
    (1-u)^{D_{t,i}^{(s)}}}{u(1-u)^{\alpha}}\,\intd u,
\end{equation*}
and,
\begin{equation*}
  V_t(\alpha)%
  \coloneqq%
  - \sum_i\Ind_{\theta_i\leq t}\Ind_{D_{t,i}^{(s)} \geq 1}\int_0 ^1
  \frac{(1-u)^{D_{t,i}^{(s)}}}{(1-u)^{1+\alpha}}\,\intd u.
\end{equation*}
Finally, the proof is finished by computing the expectations and variances of
$U_t(\alpha)$ and $V_t(\alpha)$.

\subsubsection{Control of $-\sparsfunc'(\alpha) - X_t(\alpha)$.}

By definition,
\begin{equation*}
  -\sparsfunc'(\alpha) - X_t(\alpha)%
  = \sum_i \Ind_{\theta_i \leq t}\sum_{k\geq
  1}\frac{1}{k-\alpha}\big(\Ind_{D_{t,i}> k} - \Ind_{D_{t,i}^{(s)}> k} \big).
\end{equation*}
Since $Z_{i,j} \leq \tilde{n}_{i,j}$,  $D_{t,i}^{(s)} \leq
D_{t,i}$  and
\begin{align*}
 0 \leq  -\sparsfunc'(\alpha)& - X_t(\alpha)%
  = \sum_i \Ind_{\theta_i \leq t}\sum_{k\geq 1}
    \frac{1}{k-\alpha}\Ind_{D_{t,i}^{(s)}\leq k}\Ind_{D_{t,i}> k}\\
  &\leq \sum_i\Ind_{\theta_i\leq t}\Ind_{D_{t,i} - D_{t,i}^{(s)}\geq
    1}\sum_{k\geq 1}\frac{1}{k-\alpha}\Ind_{D_{t,i} > k},
\end{align*}
where the last line follows because
$(D_{t,i}^{(s)} \leq k \ \mathrm{and}\ D_{t,i}> k) \Rightarrow D_{t,i} - D_{t,i}^{(s)}\geq 1$.
Now we remark that
$\sum_{k \geq 1}\frac{1}{k-\alpha}\Ind_{D_{t,i}>k} \leq \sum_{k \geq 1}\frac{1}{k-\alpha}\Ind_{\nedg > k} =(1+o(1)) \log(\nedg)$.
In addition,
\begin{align*}
  \sum_i\Ind_{\PPlab_i\leq t}\Ind_{D_{t,i} - D_{t,i}^{(s)}\geq 1}%
  &= \sum_i\Ind_{\PPlab_i\leq t}\Ind_{D_{t,i} - D_{t,i}^{(s)}\geq
    1}\Ind_{\tilde{n}_{i,i}= 0}%
    + \sum_i\Ind_{\PPlab_i\leq t}\Ind_{D_{t,i} - D_{t,i}^{(s)}\geq
    1}\Ind_{\tilde{n}_{i,i}\geq 1}\\
  &\leq \widetilde{N}_t + N_t^{\mathrm{se}}.
\end{align*}
Hence
\begin{align*}
 0 \leq  -\sparsfunc'(\alpha)& - X_t(\alpha) \leq 2(\widetilde{N}_t + N_t^{\mathrm{se}} ) \log(\nedg).
\end{align*}
Moreover, from \cite[\cref{main-lemma:Ntilde}]{caron:naulet:rousseau:main}, for any $\delta>0$,%
$$ \EE( \widetilde{N}_t ) =o (t^{1 + \alpha_0/2+ \delta})$$
and
$$ \EE (N_t^{\mathrm{se}} ) \leq \EE ( \sum_i \1{\theta_i\leq t} \tilde n_{ii}  ) = t \int_{\mathbb R_+} W(x,x)dx $$
so that for any $\delta>0$,
\begin{align}
  \label{eq:90}
\EE[ -\sparsfunc'(\alpha)] = \EE(X_t(\alpha) ) + O(\EE( [\widetilde{N}_t + N_t^{\mathrm{se}} ]) \log t)  = o (t^{1 + \alpha_0/2+ \delta}).
\end{align}

\subsubsection{Proof that $X_t(\alpha) = U_t(\alpha) + V_t(\alpha)$}

The decomposition  of $X_t$ into $U_t + V_t $ follows by remarking that
$\sum_{k=1}^{m-1}\frac{1}{k-\alpha} = \varphi(m-\alpha) - \varphi(1 - \alpha)$,
where $\varphi(x) = \Gamma'(x)/\Gamma(x)$ is the digamma function. Using the
integral representation of the digamma function, we can rewrite
\begin{align*}
  \varphi(m-\alpha) - \varphi(1- \alpha)
  &= \Ind_{m\geq 1}\int_0^{\infty}\frac{e^{-(1-\alpha)u}}{1 - e^{-u}}\Big(1 - e^{-(m-1)u}
    \Big)\,\intd u\\
  &= \Ind_{m\geq 1}\int_0^1 \frac{1}{u(1-u)^{\alpha}}\Big(1 - (1-u)^{m-1} \Big)\,\intd u.
\end{align*}

\subsubsection{Computation of $\EE[U_t(\alpha)+V_t(\alpha)]$ and proof of
  \texorpdfstring{\cite[\cref{main-eq:170}]{caron:naulet:rousseau:main}}{equation \ref{main-eq:170}}}
\label{sec:comp-eeu_t-proof}

By the \cref{eq:90} and by the fact that $X_t(\alpha) = U_t(\alpha) + V_t(\alpha)$, it is
obvious that the \cite[\cref{main-eq:170}]{caron:naulet:rousseau:main} will follow from the computation of
$\EE[U_t(\alpha) + V_t(\alpha)]$, which we do now.

Since
$\EE[(1 - u)^{D_{t,i}^{(s)}} \mid \PoiPr] = \exp\{\sum_{j\neq i}\Ind_{\PPlab_i\leq
  t}\log(1 - u W(\PPlat_i,\PPlat_j)) \}$
   and by  a combination of the Slivnyak-Mecke's formula, Fubini's theorem, and Campbell's
formula, we obtain
\begin{align*}
  \EE[U_t(\alpha)+V_t(\alpha)]
  &= t \int_0^{\infty} \int_0^1 \left( \frac{1 - e^{- u t\mu(x)}}{u(1-u)^{\alpha}} -\frac{  e^{- u t\mu(x)} (1- e^{-  (1-u)t\mu(x)} )}{(1- u)^{1+\alpha}} \right)\,\intd u \\
\notag
   &\eqqcolon G( t, \mu)
\end{align*}
We show below that $G(t,\mu) - G(t, \bar \mu)  = O( G(t, \bar \mu) t^{-\eta})$ for some $\eta>0$ with
$ \bar \mu(x)$ defined in \cref{barmu}, and
\begin{align*}
 G(t, \bar \mu)  &= \alpha_0 c_0 t^{1+\alpha_0} \int_0^{\infty} \int_0^1 \left( \frac{1 - e^{- u y }}{u} -\frac{  e^{- uy}- e^{-y} )}{1- u} \right)(1-u)^{-\alpha} y^{-1-\alpha_0}\,\intd \\
 & =  \alpha_0 c_0 t^{1+\alpha_0}  \left( \int_0^{\infty} \int_0^1\frac{ y^{-\alpha_0} }{ (1-u)^{\alpha}}e^{-uy}  dydu\left[ \frac{1}{\alpha_0}  - \frac{1}{\alpha} \right]  + \int_0^{\infty} \frac{ (1  - e^{-y}) y^{-\alpha_0-1} }{\alpha }  dy \right)  \\
 &= \frac{  \alpha_0 c_0 t^{1+\alpha_0} \Gamma(1-\alpha_0)}{ \alpha_0}\left( \frac{ 1 }{\alpha}
 +\left[ 1 - \frac{\alpha_0}{\alpha} \right]\int_0^1u^{\alpha_0-1} (1-u)^{\alpha} du\right) \\
 &= \frac{  \alpha_0 c_0 t^{1+\alpha_0} \Gamma(1-\alpha_0)}{ \alpha_0}\left( \frac{ 1 }{\alpha}
 +\left[ 1 - \frac{\alpha_0}{\alpha} \right]\frac{ \Gamma(1 -\alpha) \Gamma(\alpha_0) }{ \Gamma(1 + \alpha_0-\alpha)} \right).
 \end{align*}
 Hence to prove the first part of \cite[\cref{main-lem:2}]{caron:naulet:rousseau:main}, it remains to show that $G(t,\mu) - G(t, \bar \mu)  = O( G(t, \bar \mu) t^{-\eta})$.

We have, writing $\Delta(x)  = \mu(x) - \bar \mu(x) $,
\begin{align*}
  &G(t,\mu) - G(t, \bar \mu)\\
  &\qquad= t \int_0^{\infty} \int_0^1 \left( \frac{e^{- u t\bar\mu(x)  } - e^{- u t\mu(x)}}{u(1-u)^{\alpha}} -\frac{ [ e^{- u t\mu(x)} -e^{- u t\bar \mu(x)}- e^{-  t\mu(x)}  +e^{-  t\bar\mu(x)}  )  }{(1- u)^{1+\alpha}} \right)\,\intd u \\
  &\qquad=  t \int_0^{\infty} \int_0^1   \left( \frac{e^{- u t\bar\mu(x)  }(1- e^{ u t\Delta(x) } ) }{u(1-u)^{\alpha}} -\frac{ [ e^{- u t\bar \mu(x)}( 1 - e^{ u t\Delta(x)} )- e^{-  t\bar\mu(x)}  (1-e^{  t\Delta(x)}  )]  }{(1- u)^{1+\alpha}} \right)\,\intd u.
 \end{align*}
 We split the above integrals into $x \leq x_0$ and $x>x_0$ , with $x_0$ defined by \cref{Deltamu}.
We have
\begin{align*}
 I_1 &= t \int_0^{x_0} \int_0^1 \left( \frac{e^{- u t\bar\mu(x)  } - e^{- u t\mu(x)}}{u(1-u)^{\alpha}} -\frac{ [ e^{- u t\mu(x)} -e^{- u t\bar \mu(x)}- e^{-  t\mu(x)}  +e^{-  t\bar\mu(x)}  )  }{(1- u)^{1+\alpha}} \right)\,\intd u \\
     & \leq x_0 \int_0^1  \frac{e^{- u t\bar\mu(x_0)  } + e^{- u t\mu(x_0)} }{u(1-u)^{\alpha}} du + t\int_0^1\frac{ \int_0^{x_0}[ \mu(x) e^{- u t\mu(x)} + \bar \mu(x) e^{- u t\bar \mu(x)} }{ (1-u)^{\alpha}} \intd x \intd u\\
 & \lesssim  e^{- a_0 t } + 1
\end{align*}
for some $a_0>0$ when $t$ is large enough and
\begin{align*}
 I_2 &= t \int_{x_0}^\infty\int_0^1 \left( \frac{ e^{- u t\bar\mu(x)  } - e^{- u t\mu(x)} }{u(1-u)^{\alpha}} -\frac{ [ e^{- u t\mu(x)} -e^{- u t\bar \mu(x)}- e^{-  t\mu(x)}  +e^{-  t\bar\mu(x)}  ]  }{(1- u)^{1+\alpha}} \right)\,\intd u \\
 &=   t \int_{x_0}^{\infty} \int_0^1 (1 - e^{ u t\Delta(x)} ) \left( \frac{ e^{- u t\bar\mu(x)  }  }{ u }  -\frac{ [ e^{- u t\bar \mu(x)} - e^{-  t\bar\mu(x)} ]  }{1- u} \right)(1-u)^{-\alpha}\,\intd u \\
&\quad  + t \int_{x_0}^{\infty} \int_0^1 (  e^{ u t\Delta(x)} -e^{  t\Delta(x)} ) \frac{  e^{-  t\bar\mu(x)}   }{ (1-u)^{\alpha+1} }\,\intd u \\
& \leq t^2 \int_{x_0}^\infty \bar \mu(x)^{\beta+ \alpha_0+1} \int_0^1\frac{ e^{- u t\bar\mu(x)  } }{ (1-u)^\alpha} du dx + t^2 \int_{x_0}^\infty \bar \mu(x)^{\beta+ \alpha_0+1}\frac{ u| e^{- u t\bar \mu(x)} - e^{-  t\bar\mu(x)} |  }{(1- u)^{\alpha+1}} \intd u\\
&\quad +  t \int_{x_0}^{\infty} \int_0^1 e^{ u t\Delta(x)} (  1 -e^{ (1-u) t\Delta(x)} ) \frac{  e^{-  t\bar\mu(x)}   }{ (1-u)^{\alpha+1} }\,\intd u\\
& \lesssim t^{ 1-\beta}  = o( t^{1+ \alpha_0 - \eta} ).
\end{align*}
So that we finally obtain that
$$ G(t,\mu) = G ( t , \bar \mu ) (1 + o(1) ) $$
and
$$ \EE[X_t(\alpha)]
    = t \mu^{-1}(t^{-1})\Gamma(1-\alpha_0)\Big(1 +
    O\big(t^{-\eta}\big) \Big)\Big\{ \frac{1}{\alpha} + \Big(1 -
    \frac{\alpha_0}{\alpha}\Big)\frac{\Gamma(1-\alpha)\Gamma(\alpha_0)}{\Gamma(1+\alpha_0
      - \alpha)}\Big\}.
$$

\subsubsection{Proof of \texorpdfstring{\cite[\cref{main-eq:232}]{caron:naulet:rousseau:main}}{equation \ref{main-eq:232}}}

By the results of
\begin{align*}
  &\PP( |-\sparsfunc'(\alpha) - \EE[-\sparsfunc'(\alpha)] | > t^{ 1 + \alpha_0 - \eta+\delta} )\\
  &\qquad\qquad\leq \PP( |X(t) - \EE[X_t(\alpha)] | > t^{ 1 + \alpha_0 - \eta+\delta}/2 ) + \PP(  [\widetilde{N}_t + N_t^{\mathrm{se}} ] > t^{ 1 + \alpha_0 - \eta+\delta}/2 )\\
  &\qquad\qquad\leq \PP( |X(t) - \EE[X_t(\alpha)] | > t^{ 1 + \alpha_0 - \eta+\delta}/2 ) + t^{ -\alpha_0/2 + \eta - \delta })\\
  &\qquad\qquad\leq \PP( |X(t) - \EE[X_t(\alpha)] | > t^{ 1 + \alpha_0 - \eta+\delta}/2 ) + 2t^{-\delta},
\end{align*}
by choosing $\eta \leq \alpha_0$. We control $\PP( |X(t) - \EE[X_t(\alpha)] | > t^{ 1 + \alpha_0 - \eta+\delta}/2 $ by bounding the variance of $U_t$ and of $V_t$.

\subsubsection{Variance of  $U_t(\alpha)$}

We split $\EE[U_t(\alpha)^2]$ into two terms:
\begin{multline*}
  \EE\Big[\sum_{i \in \Nats}\Ind_{\PPlab_i\leq t}\int_{[0,1]^2}%
  \frac{\big( 1 - (1-u)^{D_{t,i}^{(s)}} \big)\big( 1 - (1-v)^{D_{t,i}^{(s)}}
    \big)}{uv(1-u)^{\alpha}(1-v)^{\alpha}}
  \intd u \intd v \Big]\\
  +\EE\Big[\sum_{i \in \Nats}\sum_{j\ne i}\Ind_{\PPlab_i\leq t}\Ind_{\PPlab_j
    \leq t} \int_{[0,1]^2}%
  \frac{\big( 1 - (1-u)^{D_{t,i}^{(s)}} \big)\big( 1 - (1-v)^{D_{t,j}^{(s)}}
    \big)}{uv(1-u)^{\alpha}(1-v)^{\alpha}}
  \intd u \intd v \Big].
\end{multline*}
Let's call for simplicity those two terms $A_1$ and $A_2$, respectively. We
start with the more delicate one, $A_2$. For simplicity from now on we will
write $D_i = D_{t,i}^{(s)}$, and we also define $D_i^{-j} = \sum_{k\ne
  j}Z_{i,k}\Ind_{\PPlab_k\leq t}$.

\paragraph*{Bound on $A_2$}

Using that $D_i = D_i^{-j} + Z_{i,j}$, we decompose,
\begin{align*}
  (1 - (1-u)^{D_i}  )(1 - (1-v)^{D_j} )%
  &= (1 - (1-u)^{D_i^{-j}} )(1 - (1-v)^{D_j^{-i}} )\\
  &\quad + (1-v)^{D_j^{-i}}(1 - (1-v)^{Z_{i,j}})\\
  &\quad + (1-u)^{D_i^{-j}}(1 - (1-u)^{Z_{i,j}})\\
  &\quad - (1-u)^{D_i^{-j}}(1-v)^{D_j^{-i}}(1 - (1-u)^{Z_{i,j}}(1-v)^{Z_{i,j}}),
\end{align*}
which is equal to,
\begin{align*}
  &(1 - (1-u)^{D_i^{-j}} )(1 - (1-v)^{D_j^{-i}} )\\
  &+ (1-v)^{D_j^{-i}}(1-(1-u)^{D_i^{-j}})(1 - (1-v)^{Z_{i,j}})\\
  &+ (1-u)^{D_i^{-j}}(1-(1-v)^{D_j^{-i}})(1 - (1-u)^{Z_{i,j}})\\
  &+ (1-u)^{D_i^{-j}}(1-v)^{D_j^{-i}}(1 - (1-u)^{Z_{i,j}} - (1-v)^{Z_{i,j}} + (1-u)^{Z_{i,j}}(1-v)^{Z_{i,j}}).
\end{align*}
Using that $Z_{i,j} \in \Set{0,1}$, we deduce that that,
\begin{align*}
  (1 - (1-u)^{D_i}  )(1 - (1-v)^{D_j} )%
  &=(1 - (1-u)^{D_i^{-j}} )(1 - (1-v)^{D_j^{-i}} )\\
  &\quad+ (1-v)^{D_j^{-i}}(1-(1-u)^{D_i^{-j}})(1 - (1-v)^{Z_{i,j}})\\
  &\quad+ (1-u)^{D_i^{-j}}(1-(1-v)^{D_j^{-i}})(1 - (1-u)^{Z_{i,j}})\\
  &\quad+ (1-u)^{D_i^{-j}}(1-v)^{D_j^{-i}}(1 - (1-uv)^{Z_{i,j}}).
\end{align*}
Then by symmetry, $A_2 \leq A_{2,1} + 2A_{2,2} + A_{2,3}$, where,
\begin{gather*}
  A_{2,1}%
  \coloneqq%
  \EE\Big[\sum_{i \in \Nats}\sum_{j\ne i}\Ind_{\PPlab_i\leq t}\Ind_{\PPlab_j
    \leq t} \int_{[0,1]^2}%
  \frac{\big( 1 - (1-u)^{D_i^{-j}} \big)\big( 1 -
    (1-v)^{D_j^{-i}}\big)}{uv(1-u)^{\alpha}(1-v)^{\alpha}}
  \intd u \intd v \Big],\\
  A_{2,2}%
  \coloneqq%
  \EE\Big[\sum_{i \in \Nats}\sum_{j\ne i}\Ind_{\PPlab_i\leq t}\Ind_{\PPlab_j
    \leq t} \int_{[0,1]^2}%
  \frac{(1-u)^{D_i^{-j}}(1-(1-v)^{D_j^{-i}})(1-(1-u)^{Z_{i,j}})
  }{uv(1-u)^{\alpha}(1-v)^{\alpha}} \intd u \intd v \Big],\\
  A_{2,3}%
  \coloneqq%
  \EE\Big[\sum_{i \in \Nats}\sum_{j\ne i}\Ind_{\PPlab_i\leq t}\Ind_{\PPlab_j
    \leq t} \int_{[0,1]^2}%
  \frac{(1-u)^{D_i^{-j}}(1-v)^{D_j^{-i}}(1-(1-uv)^{Z_{i,j}}) }{uv(1-u)^{\alpha}(1-v)^{\alpha}} \intd u \intd v \Big]
\end{gather*}

\paragraph*{Bound on $A_{2,1}$}

Conditional on $\PoiPr$, the variables $D_i^{-j}$ and $D_j^{-i}$ are independent
as long as $i \ne j$. Then, we obtain that,
\begin{align*}
  \EE[(1 - (1-u)^{D_i^{-j}})(1-(1-v)^{D_j^{-i}})\mid\PoiPr]%
  &= \Big(1 - e^{\sum_{k\ne j}\Ind_{\PPlab_k\leq t}\log(1 -
    uW(\PPlat_i,\PPlat_k) }\Big)\\
  &\quad \times\Big(1 - e^{\sum_{k\ne i}\Ind_{\PPlab_k\leq t}\log(1 -
    uW(\PPlat_j,\PPlat_k)} \Big).
\end{align*}
Write
$\lambda_{u,x} \coloneqq \sum_k\Ind_{\PPlab_k\leq t}\log(1-uW(\PPlat_k,x))$ for
simplicity. Recall that by assumption that $W(x,x)$ by assumption. From the last
display and the Slivnyak-Mecke formula,
\begin{align*}
  A_{2,1} = t^2\int_{\NNReals^2} \int_{[0,1]^2} \frac{\EE[(1-e^{\lambda_{u,x}})(1-e^{\lambda_{v,y} })]}{u(1-u)^{\alpha} v(1-v)^{\alpha}} \,\intd u\intd
  v\intd x\intd y.
\end{align*}
We compute the expectation within the last display using Campbell's formula,
\begin{align*}
  \EE[(1-e^{\lambda_{u,x}})(1-e^{\lambda_{v,y} })]
  &=1%
    -e^{-u t\mu(x)}%
    -e^{-v t \mu(y)}%
    + e^{-u t\mu(x)}e^{-vt\mu(y)}e^{uv t\nu(x,y)}\\
  &=%
    \big(1-e^{-ut\mu(x)}\Big)\Big(1-e^{-vt\mu(y)}\Big)%
    + e^{-u t\mu(x)}e^{-vt\mu(y)}\big(e^{uv t\nu(x,y)} - 1\big).
\end{align*}
From the computations of \cref{sec:comp-eeu_t-proof}, it is easily deduced that
\begin{equation*}
  \EE[U_t(\alpha)] = t \int_{\NNReals}\int_{[0,1]} \frac{1- e^{-ut\mu(x)}}{u(1-u)^{\alpha}}\intd u \intd x.
\end{equation*}
Therefore,
\begin{align*}
  A_{2,1}%
  &\leq \EE[U_t(\alpha)]^2 + t^2 \int_{\NNReals^2}\int_{[0,1]^2}
    \frac{e^{-u t\mu(x)}e^{-vt\mu(y)}\big(e^{uvt\nu(x,y)}-1
    \big)}{u(1-u)^{\alpha}v(1-v)^{\alpha}}\intd u \intd v \intd x \intd y.
\end{align*}
Since $e^{uvt\nu(x,y)}-1 \leq uvt\nu(x,y)e^{uvt\nu(x,y)} \leq
uvt\nu(x,y)e^{\frac{u^2}{2}t\mu(x)}e^{\frac{v^2}{2}t\mu(y)}$, see for instance
\cite[Proof of Lemma~S3.7]{caron:rousseau:18}
then $e^{uvt\nu(x,y)}-1 \leq \ell_3(x)\ell_3(y) \mu(x)^a\mu(y)^a uv t
e^{\frac{u}{2}\mu(x)}e^{\frac{v}{2}\mu(y)}$ under the \cite[\cref{main-assumpt:2}]{caron:naulet:rousseau:main}, and thus
\begin{align*}
  A_{2,1}%
  &\leq%
    \EE[U_t(\alpha)^2]%
    +t^3\Big\{
    \int_{\NNReals}\int_{[0,1]}\frac{\ell_3(x)\mu(x)^ae^{-\frac{u}{2}t\mu(x)}}{(1
    -u)^{\alpha}}\,\intd u\intd x \Big\}^2
\end{align*}
But,
\begin{align*}
  \int_0^1\frac{e^{-\frac{1}{2}ut\mu(x)}\intd u}{(1-u)^{\alpha}}%
  &=\int_0^{1/2}\frac{e^{-\frac{1}{2}ut\mu(x)}\intd u}{(1-u)^{\alpha}} +
    \int_{1/2}^1\frac{e^{-\frac{1}{2}ut\mu(x)}\,\intd u}{(1-u)^{\alpha}}\\
  &\leq 2^{\alpha}\int_0^{1/2}e^{-\frac{1}{2}ut\mu(x)}\,\intd u%
    + e^{-\frac{1}{4}t\mu(x)}\int_{1/2}^1\frac{\intd u}{(1-u)^{\alpha}}\\
  &\leq 2^{\alpha}\frac{1 - e^{-\frac{1}{4}t\mu(x)}}{t\mu(x)}%
    + \frac{1}{1-\alpha}e^{-\frac{1}{4}t\mu(x)}.
\end{align*}
Therefore,
\begin{multline*}
  t^{3/2}\int_{\NNReals}\int_{[0,1]}\frac{\ell_3(x)\mu(x)^ae^{-\frac{u}{2}t\mu(x)}}{(1 -u)^{\alpha}}\,\intd u\intd
  x\\%
  \lesssim t^{3/2}\int_{\NNReals} \ell_3(x)\mu(x)^a \frac{(1-e^{-\frac{1}{4}t\mu(x)})}{t \mu(x)}\intd
  x%
  + t^{3/2}\int_{\NNReals}\ell_3(x)\mu(x)^ae^{-\frac{1}{4}t\mu(x)}\intd x.
\end{multline*}
The second integral in the rhs of the last display is $O(t^{3/2+\alpha_0-a+\delta})$ for
every $\delta > 0$ by
\cite[Lemma~S3.4]{caron:rousseau:18}%
Regarding the first integral, we observe that $(1-e^{-x})/x \leq 1$ for all
$x > 0$, so that
\begin{align*}
  \int_{\NNReals}\1{t\mu(x) \leq 1} \ell_3(x)\mu(x)^a \frac{(1-e^{-\frac{1}{4}t\mu(x)})}{t \mu(x)/4}\intd x%
  &\leq \int_{\NNReals}\1{t\mu(x)\leq 1}\ell_3(x)\mu(x)^a\intd x\\
  &\lesssim \int_{\NNReals}\ell_3(x)\mu(x)^a e^{-\frac{1}{4}t\mu(x)}\intd x\\
  &= O(t^{\alpha_0-a+\delta}),
\end{align*}
by the same argument as above, and also because $a \leq 1$ necessarily,%
\begin{align*}
  \int_{\NNReals}\1{t\mu(x) > 1} \ell_3(x)\mu(x)^a \frac{(1-e^{-\frac{1}{4}t\mu(x)})}{t \mu(x)}\intd x%
  &\lesssim \frac{1}{t^a} \int_{\NNReals}\1{t\mu(x)> 1}\ell_3(x)\frac{1 - e^{-\frac{1}{4}t\mu(x)}}{(t\mu(x))^{1-a}}\intd x\\
  &\lesssim \frac{1}{t^a} \int_{\NNReals}\ell_3(x)(1 - e^{-\frac{1}{4}t\mu(x)})\intd x\\%
  &= O(t^{\alpha_0-a+\delta}),
\end{align*}
where the last equality follows from \cite[Lemma~S3.5]{caron:rousseau:18}.
We have shown that, for any $\delta > 0$,%
\begin{equation}
  \label{eq:223}
  A_{2,1} \leq \EE[U_t(\alpha)]^2 + O(t^{3-2a+2\alpha_0+\delta}).
\end{equation}

\paragraph*{Bound on $A_{2,2}$}

Obviously $A_{2,2} \geq 0$. Further, since $(1-u)^{D_i^{-j}} \leq 1$, we have
the upper bound,
\begin{align*}
  A_{2,2}%
  &\leq  \EE\Big[\sum_{i\in \Nats}\sum_{j\ne i}\Ind_{\PPlab_i \leq t}\Ind_{\PPlab_j
    \leq t} \int_{[0,1]^2} \frac{(1-(1-v)^{D_j^{-i}})(1-(1-u)^{Z_{i,j}})}{uv(1-u)^{\alpha}(1-v)^{\alpha}} \,\intd u \intd v \Big]\\
  &= \EE\Big[\sum_{i\in \Nats}\sum_{j\ne i}\Ind_{\PPlab_i \leq
    t}\Ind_{\PPlab_j \leq t}\int_{[0,1]^2}
    \frac{\big(1 - e^{\sum_{k\ne i}\Ind_{\PPlab_k\leq t}\log(1-v
    W(\PPlat_k,\PPlat_j)) } \big)}{v(1-v)^{\alpha}} \frac{W(\PPlat_i,\PPlat_j)
    }{(1-u)^{\alpha}} \,\intd u \intd v \Big]\\
  &= \EE\Big[\sum_{i\in \Nats}\sum_{j\ne i}\Ind_{\PPlab_i \leq t}\Ind_{\PPlab_j
    \leq t}\frac{W(\PPlat_i,\PPlat_j)}{1-\alpha}%
    \int_{[0,1]}\frac{\big(1 - e^{\sum_{k\ne i}\Ind_{\PPlab_k\leq t}\log(1-v
    W(\PPlat_k,\PPlat_j)) } \big)}{v(1-v)^{\alpha}}\,\intd v
    \Big]
\end{align*}
where the second line follows by taking the conditional expectation with respect
to $\PoiPr$. Then, by the Slivnyak-Mecke formula, and then by Campbell's
formula,
\begin{align*}
  A_{2,2}%
  &\leq \frac{t^2}{1-\alpha}
    \int_{\NNReals^2} W(x,y) \int_{[0,1]}\frac{1 -
    \EE[e^{\sum_k\Ind_{\PPlab_k\leq t}\log(1-vW(\PPlat_k,y))}](1-v W(y,y))}{v(1-v)^{\alpha}}%
    \,\intd v \intd x \intd y\\
  &= \frac{t^2}{1-\alpha}
    \int_{\NNReals} \mu(y)\int_{[0,1]}\frac{1 - e^{-t v
    \mu(y)}(1-vW(y,y))}{v(1-v)^{\alpha}}\, \intd v\intd y\\
  &=\frac{t^2}{1-\alpha}
    \int_{\NNReals} \mu(y)\int_{[0,1]}\frac{1 - e^{-t v
    \mu(y)}}{v(1-v)^{\alpha}}\, \intd v\intd y%
    + \frac{t^2}{1-\alpha}
    \int_{\NNReals} \int_{[0,1]}\frac{\mu(y)W(y,y)e^{-t v
    \mu(y)}}{(1-v)^{\alpha}}\, \intd v\intd y
\end{align*}
The second term of the last display is $o(t^2)$ by dominated convergence, so it
is enough to bound the first term. For some $q \in (0,1/2)$, we first rewrite
the inner integral as
\begin{align*}
  \int_{[0,1]}\frac{1 - e^{-u t\mu(x)}}{u(1-u)^{\alpha}}\,\intd u%
  &=\int_0^q\frac{1 - e^{-u t\mu(x)}}{u(1-u)^{\alpha}}\,\intd u%
    + \int_q^1 \frac{1 - e^{-u t\mu(x)}}{u(1-u)^{\alpha}}\,\intd u\\
  &\leq \frac{1}{(1-q)^{\alpha}}\int_0^q\frac{1 - e^{-ut\mu(x)}}{u}\,\intd u%
    +\Big(1 - e^{-t\mu(x)}\Big) \int_q^1 \frac{\intd u }{u(1- u)^{\alpha}}\\
  &\leq \frac{1}{(1-q)^{\alpha}}\int_0^q t\mu(x) \,\intd u%
    + \Big(1 - e^{-t\mu(x)}\Big)\Big\{2^{\alpha}\int_q^{1/2}\frac{1}{u} + \int_{1/2}^1\frac{\intd
    u}{(1-u)^{\alpha}}\Big\}\\
  &\leq \frac{q t\mu(x)}{(1-q)^{\alpha}} + \Big(1 - e^{-t\mu(x)}\Big)\Big( 2^{\alpha}\log \frac{1}{q}%
    + \frac{1}{1-\alpha}\Big).
\end{align*}
Choosing $q = 1/t$, we obtain that as $t \to \infty$,
\begin{equation}
  \label{eq:211}
  \int_{[0,1]}\frac{1 - e^{-ut\mu(x)}}{u(1-u)^{\alpha}} \intd u%
  \leq \mu(x)(1 + o(1))%
  + 2^{\alpha}\Big(1 - e^{-t\mu(x)}\Big)(1+o(1))\log(t).
\end{equation}
Then,%
\begin{align}
  \notag
  A_{2,2}%
  &\leq%
  \frac{t^2(1+o(1))}{1-\alpha}\int_{\NNReals}\mu(x)^2\,\intd x%
  + \frac{2^{\alpha}t^2(1+o(1))}{(1-\alpha)^2}\int_{\NNReals}\mu(x)\,\intd x%
  + o(t^2)\\
  \label{eq:40}
  &=O\big(t^2\log(t)\big),
\end{align}
as both $\mu$ and $\mu^2$ are integrable by assumption.

\paragraph*{Bound on $A_{2,3}$}

Obviously $A_{2,3} \geq 0$. Furthermore, $(1 - u)^{D_i^{-j}} \leq 1$ and
$(1-v)^{D_j^{-i}} \leq 1$ too, so that we have the upper bound,
\begin{align}
  \notag
  A_{2,3}%
  &\leq \EE\Big[\sum_{i \in \Nats}\sum_{j\ne i}\Ind_{\PPlab_i\leq t}\Ind_{\PPlab_j \leq
    t}\int_{[0,1]^2} \frac{1 - (1-uv)^{Z_{i,j}}}{uv(1-u)^{\alpha}(1-v)^{\alpha}}
    \,\intd u \intd v\Big]\\
  \notag
  &= \EE\Big[\sum_{i \in \Nats}\sum_{j\ne i}\Ind_{\PPlab_i\leq t}\Ind_{\PPlab_j \leq
    t} W(\PPlat_i,\PPlat_j) \Big\{\int_{[0,1]} \frac{1}{(1-u)^{\alpha}}
    \,\intd u \Big\}^2\Big]\\
  \label{eq:209}
  &= \frac{t^2}{(1-\alpha)^2} \int_{\NNReals^2}W(x,y)\,\intd x\intd y,
\end{align}
where the second line follows by taking the conditional expectation with respect
to $\PoiPr$, and third line by the Slivnyak-Mecke formula.

\paragraph*{Bound on $A_1$}

Remark that we have,
\begin{align*}
  \EE[(1-(1-u)^{D_{t,i}^{(s)}})(1 - (1-v)^{D_{t,i}^{(s)}})  \mid \PoiPr]%
  &= 1%
    - e^{\sum_j\Ind_{\PPlab_j\leq t}\log(1-uW(\PPlat_i,\PPlat_j))}\\
  &\quad%
    - e^{\sum_j\Ind_{\PPlab_j\leq t}\log(1-vW(\PPlat_i,\PPlat_j))}\\
  &\quad%
    + e^{\sum_j\Ind_{\PPlab_j\leq t}\log(1 - (u + v -uv)W(\PPlat_i,\PPlat_j))}.
\end{align*}
Recall that $W(x,x) =0$ by assumption. Hence, by the Slivnyak-Mecke formula and
by Campbell's theorem,
\begin{align*}
  A_1%
  &=t \int_{\NNReals}\int_{[0,1]^2}\frac{F(x,u,v)}{uv(1-u)^{\alpha}(1-v)^{\alpha}}\,\intd u
    \intd v \intd x,
\end{align*}
where,
\begin{align*}
  F(x,u,v)%
  &\coloneqq%
    1 - e^{-ut\mu(x)}%
    -e^{-vt \mu(x)}%
    +e^{-(u+v-uv)t\mu(x)}\\
  &= \big(1 - e^{-ut\mu(x)} \big)\big(1 - e^{-t\mu(y)}\big)%
    +  e^{-(u+v-uv)t\mu(x)}\big(1 - e^{-uvt\mu(x)}\big).
\end{align*}
Then,
\begin{align*}
  A_1%
  &\leq%
    t \int_{\NNReals}\Big\{ \int_{[0,1]}\frac{1 - e^{-ut\mu(x)}}{u(1-u)^{\alpha}}\,\intd u \Big\}^2\, \intd x\\
  &\quad%
    + t \int_{\NNReals} \int_{[0,1]^2}\frac{e^{-(u+v-uv)t\mu(x)}(1 - e^{-uv
    t\mu(x)})}{uv(1-u)^{\alpha}(1-v)^{\alpha}}\,\intd u\intd v \intd x.
\end{align*}
Using that $e^{-(u+v-uv)t\mu(x)} \leq 1$ and
$1 - e^{-uvt(\mu(x)} \leq uv t\mu(x)$, it is immediately seen that the second
term is no more than $\frac{t^2}{(1-\alpha)^2}\int_{\NNReals}\mu(x) \,\intd
x$. So it is enough to bound the first term. We remark that the integral within
brackets has already been bounded in \cref{eq:211} and is no more than a $1+o(1)$
times $\mu(x) + 2^{\alpha}(1- e^{-t\mu(x)})\log(t)$. Then,
\begin{align}
  \label{eq:12}
  A_1
  &\lesssim t \int_{\NNReals}\mu(x)^2\,\intd x%
    + t \log^2(t) \int_{\NNReals}(1 - e^{-t\mu(x)})^2\,\intd x + O(t^2)%
  = O(t^2),
\end{align}
as $\mu^2$ is integrable by assumption, and as
$\int_{\NNReals}(1 - e^{-t\mu(x)})^2\intd x \leq \int_{\NNReals}(1 - e^{-t\mu(x)})\intd x = O(t^{\alpha})$;
see for instance \cite[Lemma~S3.5]{caron:rousseau:18}.

\paragraph*{Conclusion}

Gathering \cref{eq:223,eq:40,eq:209,eq:12}, as $t \to \infty$, for any
$\delta > 0$,
\begin{equation*}
  \EE[U_t(\alpha)^2]%
  \leq \EE[U_t(\alpha)]^2 + O\Big(t^2\log(t) \bigvee t^{3-2a+2\alpha_0+\delta} \Big).%
\end{equation*}

\subsubsection{Variance  of $V_t(\alpha)$}

Here again we write $D_i \equiv D_{t,i}^{(s)}$ and
$D_i^{-j} \equiv \sum_{k\ne j}Z_{i,k}\Ind_{\PPlab_k \leq t}$ for simplicity. We split
$\EE[V_t(\alpha)^2]$ into two terms
\begin{multline*}
  \EE\Big[\sum_{i\in \Nats}\Ind_{\PPlab_i \leq t}\Ind_{D_i \geq
    1}\int_{[0,1]^2} \frac{(1-u)^{D_i}(1-v)^{D_i}}{(1-u)^{1+\alpha}(1-v)^{1+\alpha}} \,\intd u \intd v \Big]\\
  + \EE\Big[\sum_{i \in \Nats}\sum_{j\ne i}\Ind_{\PPlab_i \leq t}\Ind_{\PPlab_j
    \leq t}\Ind_{D_i \geq 1}\Ind_{D_j \geq
    1}\int_{[0,1]^2} \frac{(1-u)^{D_i}(1-v)^{D_j}}{(1-u)^{1+\alpha}(1-v)^{1+\alpha}}  \,\intd u \intd v \Big].
\end{multline*}
Let call these two terms $A'_1$ and $A'_2$, respectively. We start by bounding
$A'_2$.

\paragraph*{Bound on $A'_2$}

Using that $Z_{i,j}\in \Set{0,1}$, we decompose
\begin{align*}
  (1-u)^{D_i}(1-v)^{D_j}\Ind_{D_i \geq 1}\Ind_{D_j\geq 1}%
  &= (1-u)^{D_i^{-j}}(1-v)^{D_j^{-i}}(1-u)^{Z_{i,j}}(1-v)^{Z_{i,j}}%
    \Ind_{D_i \geq 1}\Ind_{D_j \geq 1}\\
  &=(1-u)^{D_i^{-j}}(1-v)^{D_j^{-i}}\Ind_{D_i\geq 1}\Ind_{D_j\geq
    1}\Ind_{Z_{i,j}=0}\\%
  &\quad%
    +(1-u)^{D_i^{-j}+1}(1-v)^{D_j^{-i}+1}\Ind_{D_i\geq 1}\Ind_{D_j\geq
    1}\Ind_{Z_{i,j}=1}.
\end{align*}
Then,
\begin{equation*}
  (1-u)^{D_i}(1-v)^{D_j}\Ind_{D_i \geq 1}\Ind_{D_j\geq 1}%
  \leq%
  (1-u)^{D_i^{-j}}(1-v)^{D_j^{-i}}\Ind_{D_i^{-j}\geq 1}\Ind_{D_j^{-i}\geq
    1}%
  + (1-u)(1-v)\Ind_{Z_{i,j}=1}.
\end{equation*}
Since $D_i^{-j}$, $D_j^{-i}$ and $Z_{i,j}$ are independent conditional on
$\PoiPr$, we obtain,
\begin{multline*}
  \EE[(1-u)^{D_i}(1-v)^{D_j}\Ind_{D_i \geq 1}\Ind_{D_j\geq 1} \mid \PoiPr]\\%
  \begin{aligned}
    &\leq
    e^{\sum_{k\ne j}\Ind_{\PPlab_k\leq t}\log(1-uW(\PPlat_i,\PPlat_k))}%
    e^{\sum_{k\ne i}\Ind_{\PPlab_k\leq t}\log(1-vW(\PPlat_j,\PPlat_k))}\\%
    &\quad%
    - e^{\sum_{k\ne j}\Ind_{\PPlab_k\leq t}\log(1-uW(\PPlat_i,\PPlat_k))}%
    e^{\sum_{k\ne i}\Ind_{\PPlab_k\leq t}\log(1-W(\PPlat_j,\PPlat_k))}\\
    &\quad%
    - e^{\sum_{k\ne j}\Ind_{\PPlab_k\leq t}\log(1-W(\PPlat_i,\PPlat_k))}%
    e^{\sum_{k\ne i}\Ind_{\PPlab_k\leq t}\log(1-vW(\PPlat_j,\PPlat_k))}\\
    &\quad%
    + e^{\sum_{k\ne j}\Ind_{\PPlab_k\leq t}\log(1-W(\PPlat_i,\PPlat_k))}%
    e^{\sum_{k\ne i}\Ind_{\PPlab_k\leq t}\log(1-W(\PPlat_j,\PPlat_k))}\\%
    &\quad+ (1-u)(1-v)W(\PPlat_i,\PPlat_j).
  \end{aligned}
\end{multline*}
Recall that $W(x,x) = 0$ by assumption. Then, by Slivnyak-Mecke's formula and
by Campbell's theorem,
\begin{align*}
  A_2'%
  &\leq%
    t^2
    \int_{\NNReals^2}\int_{[0,1]^2}\frac{F(x,y,u,v)\, \intd u \intd v\intd d
    x\intd y}{(1-u)^{1+\alpha}(1-v)^{1+\alpha}}%
    + \frac{t^2}{(1-\alpha)^2}\int_{\NNReals^2}W(x,y)\intd x\intd y,
\end{align*}
where,
\begin{align*}
  F(x,y,u,v)%
  &=e^{-ut\mu(x)}e^{-vt\mu(y)}e^{uvt\nu(x,y)}%
    - e^{-ut\mu(x)}e^{-t\mu(y)}e^{ut\nu(x,y)}\\
  &\quad%
    - e^{-t\mu(x)}e^{-vt(\mu(y)}e^{vt \nu(x,y)}%
    + e^{-t\mu(x)}e^{-t\mu(y)}e^{t\nu(x,y)}.
\end{align*}
In view of the computations made in \cref{sec:comp-eeu_t-proof}, we have
\begin{equation*}
  \EE[V_t(\alpha)]%
  = - \int_{\NNReals}\int_{[0,1]}\frac{e^{-ut\mu(x)}(1-e^{-(1-u)t\mu(x)})}{(1+u)^{\alpha}}  \intd u \intd x,%
\end{equation*}
so the previous rewrites as
\begin{align}
  \label{eq:56}
  A_2'%
  \leq \EE[V_t(\alpha)]^2%
  + t^2
  \int_{\NNReals^2}\int_{[0,1]^2}\frac{\tilde{F}(x,y,u,v)\, \intd u \intd v\intd d
  x\intd y}{(1-u)^{1+\alpha}(1-v)^{1+\alpha}}%
  + O(t^2),
\end{align}
where,
\begin{align*}
  \tilde{F}(x,y,u,v)%
    &=e^{-ut\mu(x)}e^{-vt\mu(y)}(e^{uvt\nu(x,y)}-1)%
    - e^{-ut\mu(x)}e^{-t\mu(y)}(e^{ut\nu(x,y)} - 1)\\
  &\quad%
    - e^{-t\mu(x)}e^{-vt\mu(y)}(e^{vt \nu(x,y)} - 1)%
    + e^{-t\mu(x)}e^{-t\mu(y)}(e^{t\nu(x,y)} - 1).
\end{align*}
The main difficulty here is that none of term composing $\tilde{F}(u,v;x,y)$ is
integrable with respect to the measure
$\frac{\intd u \intd v }{(1-u)^{1+\alpha}(1-v)^{1+\alpha}}$, though their sum
is. We bypass the difficulty by decomposing the region of integration into four
subdomains. For some $q \in [0,1]$ to be chosen accordingly later, we let
$D_1 \coloneqq \Set{(u,v)\in [0,1]^2 \given 0 \leq u \leq q,\, 0 \leq v \leq
  q}$,
$D_2 \coloneqq \Set{(u,v) \in [0,1]^2 \given q < u \leq 1,\, 0 \leq v \leq q}$,
$D_3 \coloneqq \Set{(u,v)\in [0,1]^2 \given q < u \leq 1, q < v \leq 1}$, and
$D_4 \coloneqq \Set{(u,v) \in [0,1]^2 \given 0 \leq u \leq q,\, q < v \leq
  1}$. By symmetry, the integral of $(u,v) \mapsto \tilde{F}(u,v;x,y)$ over
$D_4$ is the same as the integral over $D_2$, and thus
\begin{equation*}
  A'_{2,1} %
  \leq  \EE[V_t(\alpha)]^2 + A'_{2,1} + 2 A'_{2,2} + A'_{2,3},
\end{equation*}
where, for $j=1,\dots,3$,
\begin{equation*}
  A'_{2,j} \coloneqq%
  t^2\int_{\NNReals^2}\int_{D_j}%
  \frac{\tilde{F}(u,v;x,y)}{(1-u)^{1+\alpha}(1-v)^{1+\alpha}}\,\intd u \intd v\intd x
  \intd y.
\end{equation*}

\paragraph*{Bound on $A'_{2,1}$}

Over $D_1$, we can bound rather quickly the integral of
$(u,v)\mapsto \tilde{F}(u,v;x,y)$ as there is no convergence issue. Indeed, it
is enough to keep the non-negative terms,%
\begin{align*}
  \tilde{F}(u,v;x,y)%
  &\leq e^{-ut\mu(x)}e^{-vt\mu(y)}\big(e^{uvt\nu(x,y)} - 1\big)%
  + e^{-t\mu(x)}e^{-t\mu(y)}\big(e^{t\nu(x,y)} -1 \big)\\
  &\leq 2t e^{-\frac{u}{2}t\mu(x)}e^{-\frac{v}{2}t\mu(y)} \ell_3(x)\ell_3(y)\mu(x)^a\mu(y)^a,
\end{align*}
where the second line follows from the same arguments as in
\cite[Lemma~S3.7]{caron:rousseau:18}. Therefore,
\begin{align}
  \notag
  A_{2,1}'%
  &\leq 2t^3\Big\{ \int_{\NNReals}\ell_3(x) \mu(x)^a \int_0^q
    \frac{e^{-\frac{u}{2}t\mu(x)}}{(1-u)^{1+\alpha}} \,\intd u \intd x
    \Big\}^2\\
  \notag
  &\leq \frac{8t}{(1-q)^{2+2\alpha}} \Big\{
    \int_{\NNReals}\ell_3(x)\mu(x)^{a-1}(1 - e^{-\frac{q}{2}t\mu(x)})\intd x
    \Big\}^2\\
  \label{eq:226}
  &= O\big(t^{3-2a+2\alpha_0 + \delta} \big).
\end{align}

\paragraph*{Bound on $A'_{2,2}$}

The challenge here is to reorganize the terms in $I(u,v;x,y)$ such that we can
obtain bounds and all the integrals still converge. We rewrite,
\begin{align*}
  \tilde{F}(u,v;x,y)%
  &=%
    e^{-vt\mu(y)}\big(e^{-ut\mu(x)} -e^{-t\mu(x)}
    \big)\big(e^{uvt\nu(x,y)} - 1\big) \\
  &\quad+ e^{-t\mu(x)}e^{-vt\mu(y)}\big(e^{uvt\nu(x,y)} - e^{vt\nu(x,y)}\Big)\\
  &\quad +e^{-t\mu(x)}e^{-t\mu(y)}\big(e^{t\nu(x,y)}- e^{ut\nu(x,y)}\big)\\
  &\quad - e^{-t\mu(y)}\big(e^{-ut \mu(x)} - e^{-t\mu(x)} \big) \big(e^{ut\nu(x,y)}
    - 1 \big).
\end{align*}
As we only need an upper bound, we keep only the non-negative terms,
\begin{align*}
  \tilde{F}(u,v;x,y)%
  &\leq%
    e^{-vt\mu(y)}\big(e^{-ut\mu(x)} -e^{-t\mu(x)}
    \big)\big(e^{uvt\nu(x,y)} - 1\big) \\
  &\quad +e^{-t\mu(x)}e^{-t\mu(y)}\big(e^{t\nu(x,y)}- e^{ut\nu(x,y)}\big)
\end{align*}
With the same arguments as usual \cite[see][Lemma~S3.7]{caron:rousseau:18}%
\begin{align*}
  \tilde{F}(u,v;x,y)%
  &\leq t^2 uv (1-u)\ell_3(x)\ell_3(y)\mu(x)^{1+a}\mu(y)^a
    e^{-\frac{v}{2}t\mu(x)}e^{- \frac{u}{2}t\mu(x)}\\
  &\quad+ t(1-u)\ell_3(x)\ell_3(y) \mu(x)^a\mu(y)^a
    e^{-\frac{1}{2}t\mu(x)}e^{-\frac{1}{2}t\mu(y)}.
\end{align*}
Then,
\begin{align*}
  A_{2,2}'%
  &\leq%
    t^4 \Big\{\int_{\NNReals} \int_q^1
    \frac{u\ell_3(x)\mu(x)^{1+a}e^{-\frac{u}{2}t\mu(x)}\,\intd u \intd x}{(1-u)^{\alpha}}
    \Big\} \Big\{\int_{\NNReals}\int_0^q \frac{v \ell_3(y) \mu(y)^a e^{-\frac{v}{2}t\mu(y)}\,\intd
    v \intd y}{(1-v)^{1+\alpha}}\Big\}\\
  &\quad+ t^3 \Big\{ \int_{\NNReals} \int_q^1
    \frac{\ell_3(x)\mu(x)^ae^{-\frac{1}{2}t\mu(x)}\intd u \intd x}{(1-u)^{\alpha}}\Big\}%
    \Big\{\int_{\NNReals}\int_0^q
    \frac{\ell_3(y)\mu(y)^ae^{-\frac{v}{2}t\mu(y)}\intd v \intd y}{(1-v)^{1+\alpha}}\Big\}.
\end{align*}
Again with the usual arguments \cite[see][Lemma~S3.4 and Lemma~S3.5]{caron:rousseau:18}
\begin{equation}
  \label{eq:228}
  A'_{2,2}%
  \leq O\big(t^{3-2a + 2\alpha_0 + \delta} \big).
\end{equation}

\paragraph*{Bound on $A'_{2,3}$}

The main challenge is to reorganize the terms in a way such that we can get
sharp upper-bounds and such that the integrals still converges.  Using that
$e^{-ut\mu(x)} = e^{-t\mu(x)} + (e^{-u t\mu(x)} - e^{-t\mu(x)})$, similarly for
$e^{-vt\mu(y)}$, we can rewrite that
\begin{align*}
  \tilde{F}(u,v;x,y)%
  &=%
    e^{-t\mu(x)}e^{-t\mu(y)}\big(e^{uv t \nu(x,y)} - 1\big)%
    +e^{-t\mu(x)}\big(e^{-vt\mu (y)}-e^{-t\mu(y)}\big)\big(e^{uvt\nu(x,y)}-1\big)\\
  &\quad%
    +\big(e^{-ut\mu(x)}-e^{-t\mu(x)}\big)\big(e^{-vt\mu(y)}-e^{-t\mu(y)}\big)\big(e^{uvt\nu(x,y)}-1\big)\\
  &\quad%
    -e^{-t\mu(x)}e^{-t\mu(y)}\big(e^{ut\nu(x,y)}-1\big)%
    -\big(e^{-ut\mu(x)}- e^{-t\mu(x)}\big)e^{-t\mu(y)}\big(e^{ut\nu(x,y)}- 1\big)\\
  &\quad%
    -e^{-t\mu(x)}e^{-t\mu(y)}\big(e^{vt\nu(x,y)}-1\big)%
    -e^{-t\mu(x)}\big(e^{-vt\mu(y)}-e^{-t\mu(y)}\big)\big(e^{vt\nu(x,y)}-1\big)\\
  &\quad%
    +\big(e^{-u t\mu(x)}-e^{-t\mu(x)}\big)e^{-t\mu(y)}\big(e^{uvt\nu(x,y)}-1\big)%
    + e^{-t\mu(x)}e^{-t\mu(y)}\big(e^{t\nu(x,y)} - 1\big).
\end{align*}
Reorganizing the previous, we find that,
\begin{align*}
  \tilde{F}(u,v;x,y)%
  &= e^{-t\mu(x)}e^{-t\mu(y)}\big(e^{uvt\nu(x,y)}  + e^{t\nu(x,y)} - e^{ut\nu(x,y)} -
    e^{-vt\nu(x,y)} \big)\\
  &\quad%
    + e^{-t\mu(x)}\big(e^{-vt\mu(y)} - e^{-t\mu(y)} \big)\big(
    e^{uvt\nu(x,y)}-e^{vt\nu(x,y)}\big)\\
  &\quad%
    + e^{-t\mu(y)}\big(e^{-ut\mu(x)} - e^{-t\mu(x)} \big)\big(e^{uvt\nu(x,y)} -
    e^{ut\nu(x,y)}\big)\\
  &\quad%
  +\big(e^{-ut\mu(x)}-e^{-t\mu(x)}\big)\big(e^{-vt\mu(y)}-e^{-t\mu(y)}\big)\big(e^{uvt\nu(x,y)}-1\big).
\end{align*}
As we are only interested in an upper bound, it is enough to keep only the
non-negative terms. That is,
\begin{align}
  \tilde{F}(u,v;x,y)%
  &\leq%
    \label{eq:78a}
    e^{-t\mu(x)}e^{-t\mu(y)}\big(e^{uvt\nu(x,y)}  + e^{t\nu(x,y)} - e^{ut\nu(x,y)} -
    e^{-vt\nu(x,y)} \big)\\
  \label{eq:78b}
  &\quad%
    +\big(e^{-ut\mu(x)}-e^{-t\mu(x)}\big)\big(e^{-vt\mu(y)}-e^{-t\mu(y)}\big)\big(e^{uvt\nu(x,y)}-1\big).
\end{align}
We now bound each of the terms in the last display. For fixed $(x,y)$, let
define $\phi(u,v) \coloneqq e^{-t(1-uv)\nu(x,y)}$, so that the term
\cref{eq:78a} can be rewritten as
$e^{-t\mu(x)}e^{-t\mu(y)}e^{t\nu(x,y)}(\phi(1,1) + \phi(u,v) - \phi(u,1) -
\phi(1,v))$. By a Taylor expansion of $\phi$, for any $(u,v) \in D_3$,
\begin{equation*}
  \phi(1,1)  - \phi(u,1) + \phi(u,v) - \phi(1,v)%
  \leq (1 - u)(1-v)\sup_{(\bar{u},\bar{v}) \in D_3}
  \partial_u\partial_v\phi(\bar{u},\bar{v}).
\end{equation*}
It is clear that,
\begin{equation*}
  \partial_u\partial_v \phi(u,v)%
  = t\nu(x,y) \phi(u,v) + t^2 uv \nu(x,y)^2 \phi(u,v)
\end{equation*}
Therefore, $\phi(1,1) - \phi(u,1) + \phi(u,v) - \phi(1,v)\leq
(1-u)(1-v)(t\nu(x,y) + t^2\nu(x,y)^2)$, at least when $(u,v) \in D_3$. By the
usual arguments, it is rapidly seen that the term in \cref{eq:78b} is bounded by
$t^3uv(1-u)(1-v)\mu(x)\mu(y)\nu(x,y)e^{-ut\mu(x)}e^{-vt\mu(y)}e^{uvt\nu(x,y)}$,
and then when $(u,v)\in D_3$,
\begin{align*}
  \tilde{F}(u,v;x,y)%
  &\leq%
    t(1-u)(1-v)\ell_3(x)\ell_3(y)\mu(x)^a\mu(y)^ae^{-\frac{1}{2}t\mu(x)}e^{-\frac{1}{2} t\mu(y)}\\%
  &\quad%
    +t^2(1-u)(1-v)\ell_3(x)^2\ell_3(y)^2\mu(x)^{2a}\mu(y)^{2a}e^{-\frac{1}{2}t\mu(x)}e^{-\frac{1}{2}
    t\mu(y)}\\%
  &\quad%
    +
    t^3uv(1-u)(1-v)\ell_3(x)\ell_3(y)\mu(x)^{1+a}\mu(y)^{1+a}e^{-\frac{u}{2}t\mu(x)}e^{-\frac{v}{2}t\mu(y)}.
\end{align*}
Therefore,
\begin{align*}
  A_{2,3}'%
  &\leq%
    t^3 \Big\{ \int_{\NNReals}\int_q^1
    \frac{\ell_3(x)\mu(x)^ae^{-\frac{1}{2}t\mu(x)} \intd u \intd
    x}{(1-u)^{\alpha}}  \Big\}^2\\
  &\quad%
    + t^4 \Big\{ \int_{\NNReals}\int_q^1
    \frac{\ell_3(x)^2\mu(x)^{2a}e^{-\frac{1}{2}t\mu(x)}\intd u \intd
    x}{(1-u)^{\alpha}} \Big\}^2\\
  &\quad%
    + t^5 \Big\{ \int_{\NNReals}\int_q^1
    \frac{\ell_3(x)\mu(x)^{1+a}e^{-\frac{u}{2}t\mu(x)}\intd u \intd x}{(1-u)^{\alpha}} \Big\}^2.
\end{align*}
Thus, by \cite[Lemma~S3.4 and Lemma~S3.5]{caron:rousseau:18} for any $\delta > 0$,
\begin{equation}
  \label{eq:229}
  A_{2,3}'%
  \leq O\big(t^{3-2a + 2\alpha_0 + \delta}  \big).
\end{equation}

\paragraph*{Bound on $A'_1$}

We first remark that, for any $D_i \geq 1$,
\begin{align*}
  \int_{[0,1]^2}\frac{(1-u)^{D_i}(1-v)^{D_i}}{(1-u)^{1+\alpha}}(1-v)^{1-\alpha} %
  = \Big\{\int_{[0,1]} \frac{(1-u)^{D_i}}{(1-u)^{1+\alpha}} \,\intd u\Big\}^2
  = \frac{1}{(D_i - \alpha)^2}%
  \leq \frac{1}{(1 - \alpha)^2}.
\end{align*}
Therefore,
\begin{align}
  \label{eq:147}
  A'_1%
  &\leq \frac{1}{(1-\alpha)^2}\EE\Big[\sum_{i\in \Nats}\Ind_{\PPlab_i \leq t}\Ind_{D_i\geq 1}\Big]
  = \frac{\EE[\nver]}{(1-\alpha)^2}%
  = O\big(t^{1+\alpha_0}\ell_1(t) \big),
\end{align}
where the last estimate follows by \cite[\cref{main-th:asympmultigraphex}]{caron:naulet:rousseau:main}.

\paragraph*{Conclusion}

Gathering the \cref{eq:56,eq:226,eq:228,eq:229,eq:147}, we have as $t
\to \infty$, for any $\delta > 0$
\begin{equation*}
  \EE[V_t(\alpha)^2]%
  \leq \EE[V_t(\alpha)]^2 + O\Big(t^2 \bigvee t^{3-2a+2\alpha_0 + \delta} \Big).
\end{equation*}

\section{Proofs of the examples}
\label{sec:proofs-examples}

\subsection{Proof of \texorpdfstring{\cite[\cref{main-lem:configuration-models}]{caron:naulet:rousseau:main}}{Lemma \ref{lem:z9f-xsj8-bpx}}}

Let define on $(0,1)$ the function
$F^{*}(\alpha) \coloneqq \sum_{j\geq 2}f_j \sum_{k=1}^{j-1}\frac{\alpha}{k-\alpha} - 1$;
so that $\alpha_0$ is a solution of $F^{*}(\alpha)= 0$. In the conditions of the
lemma, it is the case that $f_1 < 1$. Then, the function $\ell$ is monotonic
with $\lim_{\alpha \to 0}F^{*}(\alpha) = -1$ and
$\lim_{\alpha\to 1}F^{*}(\alpha) = \infty$. This establishes existence and
uniqueness of $\alpha_0$.

We now prove that the second part of condition~\eqref{main-eq:assumpt1_cond1} of
Assumption~\ref{main-ass:concent:degdist} is satisfied (in probability). We define for simplicity
$F_t(\alpha) = \frac{-\alpha \mathcal{C}_t'(\alpha)}{t} - 1$. Then,
\begin{align*}
  F_t(\hat{\alpha}_t) - F_t(\alpha_0)%
  &= \sum_{j\geq 2} \frac{N_{t,j}}{t} \sum_{k=1}^{j-1}\Bigg( \frac{\hat{\alpha}_t}{k - \hat{\alpha}_t} - \frac{\alpha_0}{k-\alpha_0} \Bigg)\\
  &= (\hat{\alpha}_t - \alpha_0)\sum_{j\geq 2}\frac{N_{t,j}}{t}\sum_{k=1}^{j-1} \frac{k}{(k-\alpha_0)(k - \hat{\alpha}_t)}\\
  &\geq (\hat{\alpha}_t - \alpha_0)\sum_{j\geq 2}\frac{N_{t,j}}{t}\sum_{k=1}^{j-1} \frac{k}{(k-\alpha_0)^2}.
\end{align*}
Similarly,
\begin{align*}
  F_t(\hat{\alpha}_t) - F_t(\alpha_0)%
  &= (\hat{\alpha}_t - \alpha_0) \sum_{j\geq 2} \frac{N_{t,j}}{t} \sum_{k=1}^{j-1} \frac{k}{(k-\alpha_0)^2} \frac{1}{1 - \frac{\hat{\alpha}_t - \alpha_0}{k-\alpha_0}}\\
  &\leq \frac{\hat{\alpha}_t - \alpha_0}{1 - \frac{\hat{\alpha}_t - \alpha_0}{1-\alpha_0}}
    \sum_{j\geq 2}\frac{N_{t,j}}{t} \sum_{k=1}^{j-1} \frac{k}{(k-\alpha_0)^2},
\end{align*}
Since $F_t(\hat{\alpha}_t) = 0$ by definition of $\hat{\alpha}_t$, it follows
\begin{align*}
  \hat{\alpha}_t - \alpha_0%
  \leq%
  \underbrace{\frac{- F_t(\alpha_0)}{\sum_{j\geq 2}\frac{N_{t,j}}{t}\sum_{k=1}^{j-1} \frac{k}{(k-\alpha_0)^2}}}_{\eqqcolon Z_t}%
  \leq \frac{\hat{\alpha}_t - \alpha_0}{1 - \frac{\hat{\alpha}_t - \alpha_0}{1-\alpha_0}}
\end{align*}
or in other words,
\begin{align}
  \label{eq:cf7}
  \frac{Z_t}{1 + \frac{Z_t}{1-\alpha_0}}%
  \leq \hat{\alpha}_t - \alpha_0 \leq Z_t.
\end{align}
Hence to prove that the second part of condition~\eqref{main-eq:assumpt1_cond1}
of Assumption~\ref{main-ass:concent:degdist} is satisfied (in probability) it is
enough to show that $\log(t)|Z_t| \to 0$ (in probability). First,
\begin{equation}
  \label{eq:cf39}
  |Z_t|%
  \leq%
  \frac{(1-\alpha_0)^2|F_t(\alpha_0)|}{\sum_{j\geq 2} \frac{N_{t,j}}{t}}%
  = \frac{(1-\alpha_0)^2|F_t(\alpha_0)|}{1 - N_{t,1}/t}.
\end{equation}
Second, Recall that
$N_{t,j} = \sum_{i=1}^t\Ind_{\tilde{D}_{t,i}=j}$. Hence,
\begin{align*}
  |F_t(\alpha_0)|%
  &= \Bigg| \frac{1}{t}\sum_{i=1}^{t} \sum_{j\geq 2} \Ind_{\tilde{D}_{t,i} = j}\sum_{k=1}^{j-1} \frac{\alpha_0}{k-\alpha_0} - 1 \Bigg|\\
  &\leq \Bigg| \frac{1}{t}\sum_{i=1}^{t} \sum_{j\geq 2} \Ind_{D_{t,i} = j}\sum_{k=1}^{j-1} \frac{\alpha_0}{k-\alpha_0} - 1 \Bigg|%
    + \frac{1}{t}\sum_{j\geq 2}|\Ind_{\tilde{D}_{t,t} = j} - \Ind_{D_{t,t} = j}|\sum_{k=1}^{j-1}\frac{\alpha_0}{k-\alpha_0}\\
  &\leq \Bigg| \frac{1}{t}\sum_{i=1}^{t} \sum_{j\geq 2} \Ind_{D_{t,i} = j}\sum_{k=1}^{j-1} \frac{\alpha_0}{k-\alpha_0} - 1 \Bigg|%
    + \frac{1}{t}\sum_{k= 1}^{D_{\max,t}-1}\frac{\alpha_0}{k-\alpha_0}.
\end{align*}
[Note that the previous estimate will cost a $\log(D_{\max,t})/t$ term that
can be improved to $1/t$ at the price of longer computations, which is not worth
for our purpose]. But, by the fact that $F^{*}(\alpha_0) = 0$,
\begin{align*}
  \EE\Bigg[\frac{1}{t}\sum_{i=1}^{t} \sum_{j\geq 2} \Ind_{D_{t,i} = j}\sum_{k=1}^{j-1} \frac{\alpha_0}{k-\alpha_0} \Bigg]%
  &= \frac{1}{\sum_{k=1}^{D_{\max,t}}f_k} \sum_{j\geq 2} f_j \sum_{k=1}^{j-1} \frac{\alpha_0}{k - \alpha_0}\\
  &= \frac{1}{1 - \sum_{k>D_{\max,t}}f_k}\\
  &= 1 + \frac{\sum_{k>D_{\max,t}}f_k}{1 - \sum_{k>D_{\max,t}}f_k}
\end{align*}
Also,
\begin{align*}
  \var\Bigg(\frac{1}{t}\sum_{i=1}^{t} \sum_{j\geq 2} \Ind_{D_{t,i} = j}\sum_{k=1}^{j-1} \frac{\alpha_0}{k-\alpha_0} \Bigg)%
  &\leq \frac{1}{t}\sum_{j\geq 2} \frac{f_j}{1 - \sum_{k>D_{\max,t}}f_k} \Bigg(\sum_{k=1}^{j-1} \frac{\alpha_0}{k-\alpha_0} \Bigg)^2.
\end{align*}
It follows from these estimates that and the fact that $D_{\max,t}$ cannot
exceed some power of $t$
\begin{align}
  \label{eq:cf38}
  |F_t(\alpha_0)|%
  = O_p\Big( D_{\max,t}^{-\alpha_1} \bigvee \frac{\log(t)^2}{t} \Big)%
  = o\Big(\frac{1}{\log(t)} \Big).
\end{align}
With a similar reasoning, it is easily seen that
$1 - \frac{N_{t,1}}{t} = 1 - f_1 + o_p(1)$, so that by combining the
equations~\eqref{eq:cf7},~\eqref{eq:cf39}, and~\eqref{eq:cf38} we obtain that
$\log(t)|\hat{\alpha}_t - \alpha_0| = o_p(1)$.

We now prove that the first part of condition~\eqref{main-eq:assumpt1_cond1} of
Assumption~\ref{main-ass:concent:degdist} is satisfied (in probability). We see that
$D_t^{\star} = \sum_{i=1}^t\tilde{D}_{t,i} = \sum_{i=1}^tD_{t,i} + O(1)$
almost-surely, and
\begin{equation*}
  \EE\Big[\sum_{i=1}^tD_{t,i}\Big]%
  = t \frac{\sum_{j\geq 1}j  p_j\Ind_{j\leq D_{\max,t}}}{\sum_{j\geq 1}p_j\Ind_{j\leq D_{\max,t}}}%
  \sim \frac{\alpha_1L}{1-\alpha_1} \cdot t D_{\max,t}^{1-\alpha_1}%
  \eqqcolon \bar{D}_t^{\star}
\end{equation*}
as $t \to \infty$ by Lemma~\ref{lem:cf2}. Also, by the same Lemma, as $t \to \infty$
\begin{align*}
  \var\Big(\sum_{i=1}^tD_{t,i} \Big)%
  &\leq t \frac{\sum_{j\geq 1}j^2 p_j\Ind_{j\leq D_{\max,t}}}{\sum_{j\geq 1}p_j\Ind_{j\leq D_{\max,t}}}%
    \sim \frac{\alpha_1 L}{2-\alpha_1} \cdot t D_{\max,t}^{2-\alpha_1}.
\end{align*}
Therefore,
\begin{equation*}
  \frac{D_t^{\star} - \bar{D}_t^{\star}}{\bar{D}_t^{\star}}%
  = O_p\Bigg( \sqrt{\frac{D_{\max,t}^{\alpha_1}}{t}} \Bigg).
\end{equation*}
It follows that if
$D_{\max,t} \sim A\cdot t^{\frac{1-\alpha_0}{(1+\alpha_0)(1-\alpha_1)}}$ for
some constant $A > 0$, we have
\begin{align*}
  D_t^{\star} = \frac{\alpha_1 L A^{1-\alpha_1}}{1-\alpha_1} t^{\frac{2}{1+\alpha_0}}\Big(1 + o_p(1)\Big),
\end{align*}
and,
\begin{align*}
  \sqrt{2 D_t^{\star}}\Big(\frac{\alpha_0 N_t}{D_t^{\star}}\Big)^{1-\alpha_0}%
  &= \sqrt{2}\alpha_0^{1-\alpha_0} \Big( \frac{\alpha_1 L A^{1-\alpha_1}}{1-\alpha_1} \Big)^{\frac{1}{2}-(1-\alpha_0)}\Big(1 + o_p(1)\Big).
\end{align*}
This concludes the proof.

\begin{lemma}
  \label{lem:cf2}
  Let $(f_j)_{j\geq 1}$ be a a probability mass function on $\{1,2,,\ldots\}$
  such that
$1 - \sum_{k=1}^jf_k \sim Lj^{-\alpha_1}$ as $j \to \infty$, for some
$\alpha_1 \in (0,1)$ and some $L > 0$. Then, as $D \to \infty$
\begin{equation*}
  \sum_{j=1}^D j f_j \sim \frac{\alpha_1 L D^{1-\alpha_1}}{1-\alpha_1},\qquad%
  \sum_{j=1}^D j^2f_j \sim \frac{\alpha_1 L D^{2-\alpha_1}}{2-\alpha_1}.
\end{equation*}
\end{lemma}
\begin{proof}
  These are famous results about regular variations \cite{Bingham1987}, we
  briefly sketch a proof for completeness. Let $F_j \coloneqq \sum_{k=1}^jf_k$. By standard
  computations it is seen that $\sum_{j=1}^D(1 - F_j) = \sum_{j=1}^D j f_j + D(1- F_D)$.
  By assumption, for all $\varepsilon > 0$ we can find $K > 0$ such that
  $1 - \varepsilon \leq \frac{1- F_j}{Lj^{-\alpha_1}} \leq 1 + \varepsilon$ for all
  $j \geq K$. It follows, since $0 \leq 1 -F_j \leq 1$ for all $j\geq 1$, that when
  $D \gg K$,
  \begin{align*}
    \sum_{j=1}^D(1 - F_j)%
    \leq \sum_{j=1}^{K-1}(1 - F_j) + \sum_{k=K}^D \frac{1 - F_j}{Lj^{-\alpha_1}}\cdot Lj^{-\alpha_1}%
    \leq K + (1+\varepsilon) \sum_{k=K}^D Lj^{-\alpha_1}
  \end{align*}
  and,
  \begin{align*}
    \label{eq:36}
    \sum_{j=1}^D(1 - F_j)%
    \geq \sum_{k=K}^D \frac{1 - F_j}{Lj^{-\alpha_1}}\cdot Lj^{-\alpha_1}%
    \geq (1 - \varepsilon) \sum_{k=K}^D Lj^{-\alpha_1}.
  \end{align*}
  But for any $K > 0$, it can be shown that
  $\sum_{k=K}^Dj^{-\alpha_1} \sim \frac{D^{1-\alpha_1}}{1-\alpha_1}$, and since
  the last two estimates are true for any $\varepsilon > 0$, we deduce the first
  result. The other result is proved similarly.
\end{proof}

\subsection{Proof of \texorpdfstring{\cite[\cref{main-lem:z9f-xsj8-bpx}]{caron:naulet:rousseau:main}}{Lemma \ref{lem:z9f-xsj8-bpx}}}
\label{sec:yyv-dh1c-y6a}

We have $\EE[\nedg] = N^2(\bar \theta)^2p_N$, so that
\begin{equation*}
  \EE[\nedg]^{-(1+\alpha_0)/(2 - 2\alpha_0)}N^{1/(1-\alpha_0)} = c_0^{-(1+\alpha_0)/(2 - 2\alpha_0)} \eqqcolon \tau_*.
\end{equation*}
Furthermore
$$ \sum_{j=2}^N \frac{ N_{t,j} }{N}\sum_{k=1}^{j-1} \frac{ \alpha_0 }{ k - \alpha_0} =  1+\left( \sum_{j=2}^N \left[ \frac{ N_{t,j} }{ N} -p(j)\right]  \sum_{k=1}^{j-1} \frac{ \alpha_0 }{ k - \alpha_0}\right).$$
Writing $$ \text{sgn}(D_i - \mu_i)= \left\{ \begin{array}{cc}
 1 & \text{ if  } \, D_i > \mu_i \\
  0 & \text{ if  } \, D_i = \mu_i \\
   -1 & \text{ if  } \, D_i < \mu_i \end{array}
  \right. $$
we have
\begin{align*}
 \left| \sum_{j=2}^N \left[ \frac{ N_{t,j} }{ N} -p(j)\right]  \sum_{k=1}^{j-1} \frac{ \alpha_0 }{ k - \alpha_0} \right|  &=\left| \frac{ 1  }{ N}\sum_i \text{sgn}(D_i - \mu_i) \sum_{l=\mu_i\wedge D_i}^{\mu_i \vee D_i-1} \frac{ \alpha_0 }{ l - \alpha_0 }\right| , \quad \mu_i = \lceil\theta_i Np_N \bar \theta \rceil  \\
 & \leq \left| \frac{ 1  }{ N}\sum_i (D_i - \mu_i)   \frac{ \alpha_0 }{\mu_i - \alpha_0 }\right| + \left| \frac{ 1  }{ N}\sum_i (D_i - \mu_i)^2   \frac{ \alpha_0 }{(\mu_i\wedge D_i - \alpha_0)^2 }\right| .
 \end{align*}
Note that
\begin{align*}
\frac{ 1  }{ N}\sum_i (D_i - \mu_i)^2   \frac{ \alpha_0 }{(\mu_i\wedge D_i - \alpha_0)^2 } &\leq  \frac{ \alpha_0  }{ N}\sum_i   \frac{  (D_i - \mu_i)^2 }{(\mu_i - \alpha_0)^2 } +
\frac{ \alpha_0  }{ N}\sum_i   \frac{  (D_i - \mu_i)^4+| D_i - \mu_i|^3}{(\mu_i - \alpha_0)^2 }
\end{align*}
and Using Theorem 2 in \cite{ouadah20},
$$ \sum_i E  \frac{  (D_i - \mu_i)^2 }{(\mu_i - \alpha_0)^2 } \lesssim  p_N \sum_i \frac{ \theta_i N\bar \theta p_N }{ \lceil \theta_i N\bar \theta p_N \rceil } \lesssim Np_N$$
and similar computations lead to
$$ \sum_i E  \frac{  (D_i - \mu_i)^4 }{(\mu_i - \alpha_0)^2 } \lesssim Np_N$$
so that
\begin{align*}
\frac{ 1  }{ N}\sum_i (D_i - \mu_i)^2   \frac{ \alpha_0 }{(\mu_i\wedge D_i - \alpha_0)^2 } =O_p( p_N),\ \mathrm{and},\ %
\frac{ 1  }{ N}\sum_i (D_i - \mu_i)   \frac{ \alpha_0 }{ \mu_i - \alpha_0 } = O_p(\sqrt{p_N}  )
\end{align*}
which in turns implies that
$$ |\hat \alpha_t - \alpha_0 | = O_p(\sqrt{p_N})$$
and \cite[\cref{main-ass:concent:degdist}]{caron:naulet:rousseau:main} is satisfied.

  \putbib%
  \stopcontents[supp]%
\end{bibunit}


\end{document}